\documentclass[12pt]{article}

\usepackage{graphicx}
\usepackage{caption}
\usepackage{subcaption}
\usepackage{amsmath}
\usepackage{amssymb}
\usepackage{amsbsy}
\usepackage{color}
\usepackage{amsthm}
\usepackage{wasysym}
\usepackage{fullpage}
\DeclareGraphicsExtensions{.pdf,.jpg,.png,.tif}

\def\reals{{{\rm l} \kern-.15em {\rm R}}}
\definecolor{purpleheart}{rgb}{0.41, 0.21, 0.6}
\definecolor{dgreen}{rgb}{0.0, 0.5, 0.0}
\newcommand\tb{}
\newcommand\tr{}
\newcommand\tp{}
\newcommand\tg{}

\newcommand\bx{\mathbf{x}}
\newcommand\by{\mathbf{y}}
\newcommand\bc{\mathbf{c}}
\newcommand\bk{\mathbf{k}}

\newcommand\normk{\| \bk \|}
\newcommand\ei{\mathbf{e}_i}

\newcommand\be{\begin{equation}}
\newcommand\ee{\end{equation}}
\newtheorem{theorem}{Theorem}[section]
\newtheorem{lemma}[theorem]{Lemma}

\newcommand\slayer{\mathcal{S}}
\newcommand\dlayer{\mathcal{D}}

\newcommand\dlayloc{\dlayer_L}
\newcommand\dlayloch{\dlayer_L^{h}}

\newcommand{\zloc}{z^{L}}
\newcommand{\zloch}{z^{L,h}}
\newcommand{\saelem}{W}
\newcommand{\nptsx}{n_{\bx_i}}
\newcommand{\setofpanels}{{\Gamma^{loc}(\bx)}}

\newcommand{\hpanel}{\bar{h}}
\newcommand{\dup}{d_{up}}
\newcommand{\dqbx}{d_{QBX}}
\newcommand\floor[1]{\lfloor#1\rfloor}
\newcommand\ceil[1]{\lceil#1\rceil}

\usepackage{authblk}

\date{}
\title{A local target-specific QBX method for Laplace's equation in 3D
  multiply-connected domains. }
\title{A local target specific quadrature by expansion method for evaluation
of layer potentials in 3D}

\author{Michael Siegel$^1$ and Anna-Karin Tornberg$^2$\\$^1$ Department of Mathematical Sciences, New Jersey Institute of Technology, Newark, NJ 07102 (misieg@njit.edu)\\ $^2$ KTH Mathematics, Linn\'{e} Flow Centre/Swedish e-Science Research Center, 100 44 Stockholm, Sweden (akto@kth.se)}

\begin{document}

\maketitle

\begin{abstract}
Accurate evaluation of layer potentials is crucial when boundary
integral equation methods are used to solve partial differential
equations. Quadrature by expansion (QBX) is a recently introduced
method that can offer high accuracy  for singular and nearly
singular integrals, using truncated expansions to locally
represent the potential. The QBX method is typically based on a spherical harmonics expansion which when truncated at order $p$ has $O(p^2)$ terms. This expansion can equivalently be written with $p$ terms,
however paying the price that the expansion coefficients will depend
on the evaluation/target point.
Based on this observation, we develop a target specific QBX method,
and apply it to  Laplace's equation on multiply-connected
domains. The method is local in that the QBX expansions only involve
information from a neighborhood of the target point. 
An analysis of the truncation error in the QBX expansions is presented,
practical parameter choices are discussed and the method is
validated and tested on various problems.
\end{abstract}


%Then the error incurred when truncating the expansion at $\normk=p$ is
%\[
%e_T=\sum_{n=p}^{\infty} B_n(\bx_c,\bx,\by)
%\]

%\section{Notation}
%Save D for double layer potential? Then need something else than D for
%the open regions... Take \dlayer for double layer. 

\section{Introduction}

Numerical methods based on boundary integral equations have the
advantage that only the boundaries of the domain must be discretized,
which both simplifies the handling of the geometry and reduces
the number of discretization points. 
The resulting linear system after
discretization is however dense, and the evaluation of layer potentials requires accurate quadrature methods for singular and nearly singular integrals.  \tb{ Nearly singular integrals arise when evaluating solutions close to boundaries during a post-processing step, after the integral equation has already been solved. They also arise in problems involving multiply-connected domains, when the integral equation is to be solved and separate boundary components are nearly touching. Such problems are important in applications including, for example, electromagnetic scattering in media with multiple inclusions and particle Stokes flow.}

With a discretization based on a second kind integral equation, the
resulting matrix is well conditioned, with a condition number
independent on the fineness of the discretization.  The number of
iterations in an iterative method such as GMRES hence stays constant
as the discretization of the boundaries is refined, yielding a total
cost of $O(N^2)$ to solve the system, where $N$ is the number of
unknowns.  The $O(N^2)$ comes from the cost of the matrix-vector
multiply for a full matrix, and can be reduced to $O(N)$ or
$O(N \log N)$ using a fast method, such as the fast multipole method
(FMM) \cite{Greengard1987}, or an FFT based method such as a fast
Ewald method (commonly for periodic problems
%%\cite{Klinteberg2014,Lindbo2010,Lindbo2011,Lindbo2012},  
\cite{Lindbo2011},  recently also for non-periodic ones \cite{KlintebergFree}).

Regarding efficient and accurate quadrature methods for the evaluation of singular and nearly singular integrals,  excellent methods that utilize a complex variable formulation are available in two dimensions   \cite{Barnett2015,Helsing2008,Ojala2015}. 
Considering arbitrary geometries in three dimensions, this remains a
topic of current research where several methods have been introduced and
contributed different advances
\cite{Beale2004,Bremer2012,Bremer2013,Tlupova2013,Ying2006,Zhao2010}.

Quadrature by expansion (QBX) is a rather recent method \cite{Barnett2014,
  Klockner} for the numerical evaluation of singular and nearly
singular integrals. It was introduced for the Helmholtz kernel in two
dimensions, but the central principle of the method can be generalized to other
kernels in both two and three dimensions. Noting that the layer
potential is smooth away from the boundary, it can locally be
represented using an expansion centered about a point or {\it expansion
center} which is located just off the surface. Once the coefficients of this expansion have been computed,
the potential can be evaluated at a target point closer to the
surface, or even on the boundary \cite{Epstein2013} using this local
expansion \tb{(see Figure \ref{fig_qbx})}. 
Such expansions are used also in the FMM, and it is hence attractive
to integrate the QBX method into an FMM. In \cite{Rachh2016}, a
first such step is taken in two dimensions. There, the QBX method is
``global'', meaning that all information from all boundaries will
enter each QBX local expansion before evaluation. Localizing the QBX
treatment by using only information from boundaries that are near the expansion center would reduce the cost, but introduces other algorithmical
challenges, even more so in three dimensions.

\begin{figure}
\vspace{0.in}
  \center{\includegraphics[clip=true, viewport=70 100 510 345, scale=.75]{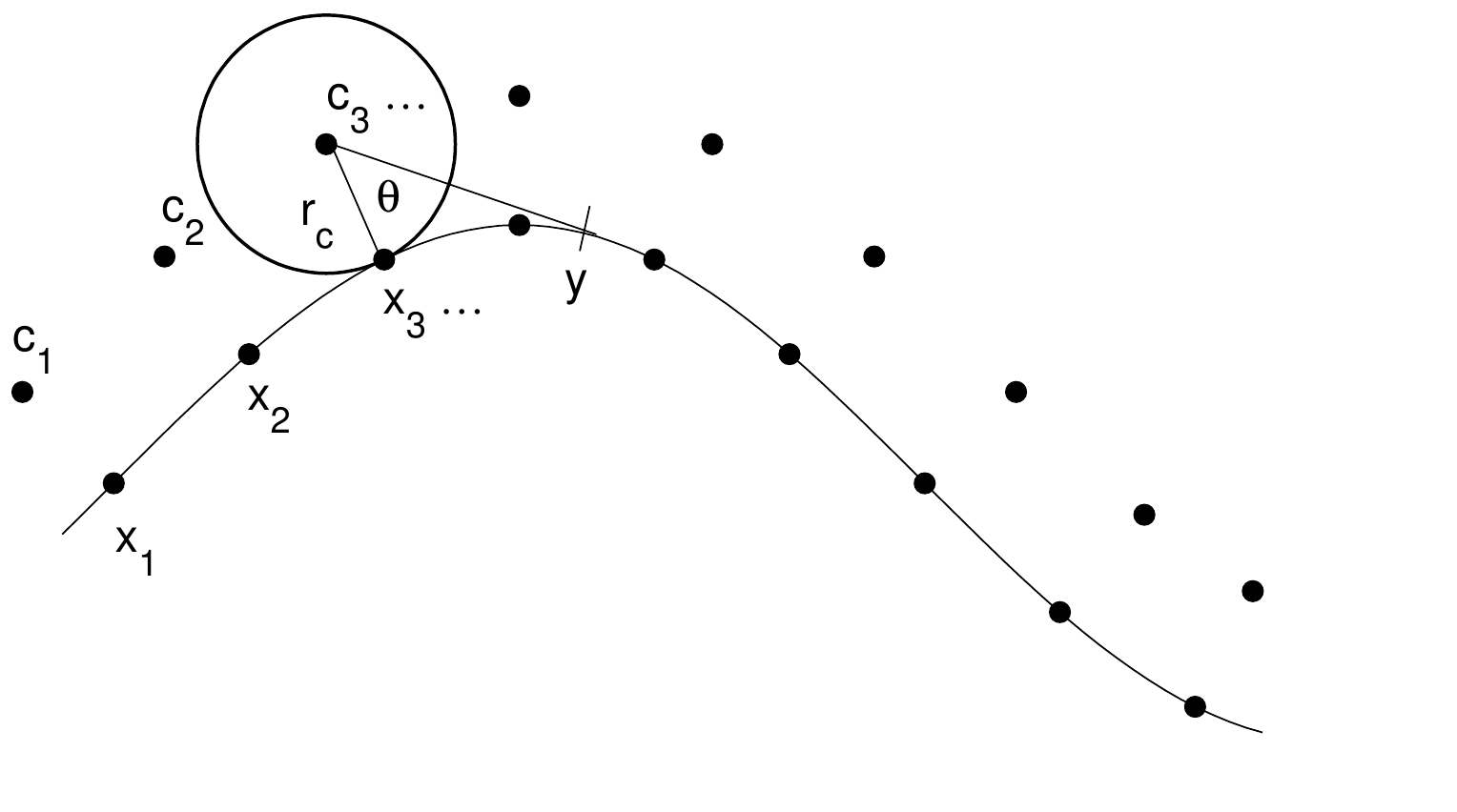}}   
  \caption{\tb{QBX  with expansion centers $\bc_1, \bc_2, \ldots$ and targets $\bx_1, \bx_2, \ldots$. The domain of convergence for the expansion has radius $r_c$, and  $\theta$ is defined in (\ref{eqn:Pn_expand}).  }} 
  \label{fig_qbx} 
\end{figure}

In \cite{Klinteberg2016}, a QBX method was presented for spheroidal
particles in three-dimensional Stokes flow. Whenever an evaluation
point is on or close to a spheroidal surface, a QBX expansion is used
to evaluate the layer potential over that surface. Contributions from
different surfaces are kept separate, even if the surfaces are
close. Once QBX centers have been chosen relative to the surface,
precomputations can be made to strongly accelerate the computation of
the QBX coefficients.  Using the axisymmetry of the body, the
storage need can be greatly reduced, and the same precomputed values
can be used for all spheroids of the same shape. 
This method is combined with an
FFT based Ewald summation method - the Spectral Ewald method
\cite{Klinteberg2014,Lindbo2010}, and yields an $O(N \log N)$ method (with $N$ the total number of
gridpoints) as the number of spheroids is increased while the
resolution on each spheroid is kept fixed.

The error in the QBX method for evaluating a layer potential on or
close to the boundary has two main sources: the error from truncation of
the local series expansion and the computation of the expansion
coefficients.
The truncation error was analyzed by Epstein et al.\ \cite{Epstein2013}
for the Laplace and Helmholtz kernels in two and three dimensions,
assuming a ``global'' QBX approach. The error due to computations of
the expansion coefficients is again a quadrature error in evaluating
the integrals defining these coefficients. This was analyzed in 
\cite{Klinteberg2016b}
using a method based on contour integration and calculus of residues in two dimensions and for some three dimensional cases. Once a procedure for taking the panel shape into account is introduced, the error estimates in 2D are remarkably precise, and allow for the development of an adaptive QBX method where parameters are selected automatically, given an error tolerance \cite{Klinteberg2017}.
  
In this paper, we consider the Laplace's equation in multiply-
connected domains in three dimensions, and we focus on the further
development of a QBX based method. In contrast to
\cite{Klinteberg2016}, we assume no specific shapes of the boundaries \tb{other than that they are smooth,}
but we will not take on the daunting task of integrating QBX into a
FMM framework for a general three dimensional problem.  Instead, we
view the QBX technique as a local correction.  We divide the surface
into panels, and only for panels close to the evaluation point will
the QBX approach be applied, for other panels the integral is well
resolved using regular quadrature. \tr{This local point of view is not new, and versions of  a local QBX method have been developed for 2D problems in \cite{Barnett2014}, \cite{RachhPhD}. } 
%With this approach 
In extending this idea to 3D, 
it is
essential that the number of terms in the QBX expansion is as small as
possible. We however do not need a separation between source and
target in these expansions, which is essential e.g.\ for the
FMM. Therefore, we will use a {\em target specific} QBX expansion,
that will need $p$ terms to achieve the same accuracy as a spherical
harmonics expansion with $p^2$ terms.
%%sum n=0 to p-1 sum m=-n to n is equal to p^2.

We refer to the method that we have developed as a {\em local
  target-specific} QBX method. \tr{ Its main significance is that it can compute both singular and nearly singular integrals for general surfaces in 3D with high accuracy and $O(p\cdot N)$ complexity. } 
Even though we are considering Laplace's equation
in this paper, the method may also be extended for the accurate
evaluation of single and double layer potentials arising in other
boundary value problems, such as potential flow, Stokes flow, electromagnetics, or
elasticity.  One main motivation for developing this method is to
efficiently and accurately evaluate nearly singular integrals for a time-dependent
geometry where precomputation is not possible \tb{ (although the method presented here makes use of precomputation, it is not necessary)}. 
  This can involve interacting drops, vesicles or blood
cells, where the geometry is changing due to the movement and
deformation of these objects, and where accurate evaluation of nearly
singular integrals is needed to resolve close interactions.

%\textcolor{red}{Need to say about truncation error analysis in the
%  paper}. 

The principal results justifying our use of a local QBX expansion scheme are the analyses of the two main sources of error, truncation and coefficient error,  presented in Section \ref{sect:qbxerrs}.  There, we provide new estimates for the truncation error of the series expansion  of  a local layer potential for the Laplace kernel in 3D (in contrast to the global layer potentials analyzed in \cite{Epstein2013}). We also make use of the analysis in \cite{Klinteberg2016b} to quantify the quadrature error in computing the QBX expansion coefficients. This latter error will be called the {\it coefficient error}.  Crucially, the estimates highlight the interplay between the grid size $h$, the distance $r_c$ of the expansion centers from the boundary, and the order of truncation $p$ on the accuracy of our method. The error estimates also give a rational basis for the choice of numerical parameters.  

The method presented here can be made to have $O(N)$ complexity, where $N$ is the total number of target points,  by scaling  the size of the local correction patch so that it has a constant number of source points as $N$ grows. Optimal complexity is then achieved by combining our local QBX method with a fast hierarchical method, such as the Fast Multipole Method, to compute the contribution to the layer potential from source points that are outside of the local correction patch \tb{ (see \S \ref{sect:full_alg})}.  One advantage of this approach is that by  decoupling the FMM from the QBX expansions the algorithm \tb{allows for the use of pre-existing or standard FMM software,} and is greatly simplified.  We will see that for this scaling, the coefficient and truncation errors remain fixed as $N$ increases, but can be made controllably small. Used in this way our QBX method is not classically convergent, but has controlled precision \cite{Klockner}.  A classically convergent scheme can be achieved by letting the number of local patch points grow (even slowly)  with $N$, at the expense of optimal complexity.   

The rest of this paper is as follows. We start by giving the
problem statement and the integral equations in \S \ref{sect:inteq}. Surprisingly enough, we
could not find a derivation in the literature of a uniquely solvable second kind integral equation for
the multiply-connected external Dirichlet problem in three dimensions, and we will here
provide a brief derivation. In \S \ref{sect:direct_quad}, we
introduce the surface discretization and the regular quadrature rules,
before we describe the local target specific QBX expansions in
\S \ref{sect:localexp}. Results from our error analysis are presented in \S \ref{sect:qbxerrs}. A
description of the full algorithm is given in \S \ref{sect:local_qbx_alg},
before we turn to presenting the numerical results in \S \ref{sec:results}.

%\textcolor{red}{Did not include anything about the complexity of
%  method, not sure what to say. If used for all N
%target points, for O(N) complexity need patch for QBX eval to be of size O(1) i.e. the number of points
%in the patch should be constant as N grows. Then if we move the
%center proportionally, r/R is constant. Hence, if we want to reduce
%the truncation error, we need to increase p. But if this depends on p
%more strongly than log N, we ruin the complexity even with the target
%specific method. Also for the quadrature error, if r/h is constant, do
%not get better accuracy, and for this, we would like to decrease r
%more slowly., which also would call for an increase in p.}

\section{Problem statement, layer potentials and the integral equations}
\label{sect:inteq}

\tp{
Let $D_1, \ldots, D_M$ be a collection of disjoint, bounded, and  open
regions in $\mathbb{R}^3$, each with $C^2$  connected boundary
$\partial D_j$, and let $D=D_1 \cup \ldots D_M$.  
We consider the  
Dirichlet problem for Laplace's equation
\begin{eqnarray}
 \nabla^2 u &=& 0~~~\mbox{in the interior or exterior of } D\\
u &=& f ~~\mbox{on} ~~\partial D
\end{eqnarray}
for continuous boundary data $f$.  For the exterior problem, it is required that $u(\bx) \rightarrow 0$ uniformly in all directions.  }   
%We assume the regions %in the collection $\{ D_j \}$
%consist of one or more fixed shapes which are repeated in space with  different positions an%d possibly different orientations.    This type of geometry arises frequently in applications, e.%g., electromagnetic scattering, particulate Stokes flow, etc.  We consider the following proble%ms for Laplace's equation in 3D space:  \\

%\noindent
%{\it Interior Dirichlet problem.}
%Find a function $u$ that is harmonic in $D$, continuous in $\bar{D}$, and which satisfies the boundary condition
%\begin{equation}
%u=f~~\mbox{on} ~~\partial D,
%\label{eq:interior_dirichlet}
%\end{equation}
%where $f$ is a given continuous function.\\

%\noindent
%{\it Exterior  Dirichlet problem.}
%Find a function $u$ that is harmonic in $\mathbb{R}^3  \backslash \bar{D}$, continuous in $\mathbb{R}^3  \backslash D$, and which satisfies the boundary condition
%\begin{equation}
%u=f~~\mbox{on} ~~\partial D,
%\label{eq:exterior_dirichlet}
%\end{equation}
%where $f$ is a given continuous function. For $|\bx| \rightarrow \infty$ it is required that $u(\bx) \rightarrow 0$ uniformly in all directions.

It is well known that each of these boundary value problems has a
unique solution which depends continuously on the boundary data.  We
now want to formulate the integral equations for solving both the
interior and exterior Dirichlet problem. We wish to have a second kind
integral equation which will yield a well-conditioned discrete
problem. This is straight forward for the interior problem, but much
less so for the exterior problem, and we therefore start by presenting
this formulation.  
In doing so, we will need to introduce layer potentials and  jump relations \cite{Kress}.

Let 
\[
G(\bx,\by)=\frac{1}{4 \pi} \frac{1}{|\bx-\by|}
\]
be the fundamental solution or free-space Green's function for Laplace's equation in $\mathbb{R}^3$.
Given a function $\sigma$ which is continuous on the boundary  $\partial D$ of a region $D$, the functions
\be  \label{single_layer}
u(\bx)=\slayer \sigma(\bx) = \int_{\partial D} \sigma(\by) G(\bx,\by) \ dS_{\by}, 
~~~\bx \in \mathbb{R}^3 \backslash \partial D,
\ee
and
\be \label{double_layer}
v(\bx)=\dlayer \sigma(\bx) = \int_{\partial D} \sigma(\by)  \frac{\partial G(\bx,\by)}{\partial \nu(\by)} \ dS_{\by},
~~~\bx \in \mathbb{R}^3 \backslash \partial D,
\ee
are called, respectively, the single layer and double layer potential with density $\sigma$. In the above, $\nu(\bx)$ is the outward normal at a point $\bx \in \partial D$, that is, pointing into the exterior domain $\mathbb{R}^3 \backslash \bar{D}$.

The single and double layer potentials represent harmonic functions
in $D$ and $\mathbb{R}^3 \backslash \bar{D}$. They are used to represent solutions to boundary value problems for Laplace's equation, with the density $\sigma(\by)$ determined by the boundary data. The solution to the Dirichlet problem can be written in terms of the double layer potential alone, while the single layer potential applies to the Neumann problem (there are also combined representations involving both the single and double layer).  For concreteness, we focus on the Dirichlet problem and henceforth our QBX method will be described for the double layer potential. Analogous methods for the single layer potential follow with obvious modifications.

The layer potentials become
singular on the boundary $\partial D$ and are difficult to evaluate
accurately by a numerical method, not only when the evaluation or
{\em target} point $\bx$ is on the boundary $\partial D$, but also 
when it is close to the boundary and the integral is nearly singular. 
% and are then to be interpreted as
%principal value integrals. 
If we let $\bx$ approach the boundary from either the interior or
exterior domain for the double layer potential, the limits are
different. This is expressed by the following   {\it jump relation}  \cite{Kress} that will be used in our formulation:

\begin{theorem}  
Under the given assumptions on $\partial D$ and $\sigma$,
%(a) The single layer potential $u$ with density $\sigma$ is continuous throughout $\mathbb{R}^3$.  On the boundary,
%\[
%u(\bx)=\int_{\partial D} \sigma(\by) G(\bx,\by) \ dS_{\by},  
%~~~\bx \in \partial D,
%\]
%where the integral exists as an improper integral.\\
the double layer potential $v$ with density $\sigma$ can be continuously extended from $D$ to $\bar{D}$ and from $\mathbb{R}^3 \backslash \bar{D} $ to  $\mathbb{R}^3 \backslash {D} $ with limiting values
\be \label{double_layer_jump}
v_\pm(\bx) = \int_{\partial D} \sigma(\by)  \frac{\partial
  G(\bx,\by)}{\partial \nu(\by)} \ dS_{\by} \pm \frac{1}{2} \sigma(\bx),
~~~\bx \in \partial D,
\ee
where 
\[
v_\pm(\bx)= \lim_{h \rightarrow 0} v(\bx \pm h \nu(\bx))),
\]
$\nu(\bx)$ is the outward normal at a point $\bx \in \partial D$, and the integral exists as an improper integral.
%(c) The single-layer potential $u$ with density $\sigma$ satisfies
%\be \label{single_layer_jump}
%\frac{\partial u_\pm}{\partial \nu} (\bx) =  \int_{\partial D} \sigma(\by)  \frac{\partial G(\bx,\by)}{\partial \nu(\bx)} \ dS_{\by}  \mp \frac{1}{2} \sigma(\bx),
%~~~\bx \in \partial D, 
%\ee
%where
%\[
%\frac{\partial u_\pm}{\partial \nu} (\bx) = \lim_{h \rightarrow 0} \nu(\bx) \cdot \nabla u(\bx \pm h \nu(\bx)),
%\]
%and where the integral exists as an improper integral.
\end{theorem}

%\subsection{Derivation for exterior problem}
\subsection{Integral equation formulation}
 Consider the exterior Dirichlet problem for a multiply-connected domain.
Let $v(\bx)=\dlayer \sigma(\bx)$, where $\dlayer$ defines the double layer potential, as in
(\ref{double_layer}).  
%\be \label{Dsigma}
%\dlayer \sigma(x) =\int_{\partial D} \sigma(y)  \frac{\partial G(x,y)}{\partial \nu(y)} \ ds(y),
%~~~x \in \mathbb{R}^3 \backslash  \bar{D}
%\ee
%denote the double-layer potential, so that the solution to our problem is 
%given by  $v(\bx) =\dlayer \sigma (\bx)$.
The jump relation (\ref{double_layer_jump}) provides a second kind integral equation for the
density $\sigma$
\be \label{second_kind}
\frac{\sigma(\bx)}{2}+\dlayer \sigma(\bx) = f(\bx)~~\mbox{for}~~\bx \in \partial D,
\ee
where we associate the function  $v_+(\bx)$ in (\ref{double_layer_jump})  with the Dirichlet data $f$.  The homogeneous version of this equation is
\be \label{homogeneous}
\left( \frac{1}{2} + \dlayer \right) \sigma(\bx)=0 ~~\mbox{for}~~\bx \in \partial D.
\ee
Unfortunately, the above equation has nontrivial solutions, for example, $\sigma_k = \chi_k$ where $\chi_k$ is the characteristic function on
boundary component $\partial D_k$.   For such $\sigma_k$, we have from a well-known result in classical potential theory \cite{Kress} that 
${\cal D} \sigma_k(\bx)=-1/2$ if $\bx \in \partial D_k$,
%\[
%\dlayer \sigma_k (\bx) = \left\{ \begin{array}{ll}
%         0 & \mbox{if  $\bx  \notin \bar{D}_k$};\\
%        -\frac{1}{2} & \mbox{if $\bx \in \partial D_k$},\\
%   -1 &\mbox{if $\bx \in  D_k$}
%\end{array} \right.
% \]
and it immediately follows that $\sigma_k$ satisfies (\ref{homogeneous}).
Moreover, it can be shown \cite{Kress} that the $M$  linearly independent functions
$\sigma_k$ for $k=1, \ldots, M$  provide a basis for the null space of the homogeneous equation.  In other words, the
dimension of the 
null space of the homogeneous operator  in (\ref{homogeneous}) is equal to the number of boundary components in our domain. 

\tp{
We will modify the second kind integral equation (\ref{second_kind}) so that it has a unique solution (see \cite{Kress} for a modified equation in the case of a single boundary component $M=1$, and \cite{Greenbaum1993}, \cite{HelsingWadbro2005} for a similar approach in 2D).  Our approach is motivated by Tausch and White \cite{TauschWhite},  who considered the so-called capacitance problem, which is the adjoint of the problem considered here. 
%Let $\mathbf{t}=(t_1, \ldots, t_M) \in \mathbb{R}^M$ be a constant vector and define the functions
%\begin{align}
%A \bt(\bx) &= \sum_{k=1}^M \frac{t_k}{\sqrt{|S_k|}} %\chi_k(\bx)~~\mbox{for}~~\bx \in \partial D, \\
%B \bt(\bx) &=  \sum_{k=1}^M t_k G(\bx_k,\bx) ~~\mbox{for}~~\bx \in \partial D, 
%\end{align}
Let 
\begin{equation}
A \sigma(\bx) = \sum_{k=1}^M \left( \frac{1}{\sqrt|S_k|} \int_{\partial D_k} \sigma(\by) \ dS_{\by} \right) G(\bx_k,\bx),
\label{eq:BAprime}
\end{equation}
where $\bx \in \partial D$ and we recall $G$ is the free-space Green's
function. 
Here $\bx_k$ is any point in the interior of region $D_k$ and 
\[
\left| S_k \right| = \int_{\partial D_k} dS_{\by} 
\]
is the surface area of $D_k$.  
%We will also need the adjoints of the operators $A$ and $B$, which are
%\begin{align}
%\left[ A' \sigma \right]_k &=\frac{1}{\sqrt{|S_k|}} \int_{\partial D_k} \sigma(\by) \ dS_{\by},  \label{A_adjoint} \\
%\left[ B' \sigma \right]_k &= \int_{\partial D} G(\bx_k,\by) \sigma(\by)  \ dS_{\by}. \label{B_adjoint}
%\end{align}
The modified second kind equation for the density $\sigma$ is
\be \label{modified_second_kind}
\left( \frac{1}{2} + \dlayer + A \right) \sigma(\bx)=f(\bx)~~\mbox{for}~~\bx \in \partial D.
\ee
}
For the modified  equation we have
\begin{theorem} \label{thm1}
The second kind integral equation (\ref{modified_second_kind}) has a unique solution. Moreover, $u= (\dlayer  + A) \sigma$ where $\sigma$ satisfies (\ref{modified_second_kind}) is a solution to the exterior Dirichlet problem.
\end{theorem}

\begin{proof}
\tp{
The uniqueness theorem can be established by showing that the null space of the adjoint of the second-kind
operator  in (\ref{modified_second_kind}) is $\left\{ 0 \right\}$ (by the First Fredholm Alternative \cite{Kress}, the dimensions of the null space of an operator  and its adjoint are the same).
This result follows similarly to the proof
of Theorem 2.1 in  
\cite{TauschWhite}. To get the solution to the exterior Dirichlet problem, note that $\tilde{\dlayer} \sigma = \dlayer \sigma + A  \sigma$ satisfies the same jump condition as $\dlayer \sigma$, so  $u(\bx) = \tilde{\dlayer} \sigma (\bx)$ for $\bx \in 
\mathbb{R}^3 \backslash \bar{D}$, where $\sigma$ is a solution to (\ref{modified_second_kind}), is a solution to our problem. }
\end{proof} 

\noindent
For the exterior problem we solve
(\ref{modified_second_kind}) for $\sigma$ and find the solution for
$\bx \in \mathbb{R}^3 \backslash \bar{D}$ by
$u= (\dlayer + A) \sigma$.
%Note that 
%the function $BA' \sigma(\bx)$  
%can be written as
%\begin{equation}
%BA' \sigma(\bx) = \sum_{k=1}^M \left( \frac{1}{\sqrt|S_k|} \int_{\partial D_k} \sigma(\by) \ dS_{\by} \right) G(\bx_k,\bx),
%\label{eq:BAprime}
%\end{equation}
%where $\bx \in \partial D$ and we recall $G$ is the free-space Green's
%function.   

%For $M$ regions and $N$ total target points on $\partial D$, the operation count in com%puting this function is $O(MN)$. This can be inefficient for  large $M$.   A more complic%ated but more efficient method
%is described below.

%\subsection{Interior and exterior Dirichlet problem}
For the interior Dirichlet problem we again let $u(\bx)=\dlayer
\sigma(\bx)$, 
%(\ref{double_layer}), 
now for $\bx \in D$.  Using the jump relation
(\ref{double_layer_jump}) for the double layer potential, we obtain
the integral equation
\begin{equation}
-\frac{\sigma(\bx) }{2}+\dlayer \sigma(\bx) = f(\bx)~~\mbox{for}~~\bx \in \partial D.
\label{eq:interior_second_kind}
\end{equation}
The homogeneous version of this integral equation has no non-trivial
solution and hence does not need to be modified. This is a second kind
integral equation that we need to solve for $\sigma$, and then find the
solution by $u(\bx)=\dlayer \sigma(\bx)$ for $\bx \in D$.

\subsection{The discrete problem}

To solve the integral equation for either the interior
 or exterior
 problem, we first need a method to
numerically evaluate integrals over $\partial D$, the boundary of the
(possibly multiply-connected) domain.  We use the notation
\[
I[f]=\int_{\partial D} f(\by) dS_{\by},
\]
which we numerically evaluate by a quadrature rule $Q_N$, that
defines a set of $N$ nodes $\by_i$ and weights $w_i$ on $\partial D$,
such that 
\begin{equation}
I[f] \approx Q_N[f]= \sum_{i=1}^N f(\by_i) w_i. 
\label{eq:QN}
\end{equation}
Details of the quadrature we actually use will be discussed in the
following two sections. 
Let us indicate by a superscript $h$ a quantity that is computed using
$Q_N$, e.g.\ 
%\[
%\dlayer^h \sigma (\bx) = Q_N \left[ \sigma(.) \nu(.) \cdot \nabla_. G(\bx,.)    \right].
%\]
\[
\dlayer^h \sigma (\bx) = Q_N \left[ \sigma(\cdot) K(\bx,\cdot)    \right], 
\]
where $K(\bx,\by)=\nu(\by) \cdot \nabla_{\by} G(\bx,\by)$ .  

We apply the Nystr\"{o}m method, where the integral equation is
enforced at the quadrature nodes, i.e.,  for $\bx_i=\by_i$, $i=1,\ldots,N$.
The boundary integral equation for the interior problem
 is approximated by the $N \times N$
linear system for the density values $\sigma(\bx_i)$,
\begin{equation}
-\frac{1}{2} \sigma(\bx_i)+\dlayer^h \sigma(\bx_i)=f(\bx_i), \quad i=1,\ldots,N. 
\label{eq:interior_discrete}
\end{equation}
and similarly for the exterior problem (\ref{modified_second_kind}), 
\begin{equation}
\frac{1}{2} \sigma(\bx_i)+\dlayer^h \sigma(\bx_i) +A^h \sigma(\bx_i)
=f(\bx_i), \quad i=1,\ldots,N. 
\label{eq:exterior_discrete}
\end{equation}
where $A$ is defined in (\ref{eq:BAprime}). 

After these discrete values of $\sigma$ have been determined, the
solution can be computed at any point in the domain by
\begin{equation}
u^h(\bx)=\dlayer^h \sigma(\bx), \quad \bx \in D, 
\label{eq:interior_postproc}
\end{equation}
for the interior problem, and 
\begin{equation}
u^h(\bx)=(\dlayer^h +A^h) \sigma(\bx), \quad \bx \in \mathbb{R}^3
\backslash \bar{D}, 
\label{eq:exterior_postproc}
\end{equation}
for the exterior problem. 

The matrix for the linear system ((\ref{eq:interior_discrete}) or
(\ref{eq:exterior_discrete})) is dense, and solving it directly would
incur a cost of $O(N^3)$. For these well-conditioned formulations, the
number of iterations needed in an iterative solution method such as
GMRES is independent of the discretization. The problem can hence be
solved at $O(N^2)$ cost, where the constant depends on the geometry of
the problem. The $O(N^2)$ cost arise from the matrix-vector multiply
for a dense system, i.e.\ evaluating the discretization of
$\dlayer^h \sigma$.  As discussed in the introduction, applying a fast
method such as a fast multipole method, a treecode or a method based
on FFTs, this cost can be further reduced to $O(N)$ or $O(N \log N)$
which means that the linear system can be solved in (essentially)
linear time. In the numerical examples given in Section \ref{sec:results}, we simply use the $O(N^2)$ direct summation to perform the matrix-vector multiply in cases that involve only a few (one or two) domains $D_k$. For problems with more domains, we use the $O(N \log N)$ treecode algorithm described in \cite{Krasny}. \tr{Details are given in Section \ref{sect:full_alg}.}

\section{Direct quadrature based on surface panels} 
\label{sect:direct_quad}

Before describing our QBX scheme, we introduce a direct surface integration method based on Gauss-Legendre quadrature that will
be used by our scheme. 
%%\subsection{Alternative formulation using Lagrange multipliers}
%\section{Surface representation and quadrature rules} 
%\label{sect:surfacerep} 
In this work, we consider as our boundaries surfaces of genus 0, that
each has a parameterization in spherical coordinates
%\textcolor{red}{No index $k$ on  $L_\theta$ and $L_\varphi$. Either
%  make it general or simply make it $\pi$ and $2 \ pi$?}

\[
\partial D_k= \left\{ \bx(\theta, \varphi): 0 \leq \theta \leq
 \tr{\pi} , 0 \leq \varphi < \tr{2 \pi }\right\}, k=1,\ldots,M, 
\]
and an integral over the surface is written as 
\be \label{eqn:surface_int}
\int_{\partial D_k} f(\by) \, dS_{\by}= 
 \int_0^{2 \pi} \int_0^{\pi} f(\by(\theta,\varphi)) \saelem(\theta,\varphi) \ d\theta d\varphi,
\ee
where $\saelem(\theta,\varphi)=|\by_\theta (\theta,\varphi) \times \by_\varphi(\theta, \varphi)|$ denotes the surface area element.

For this discretization, the $(\theta,\varphi)$ plane is for each surface
$\partial D_k$ divided so that the surface is tiled with
$N^k_P=N^k_{\theta} \times N^k_{\varphi}$ surface panels, with a total of
$N_P$ panels over all surfaces.  On each panel, we use an $q$-point
Gauss-Legendre quadrature rule in both coordinate directions, and
hence $q \times q$ quadrature points over the patch.  This gives us a
total of $N= N_P \times q^2$ quadrature points on the surfaces.

For ease of notation we can let one index cover all the quadrature points $\by_i$ and weights $w_i$, where $i=1, \ldots ,N=q^2N_p$. Then the quadrature rule is given by (\ref{eq:QN}), with the weights $w_i$ being the product of the surface area element $W(\by_i)$ and the Gauss-Legendre quadrature weights. This quadrature rule will be referred to as the direct quadrature.
We will henceforth seek the solution $\sigma$ of the
discretized integral equation (\ref{eq:exterior_discrete}) at the
quadrature points of this direct quadrature. 
When applied to compute the double layer potential,
the direct integration $Q_N \left[ \sigma(\cdot) K(\bx,\cdot) \right]$  or its upsampled version described in  \S \ref{sect:upsamp} below can be sufficiently accurate if the target point is not too close to any source  panel. If it gets too close, the QBX method will be applied as a local correction (see Figure \ref{fig_upsample}). 

%We let $(\theta^P_i,\varphi^P_j)$, $i,j=1,\ldots,q$ be the $q^2$
%quadrature points on patch $P$, with quadrature weights
%$\lambda_i \lambda_j$, and our basic quadrature approximation of the
%surface integral over the patch $P$ becomes
%\begin{equation}
%Q_q^P[f]=  \sum_{i=1}^q \sum_{j=1}^q f(\by(\theta^P_i,\varphi^P_j))
%\saelem(\theta^P_i,\varphi^P_j) \, \lambda_i \lambda_j, 
%\label{eqn:panel_quad}
%\end{equation}
%such that the full surface integral over $\partial D$ will be approximated by 
%\begin{equation}
%Q_{N_P,q}[f]= \sum_{P=1}^{N_P} Q_q^P[f].
%\sum_{i=1}^q \sum_{j=1}^q f(\by(\theta^P_i,\varphi^P_j))
%\saelem(\theta^P_i,\varphi^P_j) \, \lambda_i \lambda_j.
%\label{eqn:surface_quad}
%\end{equation}

For smooth integrals, it is often advantageous to let the whole
surface be only one panel, and refine the grid by increasing $q$,
thereby achieving spectral accuracy. We however want to keep a
locality in our discretizations, since we will modify the direct quadrature
locally close to a singular or nearly singular point. 
In our computations, $q$ is typically set to $7$ for the direct quadrature,
%and locally upsampled quadrature, 
which gives a
high order quadrature rule on each panel (15th order), without each
panel becoming too large. 
%{\em A typical value of $q$  for the local corrections using QBX is
 % $q=15$.}

%The quadrature rule in (\ref{eqn:surface_quad}) will be referred to as
%the direct quadrature. 
%For ease of notation we can simplify this expression by letting one
%index cover all quadrature points and introducing $\by_i$ and $w_i$, 
%$i=1,\ldots,N=q^2 N_P$, such that
%\begin{equation}
%Q_N[f]= \sum_{i=1}^{N}  f(\by_i) w_i=Q_{N_P,q}[f].
%\label{eqn:surface_quad_globind}
%\end{equation}
%The quadrature rule $Q_N$ was already schematically introduced in
%(\ref{eq:QN}), and we will hence seek the solution $\sigma$ of the
%discretized integral equation (\ref{eq:exterior_discrete}) at the
%quadrature points of this direct quadrature.

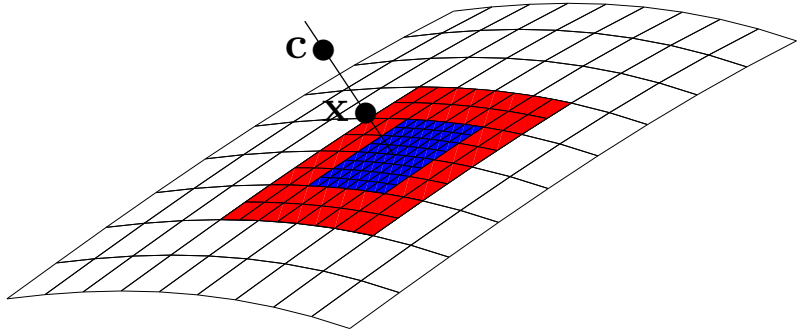
\begin{figure}
\vspace{-1.5in}
  \center{\includegraphics[clip=true, viewport=80 120 800 600, scale=.5]{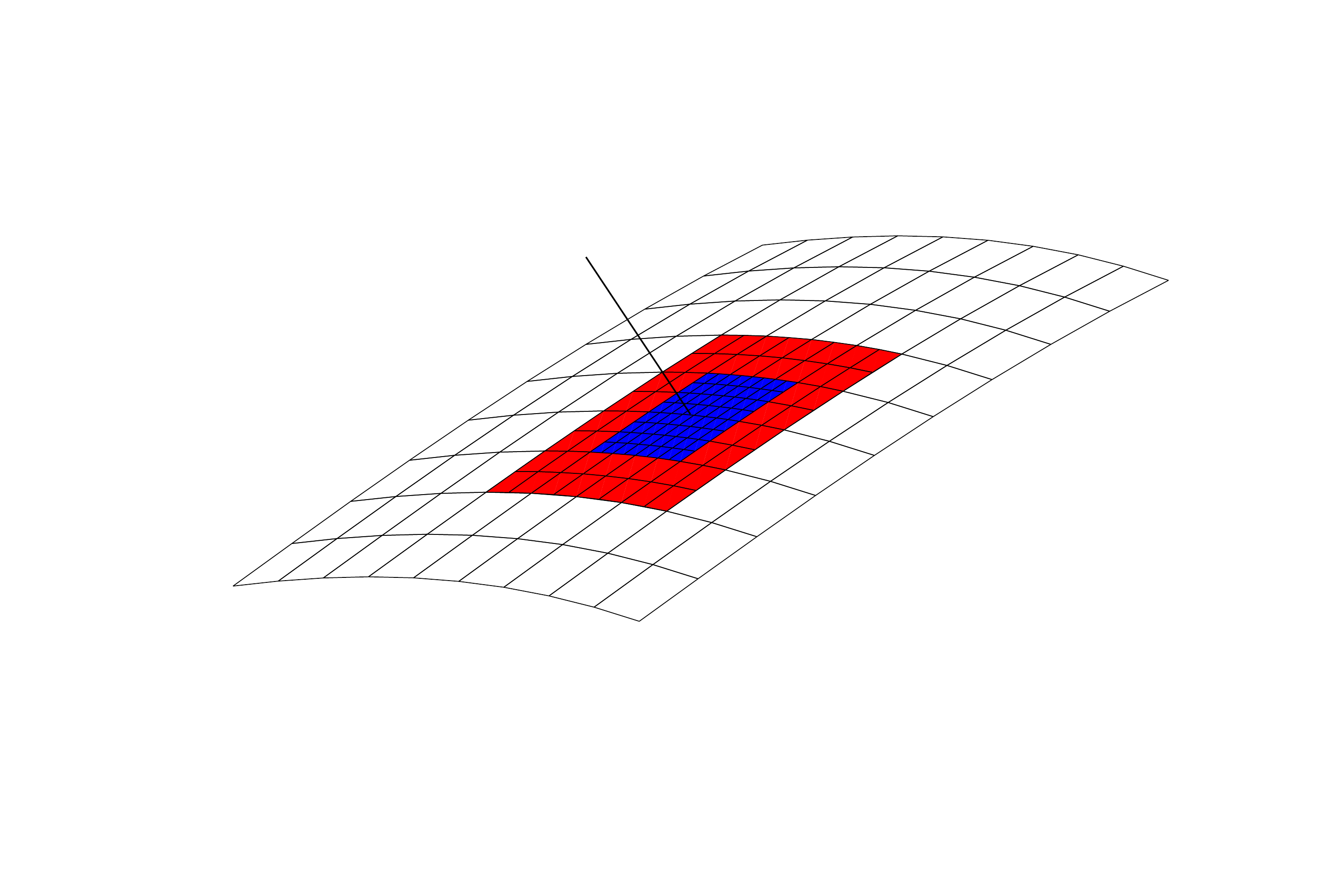}}   
	\put(-196,106){$\CIRCLE$}
       \put(-208,106){{\Large $\bx$}}
       \put(-212,130){$\CIRCLE$} 
        \put(-222,130){{\Large $\bc$}}
  \caption{\tb{ Tiling of the surface by Gauss-Legendre  panels. In this example, ${\cal D}\sigma(\bx)$ is computed using direct quadrature over the white region,   upsampled direct quadrature over the red region (with $\kappa=2$),  and  upsampled QBX over  the blue  region ($\kappa=4$).   }} 
  \label{fig_upsample} 
\end{figure}

\subsection{Upsampling for quadrature}
\label{sect:upsamp}

Assume that we want to evaluate an integral 
\be \label{eqn:surface_kernel_int}
\int_{\partial D_k} K(\bx,\by) \sigma(\by) \, dS_{\by}= 
 \int_0^{2 \pi} \int_0^{\pi} K(\bx,\by(\theta,\varphi)) \sigma(\by(\theta,\varphi)) \saelem(\theta,\varphi) \ d\theta d\varphi,
\ee
where $K(\bx,\by)$ is,  e.g.,  the kernel of the double layer
potential (\ref{double_layer}). 
Assume that the density $\sigma$ is known at 
the $N$ nodes of the direct quadrature rule.
If the kernel $K$ varies rapidly and is not well resolved on that grid, a
refinement can be made to increase the accuracy of the quadrature. 
The refinement is effective since the kernel is known analytically, but  the density $\sigma$ must be
interpolated and evaluated at these new points. 

For a panel $P$, we denote the interpolant of $\sigma$ by $\Pi^P
\sigma$. This interpolant is defined through the values of $\sigma$ at
the $q \times q$ Gauss-Legendre points. Now, split the panel $P$ into
$\kappa^2$ subpanels, each with $q^2$ Gauss-Legendre points. The
values of $\Pi^P \sigma$ at these $\kappa^2 q^2$ points can be evaluated
using barycentric Lagrange interpolation  \cite{Trefethen}.

We will refer to this interpolation procedure as an upsampling by a
factor of $\kappa$.  The upsampled quadrature of a function $f$ over a panel $P$ is a sum over all
$\kappa^2$ sub-panels of $P$, for each one using the $q^2$ interpolated
values of $f$. 
If the original quadrature over one panel is denoted $Q^P_q[f]$ 
%and defined in 
%(\ref{eqn:panel_quad}). 
then the  upsampled quadrature is denoted by  
\begin{equation}
Q^P_{q,\kappa}[f].
\label{eqn:Q_P_kappa}
\end{equation}
For the upsampled quadrature  of the double layer potential $Q^P_{q,\kappa}[K \sigma]$, $K$ is known analytically, so as described
above, this entails upsampling of $\sigma$. 
%The barycentric interpolation of $\sigma$ will give an error of order
%${q_{int}}$ in the panel size $h$.  

\section{Local expansions and the QBX method} \label{sect:localexp}

The QBX method  is used to accurately compute the double layer integrals when the source panel is close to the  target point. It makes essential use  of an expansion (Taylor's, spherical harmonic, etc.) centered about a point that is located just off of the surface, but further away  from the source panel than the target.  
To introduce the basic idea of the QBX
method, we first consider expansions of the Green's function. 
Introduce two points $\bx,\by \in \reals^3$, and a 
point $\bc \in \reals^3$ such that
$|\bx-\bc|<|\by-\bc|$. Furthermore, let 
$\theta$ be the angle between $\bx-\bc$ and $\by-\bc$ (see Figure \ref{fig_qbx}). 
We then have the following expansion around the center $\bc$,
\begin{equation}
\frac{1}{|\bx-\by|}=\sum_{n=0}^{\infty} 
\frac{|\bx-\bc|^n \ \ }{|\by-\bc|^{n+1}} P_n(\cos \theta), 
\label{eqn:Pn_expand}
\end{equation}
where $P_n$ is the Legendre polynomial of degree $n$. 

The addition theorem for Legendre polynomials, also called
the spherical harmonics addition theorem reads
\begin{equation} \label{addition_thm}
P_n (\cos \theta)=\frac{4 \pi}{2n+1} \sum_{m=-n}^n Y_n^{m}(\theta_y,\varphi_y) Y_n^{-m} (\theta_x,\varphi_x),
\end{equation}
%where $(\theta_x, \varphi_x)$ and $(\theta_y, \varphi_y)$ are spherical coordinates of 
%the unit vectors $(x-x_c)/|x-x_c|$ and $(y-x_c)/|y-x_c|$.
where $(\theta_x,\varphi_x)$ and  $(\theta_y,\varphi_y)$ are the
spherical coordinates of $\bx-\bc$ and $\by-\bc$. 
Using this theorem, (\ref{eqn:Pn_expand}) can be further expanded as
\begin{equation}
\frac{1}{|\bx-\by|}=\sum_{n=0}^{\infty} \frac{4\pi}{2n+1}
\sum_{m=-n}^{n} \frac{|\bx-\bc|^n \ \ }{|\by-\bc|^{n+1}}
Y_n^{-m}(\theta_x,\varphi_x) Y_n^m(\theta_y,\varphi_y), 
\label{eqn:sph_expansion}
\end{equation}
where $Y_n^m$ is the spherical harmonic of degree $n$ and order $m$,
%and $(\theta_x,\varphi_x)$ and  $(\theta_y,\varphi_y)$ are the
%spherical coordinates of $\bx-\bc$ and $\by-\bc$. 
and is defined as
\be \label{Ynm_def}
Y_n^m(\theta,\phi)= \sqrt{\frac{2n+1}{4 \pi} \cdot \frac{(n-|m|)!}{(n+|m|)!}} P_n^{|m|}(\cos \theta) e^{i m \phi}
\ee
where $P_n^m$ is the associated Legendre function of degree $n$ and order $m$.

Consider now the double layer potential in (\ref{double_layer}). 
%\[ 
%\slayer \sigma(\bx)=\int_{\partial D}  \sigma(\by) \,  G(\bx,\by) \, dS_{\by}.
%\]
Using the spherical harmonics expansion (\ref{eqn:sph_expansion}), 
this can be written as
\be \label{eq:qbx_sum}
\dlayer \sigma(\bx)=
\sum_{n=0}^{\infty} \frac{4\pi}{2n+1}
\sum_{m=-n}^{n} z_{nm} |\bx-\bc|^n Y_n^{-m}(\theta_x,\varphi_x) , 
\ee
where 
\be \label{eq:qbx_znm}
z_{nm}=
\frac{1}{4 \pi} \int_{\partial D}  \sigma(\by) \, 
\nu(\by) \cdot \nabla_\by \left[\frac{1 }{|\by-\bc|^{n+1}}
Y_n^m(\theta_y,\varphi_y)  \right] dS_{\by}.
\ee
The domain of convergence for this local expansion is the ball
centered at $\bc$ with a radius $r_c$, 
\be \label{eqn:rc}
r_c=\min_{\by \in \partial D} |\bc -\by|,  
\ee
and as shown in Figure \ref{fig_qbx} includes the point  where the ball touches the surface \cite{Epstein2013}. 
%The  expansion for the double layer potential in
%(\ref{single_layer}) can be derived analogously. 

In the QBX method, such expansions are used to compute layer
potentials when the target point $\bx$ is such that the integral is singular
or nearly singular. An expansion center $\bc$ is placed further away
from $\partial D$ than $\bx$, and a truncated expansion is formed
where coefficients are evaluated by numerical approximation of the integrals  (\ref{eq:qbx_znm})
defining them. 
High accuracy can be achieved if sufficiently many terms  are kept in
the expansion and care is taken in the evaluation of the
coefficients.

\subsection{Target specific QBX expansions}

In the QBX method, one expansion center is usually associated with
each discretization point on the boundary $\partial D$. This means
that the number of target evaluations based on each expansion typically is
small.  Here, we view the QBX technique as a local correction, not to
be built into an FMM.  We also want to use the method in situations
where geometry considerations do not allow for
pre-computation to speed up the QBX evaluations (although we will use precomputation when it is applicable).  In this case, it is
important that the number of terms in the expansion is as small
as possible, and that the coefficients can be computed efficiently. 
The separation between source and target that is achieved by
the spherical harmonics expansion (\ref{eqn:sph_expansion}), meaning
that the coefficients $z_{nm}$ in (\ref{eq:qbx_znm}) are independent
of the evaluation/target point, will not be of significant use.  The number of
coefficients to be computed in a 3D spherical harmonics expansion of degree $p-1$ is $p^2$.
By considering instead the
equivalent expansion using the original formulation (\ref{eqn:Pn_expand}), there are only $p$ terms, and in addition, we only have to work with
the Legendre polynomials. This yields significant cost savings.

This {\em target specific} expansion for the 
%single layer (\ref{single_layer}) and
 double layer potential  (\ref{double_layer}) 
can be written as  
\be
%\slayer \sigma(\bx)=
%\sum_{n=0}^{\infty} \zs_{n}(\bx) |\bx-\bc|^n , \quad \quad 
\dlayer \sigma(\bx)=
\sum_{n=0}^{\infty} z_{n}(\bx) |\bx-\bc|^n , 
\label{eq:target_spec_exp}
\ee
where 
\be
z_{n}(\bx)=
\frac{1}{4 \pi} \int_{\partial D}  \sigma(\by) \, 
\mathbf{\nu}(\by) \cdot \nabla_\by \left[ \frac{1 }{|\by-\bc|^{n+1}}
P_n(\cos\theta_{\bx,\bc,\by})   \right] dS_{\by}, 
\label{eq:qbx_znD}
\ee
and the notation $\theta_{\bx,\bc,\by} $ is to emphasize that this
is the angle between $\bx-\bc$ and $\by-\bc$, and hence depends on the
target.    

We can also relate the expansion  (\ref{eqn:Pn_expand}) to a Taylor expansion in Cartesian
coordinates that with $\mathbf{k}=(k_1,k_2,k_3)$, $k_i \in \mathbb{Z}_{\geq 0}$, reads
\begin{align}
\frac{1}{|\bx-\by|} = \sum_{ \mathbf{k} } \left( \frac{1}{\mathbf{k}!}
  D_x^\mathbf{k} \frac{1}{|\bc-\by|}  \right)
(\bx-\bc)^\mathbf{k} 
=\sum_{ \mathbf{k} } b_{\bk}(\bc,\by) (\bx-\bc)^\mathbf{k}. 
\label{Taylors_approx} 
\end{align}
The coefficients $b_{\bk}$ obey the following recursion relation 
\cite{DuanKrasny, Krasny}, 
%{\em (I have swapped to  $(y_i-c_i$, believe it had the wrong sign
 % before.) ??}
\begin{equation}
\normk R^2 b_{\bk}-(2 \normk -1) \sum_{i=1}^3 (y_i-c_i)b_{\bk-\ei}
+(\normk -1) \sum_{i=1}^3 b_{\bk-2\ei}=0, 
\label{eqn:rec_rel}
\end{equation}
where $\| \bk \|=k_1+k_2+k_3$, $b_0=1/|\by-\bc|$, and
$R=|\by -\bc|$.
Any coefficient $b_{\bk}$ with a negative index is set to $0$.  

In Appendix B
%%\ref{app:equivexp} 
we show that the error incurred by
truncating the Taylor expansion (\ref{Taylors_approx}) after including
all spherical shells such that $\|\bk\| \le p$ is the same as the
error obtained when truncating the spherical harmonics expansion
(\ref{eqn:sph_expansion}) at $n=p$, i.e., once all spherical harmonics
up to degree $p$ have been included.  Naturally, this is also  the same as
truncating the original expansion  (\ref{eqn:Pn_expand}) at $n=p$.  

As defined above, the coefficients  $z_n$ are referred to
as {\em global}, since they contain information from all of
$\partial D$.  In the following, we will introduce our local approach.
For the surface discretization in Section
\ref{sect:direct_quad}, each surface is divided into panels. Only
for panels close to the evaluation point will the QBX technique be
applied, meaning that the integrals defining 
%$\zs_n$ and 
$z_n$ will
only be over a part of $\partial D$ (see Figure \ref{fig_upsample})..  For smooth geometries and layer
densities, the total field produced by the layer potential is smooth,
and the coefficients in a local expansion decay rapidly.  If we
consider the contribution only from a patch of the surface, then the
field that this part produces will not be as smooth. The decay of the
coefficients in the QBX expansion will depend on the ratio of the
distance from the expansion center to the surface as compared to the
distance from the center to the edge of the surface patch.  This will
be further discussed in Section \ref{sect:qbxerrs}.

\subsection{Local target-specific QBX evaluation} 
\label{sect:localQBX} 
%The target specific expansion. Refer to appendix for equivalence
%between cartesian and spherical harmonics expansions.
%(Using this when distance to surface smaller than expansion
%radius. Otherwise better to just upsample). 

When we use the QBX evaluation for a given target point $\bx$, we want
to do so including the contribution only from a set of panels close to
$\bx$. Assume that a set of panels have been selected, and denote the
part of $\partial D$ that they constitute by $\setofpanels$ (e.g., the blue panels in Figure \ref{fig_upsample}). 

We denote this local part of the double layer potential by
$\dlayloc \sigma(\bx)$, and the expansion in
(\ref{eq:target_spec_exp}) around the expansion center $\bc$ is now
modified to become
\begin{equation} 
\dlayloc
\sigma(\bx)= \sum_{n=0}^{\infty} \zloc_{n}(\bx) |\bx-\bc|^n ,
\label{eq:qbx_loc_exp}
\end{equation}
where 
\begin{equation} 
\zloc_{n}(\bx)=\frac{1}{4 \pi}
\int_{\setofpanels}  \sigma(\by) \, 
\mathbf{\nu}(\by) \cdot \nabla_\by \left[ \frac{1}{|\by-\bc|^{n+1}}
P_n(\cos\theta_{\bx,\bc,\by})   \right] dS_{\by}. 
\label{eq:qbx_zloc}
\end{equation} 
%We will from now on focus on the evaluation of the double layer
%potential and have dropped the superscript $\mathcal{D}$ for the
%coefficients. 
The difference between the definition of $\zloc_{n}(\bx)$ and 
$z_{n}(\bx)$ in (\ref{eq:qbx_znD}) is that the integral in the
former is only over $\setofpanels$ where as the integral in the
definition of $z_{n}(\bx)$ is over all of $\partial D$. 

We now give some details on the choice of the expansion centers.
Let us introduce 
\[
r=|\bx-\bc|, 
%\quad R=|\by-\bc|, 
\quad \alpha=(\by-\bc)\cdot(\bx-\bc), ~~\mbox{and recall that} ~~ R=|\by-\bc|.
\]
We also further decompose the integral in (\ref{eq:qbx_zloc}) as $\int_{\Gamma^{loc}(\bx)}= \sum_k \int_{\Gamma^{loc}(\bx) \cap \partial D_k}$.  The choice of expansion center $\bc$ will depend on both the target point $\bx$ and the boundary component $\partial D_k$ over which the integration is performed.
%For any given target point $\bx$, we will have $r \le R$ for all  $\by$
%with $r=R$ only when $\bx \in \partial D_k$ and then only for the one
%point $\by=\bx$. 
\tb{In particular, the expansion center for $\bx$  will be chosen such that
\be \label{center_loc}
r=|\bx -\bc| < \min_{\by \in \partial D_k} |\bx -\by|.
\ee
as depicted in Figure \ref{fig_trunc} of \S \ref{sect:qbxerrs}. }
%with equality only when $\bx \in \partial D_k$.   
%The inequality (\ref{center_loc})
%holds also for the nearly singular case, i.e., when $\bx$ is not an element of the boundary component  $\partial D_k$.  
An additional requirement on  the choice of $\bc$ is that 
$\bx-\bc$ be normal to the surface $\partial D_k$ at the point $\by$ which minimizes (\ref{center_loc}).  The distance of the expansion center   from the surface $\partial D_k$  is typically chosen to be similar to  the (maximum)  grid spacing for the direct, non-oversampled  quadrature.

%We expand the expression in integral (\ref{eq:qbx_zloc}) and write
%\[
%\nabla_\by \left[ \frac{\Lambda_\xi(\bx,\by)}{R^{n+1}}
%P_n\left(\frac{\alpha}{r R}\right)   \right]  =
%\left[ \nabla_\by \Lambda_\xi(\bx,\by) \right] 
%\frac{1}{R^{n+1}} P_n\left(\frac{\alpha}{r R}\right)  
% +\Lambda_\xi(\bx,\by) \nabla_\by 
%\left[ 
%\frac{1}{R^{n+1}} P_n\left(\frac{\alpha}{r R}\right)  
%\right] 
%\]
Expanding the derivative in (\ref{eq:qbx_zloc}) yields
\[
\nabla_\by  \left[  \frac{1}{R^{n+1}} P_n\left(\frac{\alpha}{r R}\right)  \right] =
\frac{(n+1) (\bc-\by)}{  R^{n+3}} P_n \left(\frac{\alpha}{r R} \right) 
 + \left( \frac{(\bx-\bc)}{r R} - \frac{\alpha (\bc-\by)}{r R^3}
  \right)  \frac{1}{ R^{n+1}}   P_n^\prime \left(\frac{\alpha}{r R} \right).
\]
Well-known recursion relations \cite{abramowitz} can be used to efficiently compute the functions $P_n(x)$ and $P_n^\prime (x)$ which appear in the above equation.  

Now note that each coefficient $\zloc_{n}(\bx)$ is defined with a
smooth kernel, even if the original integral was singular, allowing
for discretization by regular quadrature. 
In the local target specific QBX evaluation, the expansion 
in (\ref{eq:qbx_loc_exp}) is truncated at $n=p$, and the coefficients
$\zloc_{n}(\bx)$ are evaluated using discrete quadrature. 
We denote the coefficients by $\zloch_{n}(\bx)$ and evaluate
\begin{equation} 
\dlayloch
\sigma(\bx)= \sum_{n=0}^{p} \zloch_{n}(\bx) |\bx-\bc|^n .
\label{eq:qbx_loc_exp_h}
\end{equation}
When $\bx$ lies on the boundary component $\partial D_k$, one has the choice of using an expansion center that lies in the interior or exterior of  $\partial D_k$,  or taking an average of (\ref{eq:qbx_loc_exp_h}) over both centers. This is further discussed in Section \ref{sect:onsurface}.

Next we discuss the errors introduced by the local QBX  procedure. 

%\textcolor{red}{Say about two-sided and one sided}

%Now note that each term in the sum (\ref{local_int_Bl_form}) has a smooth kernel, and so 
%can be  discretized using a smooth quadrature rule.
%Equation (\ref{local_int_Bl_form}) is our approximation of $D \sigma$ using a $p+1$-te%rm expansion and is the basis of our QBX method.

\section{Errors in QBX evaluation}
\label{sect:qbxerrs}

The error $E(\bx,\bc) =|\dlayloc \sigma(\bx)-\dlayloch \sigma(\bx)|$
with the two terms defined in (\ref{eq:qbx_loc_exp}) and
(\ref{eq:qbx_loc_exp_h}), respectively, has three parts. The first
part is that due to the truncation of the expansion at order $p$ and
the second the error from discretization of the integrals when
computing the expansion coefficients. There is a third part that
arises when we upsample (interpolate) the density $\sigma$ to a finer
grid, before the expansion coefficients are computed, which is of order $h^{q_{int}}$.  
%We will not
%consider the third part of the error in this discussion.
 If this becomes
the dominating source of error, the underlying discretization must be
refined before errors can be further decreased by improved quadrature
treatment.

We use the triangle inequality to obtain
\begin{align}
E(\bx,\bc) &\leq 
\left| \dlayloc \sigma (\bx)- \sum_{n=0}^{p} \zloc_{n}(\bx)
             |\bx-\bc|^n \right|
+ \left| \sum_{n=0}^{p} \zloc_{n}(\bx) |\bx-\bc|^n - \dlayloch \sigma
   (\bx) \right|  \notag \\
&=  E_T(\bx,\bc)+E_Q(\bx,\bc)
\label{eq:qbx_err_T_Q}
%%\label{error_t_d}
\end{align}
%Explicitly writing $\dlayloc \sigma(\bx)$ and $\dlayloch \sigma(\bx)$
%from  (\ref{eq:qbx_loc_exp}) and (\ref{eq:qbx_loc_exp_h}), we have
Using  (\ref{eq:qbx_loc_exp}) and (\ref{eq:qbx_loc_exp_h}), this yields
the truncation error
\begin{equation}
E_T=\left| \sum_{n=p+1}^{\infty} \zloc_{n}(\bx)  |\bx-\bc|^n \right|
\label{eq:Etrunc}
\end{equation}
and the coefficient error
\begin{equation}
E_Q=\left| \sum_{n=0}^{p} (\zloc_{n}(\bx) -\zloch_{n}(\bx)) |\bx-\bc|^n \right|
\label{eq:Equad}
\end{equation}\\

%\subsection{Error in the QBX coefficients} \label{coeff_error}
\noindent
\underline{\em Error in the QBX coefficients} \\

Consider first the coefficient  error in (\ref{eq:Equad}).
The error that is introduced by the discretization of the integrals
defining the expansion coefficients in a QBX framework was analyzed by
af Klinteberg and Tornberg in \cite{Klinteberg2016b}.  
%The quadrature
%errors for nearly singular integrals are estimated using a
%technique based on contour integration and and calculus of residues.
%Building on this technique, the quadrature errors associated with QBX
%when applied to the evaluation of layer potentials can be estimated. 
%In \cite{Klinteberg2016b}, 
\tp{  There, both the single layer Laplace and Helmholtz
kernels are considered in two and three dimensions. For the three dimensional
case, error estimates are derived for the single layer Laplace
potential for two different cases, when the surface geometry is that
of a spheroid and for a flat  surface panel that is discretized by a $q
\times q$ Gauss-Legendre quadrature rule. 
%In the estimate for the panel, it is assumed that the panel is flat. 
Assuming a flat panel of size $h \times h$, 
the error estimate that is derived for the single layer potential is 
%\be 
%E_Q^{S,one \ panel}(\bx)  \lesssim |\sigma(\bx)| \frac{h}{q} \sum_{l=0}^{p} 
%\frac{2 \pi^{3/2} (2l)!}{\Gamma(l+1/2)(l!)} \left(
%  \frac{qr}{h}\right)^l e^{-4 q r_P/h}
%\ee
%where $\Gamma$ is the Gamma function. 
\be \label{eq:coeff_error}
E_Q^{S,one \ panel}(\bx)  \lesssim |\sigma(\bx)| \frac{h}{q} \sum_{l=0}^{p} 
\frac{2 \pi^{3/2} (2l)!}{\Gamma(l+1/2)(l!)^2} \left(
  \frac{qr}{h}\right)^l e^{-4 q r_P/h}
\ee
where $\Gamma$ is the Gamma function.
The error from the closest panel dominates, and $r_P$ is the
closest distance between the expansion center and that panel; we also  recall that $r$ is the distance from
the evaluation point $\bx$ to the expansion center.  
%(Notation $r$ for $r_x$ and $R$ for $r_y$ from what has been used
%before. Is that what we want to use?)
 The notation $a(q) \lesssim b(q)$ denotes ``approximately less than or
equal to'' 
in the sense that  there exists a  $K(q)$  such that  $a(q) \leq K(q)$ and 
$\lim_{q \rightarrow \infty} K(q)/b(q) = O(1)$.
The estimate has been derived assuming that $q$ is large, but in
practice it also works well for moderate $q$.  The corresponding estimate for the double layer potential remains to be derived, but crucially, it is expected to contain the same exponential term. We make essential  use of this below. \\
% as long as $q>2p$. 
%%\textcolor{red}{Should it not depend on h too? }
}

%The estimate for each panel has the same form and can be summed
%together for the total error. The value of $r_P$ is the closest distance between the
%center point of the expansion and the panel, meaning that error from
%the closest panel will dominate.  Using Stirling's formula and
%discarding the term for $l=0$, the estimate can be written in the
%simpler form \cite{Klinteberg2016b},
%\be \label{eq:coeff_error}
%E_Q^{S,one \ panel}(\bx)  \lesssim C \frac{h}{q} \sum_{l=1}^{p} 
%\frac{1}{\sqrt{l}} \left(
%  \frac{4qre}{hl}\right)^l e^{-4 q r_P/h} \| \sigma \|_{\infty}
%\ee
%This result has the same generic form that is found also for 2D QBX in the case of  the
%single layer Laplace and Helmholtz kernels, when a Gauss-Legendre
%quadrature rule is applied to segments of the boundary curve.
%The result for a flat panel can be extended to a panel with
%curvature by mapping it to a flat panel, the evaluation
%point with it, before determining $r$ and $r_P$.\\
%In this paper we are considering the evaluation of the double layer
%potential. This yields essentially the same result as for the single
%layer potential above, with a power of $l+1$ for the term $(r/l)$ inside the sum. 

%\subsection{Truncation error}
\noindent
\underline{\em Truncation error}\\

The truncation error was analyzed by Epstein et
al.\ \cite{Epstein2013} assuming a  ``global'' QBX approach. 
Here, we  consider the truncation error for the {\it  local} QBX evaluation in
(\ref{eq:Etrunc}).  The  result is presented  for the double layer potential. In Appendix B, we give a detailed derivation of the 
truncation error for the simpler case of the single layer potential. The derivation for the double layer potential is similar, and is left for the reader. 

Assume that the local correction $\dlayer_L \sigma(\bx)$ to the double layer potential given in (\ref{eq:qbx_loc_exp}), (\ref{eq:qbx_zloc})  involves integration over a smooth surface patch $\setofpanels$.  Normally $\setofpanels$ is a set of surface panels, but for simplicity we assume here that $\setofpanels$ has a smooth boundary, as shown in Figure \ref{fig_trunc}. 
We  assume that $\bar{\bx}$ is the point on the surface $\setofpanels$ that
is closest to the expansion center $\bc$, that is,   if  $B_c({\bc})$  is a ball of radius $c$ about $\bc$ then $\overline{B_c}(\bc) \cap \setofpanels = \{ \bar{\bx} \}$.  
%The boundary of $\Gamma$ is denoted by $C_\Gamma$ and is assumed to be a circle of radius ${\bar{R}}$. 
The surface patch $\setofpanels$ is further assumed to be such that its projection $R_\Gamma$  onto the tangent plane at $\bar{\bx}$ is a disk of radius $\bar{R}$.
We place the  origin  $O$ of a Cartesian coordinate system at $\bar{\bx}$,  and assume the $x_3$ axis is directed along the line $\bar{\bx}-\bc$, i.e., normal to $\setofpanels$ at $\bar{\bx}$.   We make the  additional  assumption that 
\be \label{scaling_b_R}
\bar{R}^2<<|c|<<\bar{R}<<1. 
\ee
The situation is illustrated in Figure \ref{fig_trunc}.  

%In Appendix C, we  compute
%the truncation error from expanding the Green's function
%$G(\bx,\by)=1/(4 \pi |\bx-\by|))$ in a Taylor's series with respect to
%$\bx$ about the point $\bc=(0,0,c)$, which lies either above or below
%the center of the disk. We further assume that \be \label{scaling_b_R}
%\bar{R}^2<<|c|<<\bar{R}<<1.  \ee 

\begin{figure}
\vspace{0.in}
  \center{\includegraphics[clip=true, viewport=200 180 650 510, scale=.5]{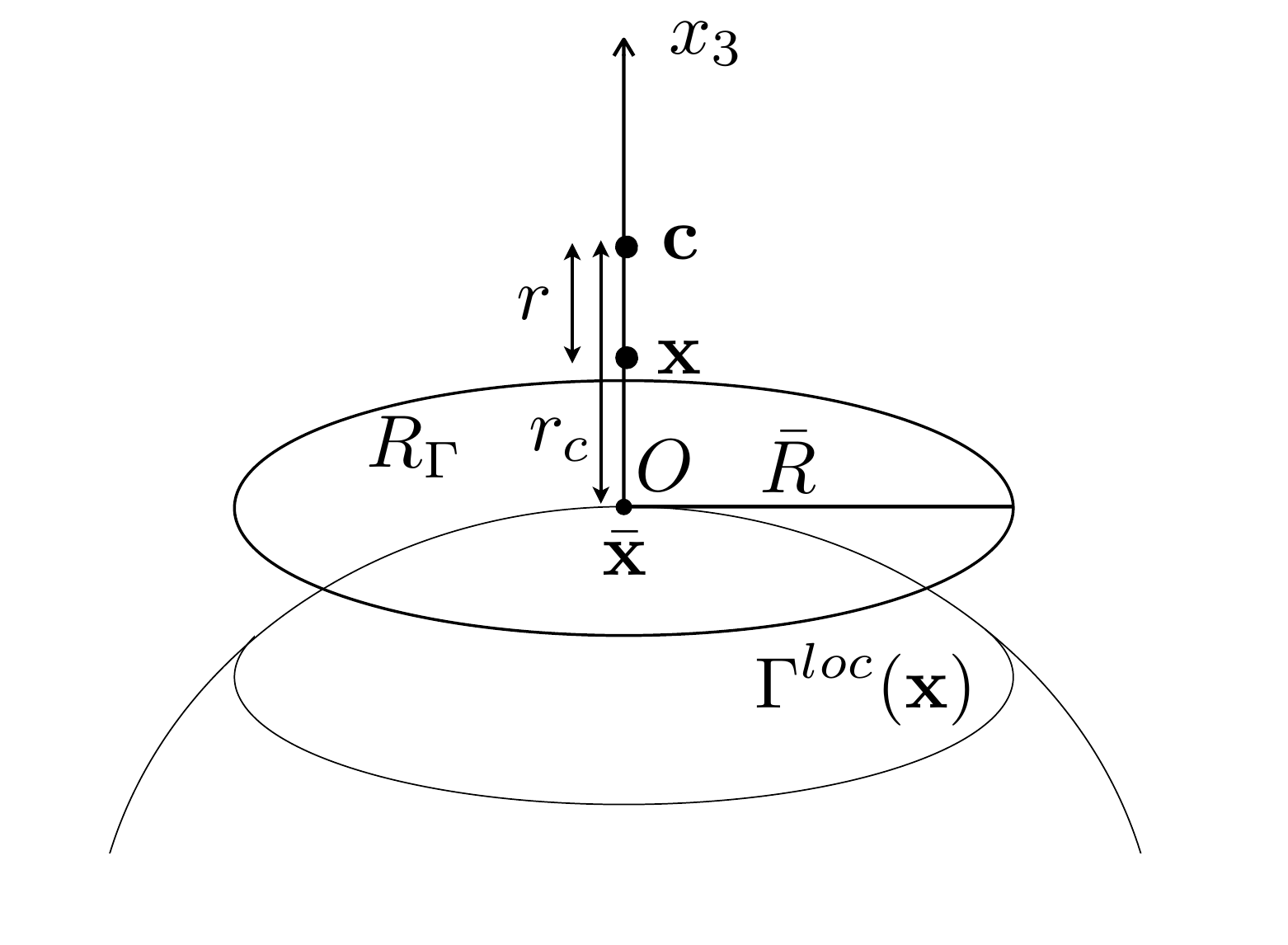}}   
  \caption{Surface patch $\setofpanels$ and its projection $R_\Gamma$ onto the tangent plane at $\bar{\mathbf{x}}$.  The expansion center is $\mathbf{c}$, target point is $\mathbf{x}$, and radius of $R_\Gamma$ is $\bar{R}$.   } 
  \label{fig_trunc} 
\end{figure}

Our  estimate for the truncation error of the double layer potential is the following:  
\begin{theorem} \label{DL_error_nonplanar}
%Let  $\setofpanels$ be a smooth surface and $\bar{\bx}$ a point on the surface such that the projection $R_\Gamma$ of $\Gamma$ onto the tangent plane at $\bar{\bx}$ is a disk of radius ${\bar{R}}$.  
Let  $E_T$ be the truncation error  of the local double layer potential $\dlayer_L \sigma (\bx)$ 
evaluated  by Taylor's expansion of order $p$ about the point $\bc=(0,0,c)$.  Then under the above assumptions,  $E_T$ evaluated at a target point $\bx=(0,0,x_3)$ inside the radius of convergence of the Taylor's series satisfies the bound
\begin{multline}
\label{eq:terror_dbllayer}
E_T \leq   C \ p \ \alpha_{p+1} \left| \sigma({\bf 0}) \right| \frac{\left[(1+\sqrt{2}) \ r \right]^{p+1}}{ \left( \sqrt{c^2+{\bar{R}}^2} \right)^{p+1}}   \left( 1 + O \left( \frac{c^2}{c^2+{\bar{R}}^2} \right) \right)\\
+ O \left( \alpha_{p}  \ p \ |\sigma({\bf 0})| \  {\cal H}  {\bar{R}} \frac{ r^{p+1}}{ (\sqrt{c^2+{\bar{R}}^2})^{p+1}} \right),
\end{multline}
where $r=|x_3-c|$, ${\cal H}$ is the mean curvature of $\setofpanels$ at $\bar{\bx}$, $C$ is a constant, and $\alpha_{p}$ is defined in  Lemma
\ref{lemma_planar}, Appendix C. 
\end{theorem}

\noindent
 The error estimates are chiefly used  to assess the order of convergence for different relative orderings of the numerical parameters. We give some examples.\\

%\noindent
%1.  The   coefficient and truncation error  behave rather differently.    In the error expressions (\ref{eq:coeff_error}) and (\ref{eq:terror_dbllayer}), we identify 
%both $|c|$ and the minimum panel-to-center distance $r_P$ with $r_c$ (cf. (\ref{eqn:rc})).   If we place the expansion center a distance $r_c=O(h^\alpha)$ from the surface, where $\alpha>0$,  and fix the other parameters,   then the truncation error is $O(h^{\alpha(p+1)})$.  Increasing  $\alpha$ improves this error.  On the other hand,  the coefficient error  is high order only if $4 q r_c/h>>1$, or $r_c>>h/4q$. This favors small $\alpha$,  otherwise, the expansion center is too close to the surface relative to the discretization and accuracy is lost. As a compromise we can, say, set  $r_c=h^{1/2}$, so that  both errors tend to  zero with high order when $p$ and $q$ are large. In this case, the method is classically convergent.\\

\noindent
1.   In the error expressions (\ref{eq:coeff_error}) and (\ref{eq:terror_dbllayer}), we identify 
both $|c|$ and the minimum panel-to-center distance $r_P$ with $r_c$ (cf. (\ref{eqn:rc})), and note that $r \leq r_c$.   If we place the expansion center a distance $r_c=O(h^{1/2})$ from the surface and fix $\bar{R}$,   then the truncation error is at most $O(h^{(p+1)/2})$.   The dominant term in the coefficient error is $e^{-4qr_c/h}$ which is $O(e^{-1/h^{1/2}})$. The third source of error, i.e., the interpolation error from upsampling, is $O(h^{q_{int}})$.  All the errors tend to  zero, and high order can be achieved by choice of $p$ and $q_{int}$  (cf. Table \ref{table1}, \S  \ref{sec:results}). In this case, the method is classically convergent.\\

\noindent
 2.  We may also desire that  the local QBX correction take  $O(1)$ work per target as the grid is refined, so that its overall complexity is  $O(N)$.  This can be achieved by setting $\bar{R}=O(h)$.    (Later, we will identify $\bar{R}$ with a numerical `distance' parameter $d_{QBX}$.) The error estimates require that we also set $r_c=O(h)$, in which case the coefficient and truncation errors remain fixed, but controllably small,  as $N$ increases. The total error then has the form
$E=O(\epsilon + h^{q_{int}})$, where $\epsilon$ is the sum of the coefficient and truncation error, and $h^{q_{int}}$ is the interpolation error.
In this case, our QBX method is not classically convergent, but converges with controlled precision   (cf. Table \ref{table2}, \S \ref{sec:results}). This type of error is also discussed in \cite{Klockner}.    \\

\noindent
3.  Numerical parameters for most of the  computational examples in \S \ref{sec:results}  
are chosen so that the quadrature  and
truncation errors are dominated by the  $O(h^{q_{int}})$ interpolation error. We make the specific choice $q_{int}=7$.

%\noindent
%3.  Although the truncation error estimates have been derived assuming the scaling (\ref{}),  they are expected to hold for $|c|/{\bar{R}}$ sufficiently small that    the expansions in Appendix D are  valid. It is not necessary for  
% $|c|/{\bar{R}}$  to tend to zero as ${\bar{R}} \rightarrow 0$, and in our numerical tests, the parameter scaling is such that $|c|/\bar{R}$ tends to a constant in this limit. 

\section{Local QBX algorithm} \label{sect:local_qbx_alg}

\subsection{On-surface evaluation}
\label{sect:onsurface}

We describe in more detail our local QBX computation of the double layer potential in (\ref{eq:qbx_loc_exp}), (\ref{eq:qbx_zloc}).  The process is different for target points on the surface, which makes use of a precomputation, then for target points that are nearby but off the surface, for which the QBX correction  is computed `on the fly.' We begin with the on-surface evaluation.  
We want to compute
\begin{equation}
u_L(\bx_i)=\dlayloch \sigma (\bx_i), \quad i=1,\ldots, N
\label{eq:DL_h_bxi}
\end{equation}
where $\dlayloc$ was defined in (\ref{eq:qbx_loc_exp}), and the
superscript $h$ indicates that this is a numerically computed
approximation. The target points $\bx_i \in \partial D$ are the points for which we
are enforcing (\ref{eq:interior_discrete}),(\ref{eq:exterior_discrete}). 

Let $\bx_i \in \partial D_k$, and assume for now that 
$\setofpanels$ includes a set of neighboring
surface panels on $\partial D_k$, but no other parts of $\partial
D$.
We will call this set of panels the {\em local patch}.  Denote the
total number of discretization points within the local patch by
$\nptsx$, their locations by $\bx_q, \, q=1,\ldots n_{\bx_i}$, and let
$\mathbf{\Sigma}_i \in \reals^{\nptsx}$ contain
$\{ \sigma(\bx_q)\}_{q=1}^{\nptsx}$.

Given an expansion center $\bc_i$ and an expansion order $p$, 
we can  precompute a vector $R_i \in \reals^{\nptsx}$ such that
\begin{equation}
u_L(\bx_i)=R_i \cdot \mathbf{\Sigma}_i.
\label{eq:target_spec_quad}
\end{equation}
The vector $R_i$ represents the resulting action after 
$i)$ upsampling the density $\sigma$ to a finer grid locally on each surface panel
(upsampling factor $\kappa$), 
$ii)$ using the fine grid to compute the expansion coefficients
(\ref{eq:qbx_zloc}) at the expansion center $\bc_i$ up to order $p$,
including only the surface panels in the local patch, and 
$iii)$ finally evaluating the expansion (\ref{eq:qbx_loc_exp}) (summing
from $n=0$ to $p$) at $\bx_i$.

The numbers in $R_i$ are the effective quadrature weights. They are
however {\em target specific}, i.e.\ for each $\bx_i$ we will in
general find a different set of values.  If the surface $\partial D_k$
is e.g.\ axisymmetric, or if there are multiple $D_k$'s with the same surface shape, then several target points can use the same $R_i$
values. Precomputation was also used in \cite{Klinteberg2016} for simulations of Stokes flow with
axisymmetric spheroidal particles.

By scaling $n_{\bx_i}$  to have $O(1)$ size as $N$ is increased, the dot product in (\ref{eq:target_spec_quad}) has $O(1)$ complexity per target, or a total complexity of $O(N)$  for the evaluation of (\ref{eq:target_spec_quad}) over all the targets.

%%\textcolor{red}{What to say about typical value for $\nptsx$?} 

%\textcolor{red}{See comment in intro about complexity. Some remarks
%  should be added here too} 

The on-surface target points are the discretization points in the
regular quadrature for $\partial D_k$, $k=1,\ldots,M$. Hence, we know
the location of the target points and we precompute and store the
vectors $R_i$ for all $\bx_i$ on $\partial D$, $i=1,\ldots,N$. They
can then be reused in each GMRES iteration.

%\textcolor{red}{Two-sided and one-sided limit?} 

By construction, QBX  evaluates the one-sided limit of a layer potential as $\bx$ approaches the boundary on the side of the expansion center. If this one-sided limit is the quantity of interest, then no further post-processing is needed. However, if the  
%principal value 
integral ${\cal D} \sigma (\bx)$ is desired for $\bx$ on the boundary $\partial D$, then 
%in view of the jump condition (\ref{double_layer_jump}) 
additional steps are necessary. One option is to add/subtract the relevant quantity from the one-sided limit, as in (\ref{double_layer_jump}). A second option is to compute both one-sided limits using QBX and average, i.e.,
\[
{\cal D}\sigma(\bx) = \frac{1}{2} \left( \lim_{\bx \rightarrow \partial D_+} 
{\cal D}\sigma(\bx)  +  \lim_{\bx \rightarrow \partial D_- }
{\cal D}\sigma(\bx)  \right)
\]
by two applications of QBX. The advantages and disadvantages of these approaches are discussed at length in  \cite{Klockner}. Both options were investigated in  the numerical examples reported in Section  \ref{sec:results}, with little or no difference in the results for the specified error tolerances.

\subsection{Off-surface target points}
\label{sect:offsurface}
%For the QBX treatment, it does not matter if the integral is singular
%or nearly singular, the same expansion coefficients must be computed,
%and the local expansion must be evaluated at the target point.  For
%the on-surface target points (the singular case), the location of the
%target points are known, and we precompute the target specific
%quadrature weights that are obtained from the QBX procedure, as was
%discussed in the previous section.

Off-surface target points that are close to the surface (the nearly
singular case) occur because another surface is close by as we are
solving the integral equation for the exterior problem 
(\ref{modified_second_kind}), or in the post-processing step when
we want to compute the solution anywhere in the domain. 
If the locations of these target points are not known before hand, 
target specific quadrature weights cannot be precomputed. 
In addition, there is a limit to how many precomputed numbers are
practical to store. 

Hence, these computations are done on the fly, as they are determined
to be necessary.
For each target point $\bx_i$, it needs to be determined what panels 
should be included in $\setofpanels$. This is done
efficiently utilizing a tree structure. 
The density  $\sigma$ must then be upsampled to a finer grid locally
on each panel such that expansion coefficients can be accurately
computed. Here we use an adaptive strategy, where the
distance from the patch to the target point will determine the 
upsampling factor $\kappa$. This will be further discussed in next
section, where the full algorithm is discussed.

%Hence, these computations are done on the fly, as they are determined
%to be necessary.
%For each target point $\bx_i$, it needs to be determined what panels 
%should be included in $\setofpanels$. This is done
%efficiently utilizing a tree structure. 
%If $\bx_i$ is a discretization point on $\partial D_k$, the target
%specific weights for the contribution
%from the panels on that same surface have already been precomputed,
%and the contribution in (\ref{eq:DL_h_bxi}) can be rapidly evaluated
%from (\ref{eq:target_spec_quad}). 
%For all other panels (or all panels in the set $\Gamma^{loc}(\bx)$  if $\bx_i$ is not a discretizati%on point on
%any $\partial D_k$), the computations must be done as described in
%$(ii)$-$(iii)$ of the next section.
%Given the location of the center $\bc_i$ and expansion order $p$, the
%first step is to upsample the density $\sigma$ to a finer grid locally
%on each panel. Here we use an adaptive strategy, where the
%distance from the patch to the target point will determine the 
%upsampling factor $\kappa$. This will be further discussed in next
%section, where the full algorithm is discussed. 

\subsection{The full algorithm}
\label{sect:full_alg}

Recall that  each boundary component $\partial D_k$, $k=1,\ldots,M$ is tiled  into $N_\theta^k \times N_\varphi^k$ surface panels,
with $q^2$ Gauss-Legendre points on each, where we use the particular value $q=7$.
% The numerical examples given in section \ref{sec:results} use $q=15$
% for the local QBX correction and $q=7$ for the direct quadrature. 
We solve the second kind integral equation (\ref{eq:interior_discrete})  or (\ref{eq:exterior_discrete}) for the interior or exterior problem, respectively, discretized by a
Nystr\"{o}m method
 with  GMRES to find the discrete values of the double layer density
$\sigma$.  Then for any given point $\bx$ in the solution domain, we
evaluate (\ref{eq:interior_postproc}) or  (\ref{eq:exterior_postproc}) to
obtain the solution at that point.

%Both for the interior and exterior problems, 
When performing a
matrix-vector multiply in a GMRES iteration, or in a post-processing
step to find the solution $u$, the double layer potential $\dlayer^h
\sigma(\bx)$ must be computed. 
In the previous sections, we have described how the direct quadrature
is not sufficient when the evaluation point $\bx$ is either close to a
boundary component, introducing a nearly singular integral, or
actually on a boundary surface (singular integral).  The on-surface
treatment is done with QBX and was discussed in Section
\ref{sect:onsurface}.  The QBX evaluation for off-surface target
points was discussed in Section \ref{sect:offsurface}.  However,
upsampling of panels can be sufficient by itself if the evaluation
(target) point is not too close to the surface.

Denote by $d_P(\bx_i)$ the closest distance from target point $\bx_i$
to panel $P$.  Assume that we are given two distances $\dup$ and
$\dqbx$ such that if $d_P(\bx_i)>\dup$, the direct quadrature on the
original grid is
sufficient, if $\dqbx < d_P(\bx_i) \le \dup$ we want to use upsampling
of the panel, and for $d_P(\bx_i) \le \dqbx$ we will use the QBX
treatment. Figure \ref{fig_localQBX} provides an illustration.

The algorithm for evaluating $\dlayer^h \sigma(\bx)$ then starts by 
evaluating $\dlayer^h \sigma(\bx)$ with the direct (coarse grid) quadrature
  at all target points using a fast hierarchical algorithm (here we use the treecode algorithm in \cite{Krasny}).  \tb{ The incorporation of an FMM or treecode is therefore entirely standard, and a `black box' algorithm can be used. }
However, this quadrature will not be accurate for panels that are close to  each target point $\bx_i$, and a local correction must then be done. 
We start by finding all panels P with $d_P(\bx_i)\le \dup$ (using a
tree structure), and for each panel in the set we subtract off the
contribution based on the direct quadrature over that panel,
$Q^P_q[\sigma K(\bx_i,.)]$, and subsequently add on a more accurate quadrature as described below.

\tb{Then as depicted in Figure \ref{fig_localQBX}: }
\begin{itemize}
%\item[i)] For each panel in the set subtract off the contribution
%  based on the direct quadrature over that panel, $Q^P_q[\sigma
%  K(\bx_i,.)]$. 
\item[i)] If $\bx_i \in \partial D_k$ for some $k$, evaluate the QBX
  contribution from all panels in the local patch $\Gamma^{loc}(\bx_i)$ that are on the same surface component as $\bx_i$,
  using the precomputed target specific weights (\ref{eq:target_spec_quad}).
\item[ii)] For all other panels 
%in  $\Gamma^{loc}(\bx_i)$ 
such that $d_{QBX}<d_P \leq \dup $, add
  upsampled quadrature contribution $Q^P_{q,\kappa}[\sigma
  K(\bx_i,.)]$ (see (\ref{eqn:Q_P_kappa})). 
\item[iii)] For all panels in  $\Gamma^{loc}(\bx_i)$ not yet included, pick a center $\bc_i$
  associated with $\bx_i$ and evaluate the contributions
   to the QBX coefficients $\zloch_n(\bx_i)$, $n=0, \ldots,p$, using
   upsampled quadrature (here we use with $q=15$). Form the sum  (\ref{eq:qbx_loc_exp_h}) to  evaluate the QBX expansion 
  at $\bx_i$.
\end{itemize}

\begin{figure}
\vspace{0.in}
  \center{\includegraphics[clip=true, viewport=40 100 510 298, scale=.75]{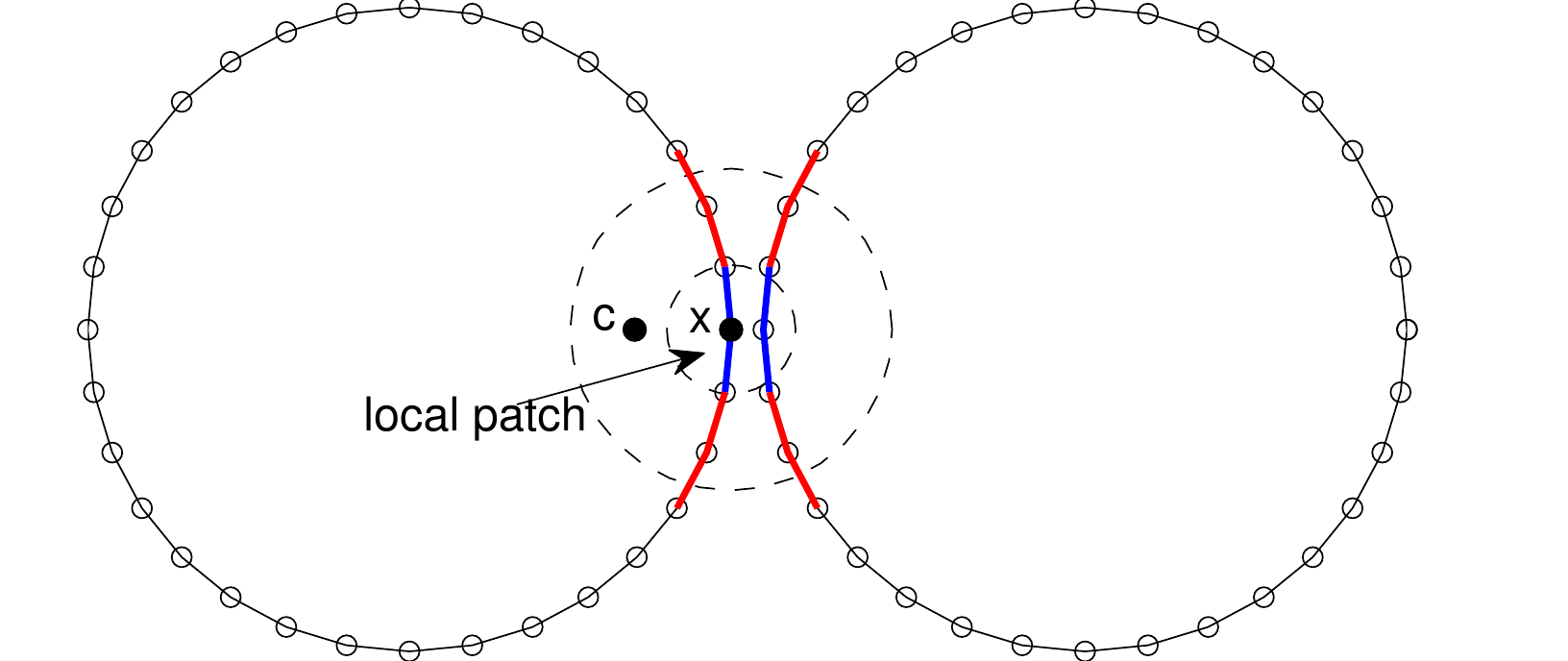}}  
  \caption{\tb{Computation of the local QBX correction at  $\bx$ using center $\bc$.    
${\cal D}_L^h \sigma(\bx)$ is computed over blue panels  using QBX, and over red panels by direct upsampled quadrature. Spheres of radius $d_{up}$ and $d_{QBX}$ (with $d_{up}>d_{QBX}$) are shown.}  } 
  \label{fig_localQBX} 
\end{figure}

There are several parameters that need to be set. 
Assume that the underlying discretization has been set, with
a total of $N_p$ panels and $q^2$ quadrature points on
each, such that  $N=N_P \times q^2$. 
If the error saturates as we increase $\dup$, $\dqbx$, 
upsampling factors and expansion order $p$ for our quadrature method,
we are seeing the error of the underlying discretization. 
Hence, if a smaller error is needed, $N$ must then be increased. 

Now assume that we have a fine enough underlying discretization, and
that we wish to set a tolerance for the relative quadrature
error, and select the parameters for the quadrature method thereafter. 
Hence, we need to set $\dup$, $\dqbx$, $p$, $r_c$ 
(or a way to select the center $\bc_i$ given a target $\bx_i$), and we
also need to set the upsampling rates, both for the direct upsampling
as described below, and for the quadrature when computing the QBX
coefficients. 

%%Upsampling rates. $\kappa$. 

At this point, we do not have explicit expressions for exactly how to set
these parameters, although   the estimates for the truncation and coefficient errors given in Section \ref{sect:qbxerrs}  guide our choice. 
In those error expressions (see (\ref{eq:coeff_error}) and
(\ref{eq:terror_dbllayer})), we identify $\bar{R}$ with the parameter
$d_{QBX}$, $r$ with the center-to-target distance $|c_i-\bx_i|$, and as noted earlier
both $|c|$ and the (minimum) panel-to-center distance $r_P$  with $r_c$ (cf. (\ref{eqn:rc})).  We typically  choose  $d_{QBX}$ and the center position $\bc_i$ as described below so that     
%$d_p$ and the center-to-target distance to be 
$d_{QBX}$, $r$, and $r_c$ are
proportional to  the panel size $h$.  As discussed in Example 3, Section \ref{sect:qbxerrs}, this has the effect of  fixing the coefficient and truncation error at a small value as $h$ is reduced. %Furthermore, we pick $p$ large enough so that the truncation and coefficient error are `small,'  or more precisely, not the dominant source of error in the computation.  
Crucially, this scaling of parameters with $h$
also leads to the desired $O(1)$ work per target in the local correction step.

%We have estimates for the truncation and coefficient
%errors in QBX, but although they give us decay rates, they
%either contains unknown constants (as the truncation error estimate)
%or become more accurate as the order of the Gauss-Legendre rule is
%increased. 
%\textcolor{red}{Say something about the estimates that we have avaliable}

 In the next section, we introduce the evaluation of the double layer
potential over the unit sphere with the density $\sigma$ set as a
spherical harmonic function $Y_n^m$. For this case, we have an analytical
solution both for on-surface and  off-surface target points. We will
use this example to set our parameters. 

We will first discuss how to set the parameters for the on-surface
evaluation (Section \ref{sect:onsurface}). In $i)$ above, we have
assumed that the target specific weights resulting from the QBX
procedure have been precomputed, as we would do when we want to solve
an integral equation. They could also of course be computed directly
with the same parameter choices. 

Now, let $\hpanel$ denote a typical panel dimension.
%($\hpanel \times \hpanel$).  
Given a collection of on-surface target
points $\bx_i$, $i=1,\ldots,N_{set}$, 
define the local patch for each target by the set of the surface panels for which the
closest distance from $\bx_i$ to the panel is less than $\hpanel$.
Then do the following: 
\begin{itemize}
\item[i)]  Set the expansion radius $r_c$ no larger than $\hpanel/2$. The
  center $\bc_i$ will then be set at a distance $r_c$ from the surface
  normally out from $\bx_i$.  Note that for the on-surface evaluation, $r=r_c$. We also set $d_{QBX}=r_c$, or a small multiple of it.
\item[ii)] Set a large upsampling ratio $\kappa$ for all panels in the
  local patch, and increase $p$ until the desired accuracy is
  reached for all targets. That sets the value of $p$. 
\item[iii)] Now, with $r_c$ and $p$ selected, start to reduce $\kappa$. 
Pick the smallest value possible before it adversely starts to affect
your accuracy. 
\item[iv)] {\it Adaptive reduction of $\kappa$}. The minimum distance from
  the closest panel to the expansion center is $r_c$. Assume that the
  closest distance from panel $P$ is $r_P$. According to the error
  estimate for the coefficient error, the contribution of the error
  from panel $P$ will be comparable to the one from the closest panel
  if $\kappa_P r_P=\kappa r_c$, where $\kappa_P$ is the upsampling ratio for panel $P$.  Hence, depending on the size of $r_P$,
  $\kappa$ could possibly be reduced (always to an integer value).  This type of adaptive reduction is described in \cite{Klockner}, and  a similar version has been implemented here for the off-surface QBX  evaluation. 
\end{itemize}

%Given the above choice of $r$, we will set $\dqbx=r$.  
We will now
discuss the off-surface evaluation, starting with the standard
upsampling.  It is easy to test for off-surface points and find the
limit $\dup$ up to which distance, but no closer, the direct
quadrature is sufficiently accurate. 
According to the error estimate (\ref{eq:coeff_error}), if the upsampling ratio is doubled,
the same accuracy can be retained for targets a distance $d_{up}/2$ from the surface. 
This gives a recipe to choose $\kappa$ at
different distances. To keep it simple, we however take $\kappa$ in
the whole upsampling region  to be the $\kappa$ needed at the distance
$\dqbx$. If the distance is smaller than $\dqbx$, we will switch to
the QBX evaluation.

The off-surface QBX evaluation will use the value of $r_c$ and the
number of expansion terms $p$ selected in the on-surface procedure. 
Here however, if a target point is at a distance $\tilde{r}_c$ from the surface, the
center will be placed in the (approximately) normal direction,
another distance $r_c$ away, yielding a distance from the center to the
surface of approximately $\tilde{r}_c+r_c$. 
The integrals to evaluate for  the QBX coefficients will hence typically
require less resolution than for the on-surface evaluation (since they are not as
nearly singular), but for target points extremely close to the
surface, it will be essentially the same, and we will keep the same
parameters as for the on-surface evaluation.

%\begin{itemize}
%\item[a.]	Global regular quadrature. Tree code. Upsampling? Dependence on xi parameter?
%\item[b.]	Local quadrature with upsampling or QBX. Different
%  zones, how to determine when to do what. How to pick center
%  locations and $p$ and other parameters. 
%\item[c.]	The full algorithm in schematic form. 
%\end{itemize}
\section{Validation/Numerical results}
\label{sec:results}

In this section, we illustrate the performance of the target specific QBX method in several examples. Table \ref{table:params} reviews the main numerical parameters. Other than $d_{up}$, which is fixed at $2 \cdot d_{QBX}$, 
the specific values of the parameters are given in each example below.
\begin{table}[htp]
\begin{center}
\begin{tabular}{| l | l | l | }
%\hline
%  Parameter  &  Definition \\
\hline
  $N_\theta(=N_\phi)$ & no.\ panels on each $\partial D_k$ is $N_\theta \times N_\phi$ & \S   \ref{sect:direct_quad} \\
%\hline
  $d_{QBX}$ & QBX criteria $0<d_P(\bx) \leq d_{QBX}$ &   \S \ref{sect:full_alg} \\ 
%\hline
  $d_{up} (=2  d_{QBX})$ & direct upsampling criteria $d_{QBX}<d_P(\bx) \leq d_{up}$ & \S \ref{sect:full_alg}\\
%\hline 
  $\kappa$  & upsampling factor & \S \ref{sect:upsamp} \\
%\hline
  $r_c$  & minimum distance from expansion  center to surface & \S \ref{sect:localexp}, eq. (\ref{eqn:rc})\\
%\hline
  $p$  & truncation level & \S \ref{sect:localQBX}, eq. (\ref{eq:qbx_loc_exp_h})\\
  \hline  
\end{tabular}
\end{center}
\caption{Main numerical parameters, and the section where each is introduced. Here $d_P(\bx)$ is the minimum distance
from the target point $\bx$ to panel $P$. 
 }
\label{table:params}
\end{table}

\begin{figure}
\vspace{-.68in}
\center{\includegraphics[clip=true, viewport=40 38 475 360, width=.75\textwidth]{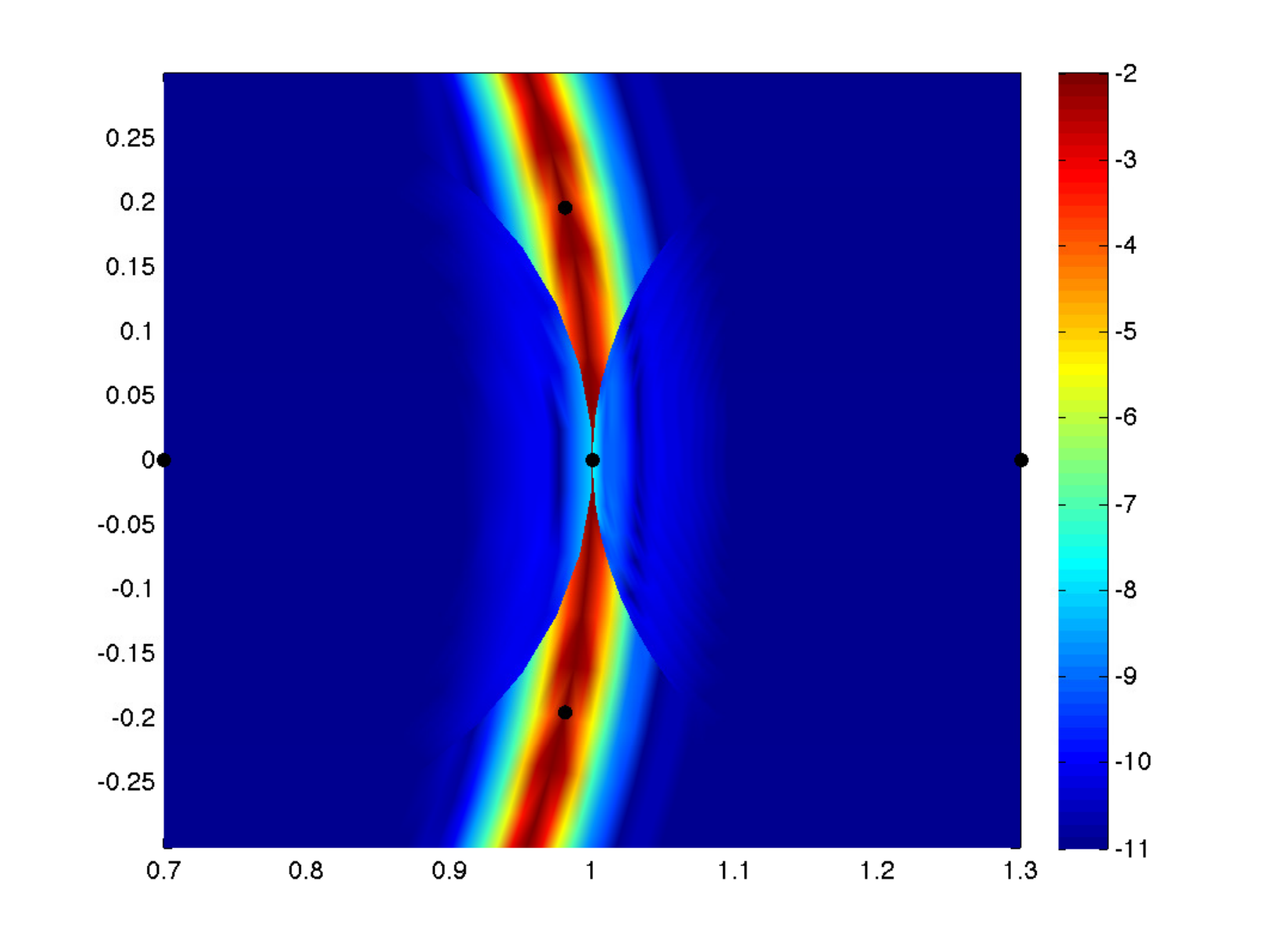}}
%  \center{\includegraphics[width=.75\textwidth]{Figures/fig1.eps}}   
  \caption{ $l^\infty$ error in the potential ${\cal D}Y_2^2$ computed  using 15 point Gauss-Legendre quadrature over $32^2$ panels, except in two small spheres $|\bx-\bc_\pm| \leq r_c$ centered at  off-surface points $\bc_\pm$ on either side of $\partial D$  a distance $r_c=0.3$ from the boundary. The potential in the spheres is computed using our local QBX method with parameters $N_\theta=16, \kappa=2,$ and $d_{QBX}=0.7$  Only a small portion of the boundary $\partial D$ is shown. The markers indicate the panel size and locations of $\bc^\pm$.    } 
  \label{fig0} 
\end{figure}

%In addition to figures in this section, we should perhaps have an illustrative color plot like Figure 3 in the Klockner et al  (2013) QBX paper? This would come earlier in our paper.

\subsection{Evaluation of layer potentials}

%A.	Eigenfunctions for one single sphere. On and off-surface evaluations. 
%\begin{itemize}
%\item[a.]	Checks the evaluations of the double layer potential. 
%\item[b.]	Here, should comment on parameter choices, trade off when it comes to accuracy and speed. Expansion radius, expansion order and upsampling. The role of xi. 
%\end{itemize}

We first validate our numerical method in calculations involving a single sphere
using exact analytical formulae for eigenfunctions of the double layer potential and separation of variable solutions for the interior and exterior Dirichlet problems. For $D$ a unit sphere, it is straightforward to show that 
\begin{equation} \label{eigenfunction}
 \int_{\partial D} Y_l^m(\theta', \phi')  \frac{\partial G(\bx(\theta, \phi), \by(\theta', \phi')) }{\partial \nu(\by (\theta', \phi'))} W(\theta',\phi') ~ d \theta' d \phi' = - \frac{Y_l^m(\theta, \phi)}{4l+2},
\end{equation}
i.e., the spherical harmonic function $Y_l^m$ is an eigenfunction of the double layer potential with eigenvalue $-(4l+2)^{-1}$. Furthermore, substituting  
$\sigma(\theta, \phi)=Y_l^m(\theta, \phi)$  into (\ref{double_layer_jump})   and comparing with the separation of variables solution in spherical coordinates, we see that 
\begin{equation} \label{exact_solution_int_ext}
u(\rho, \theta, \phi) = {\cal D} Y_l^m(\rho,\theta, \phi) =  \left\{
\begin{array}{ll}
 & -\frac{l+1}{2l+1} \ \rho^l Y_l^m(\theta, \phi)~~\mbox{for}~~ |\rho|<1, \\
 &\frac{l}{2l+1} \ \rho^{-(l+1)} Y_l^m(\theta, \phi) ~~\mbox{for}~~|\rho|>1,
\end{array}
\right.
\end{equation}
is the solution to the interior and exterior Dirichlet problems  with boundary data given by the respective solution in
(\ref{exact_solution_int_ext}) evaluated at  $\rho=1$. 

%\textcolor{red}{$r_c$ is not defined anymore since we have no POU
% fcn. Needs to be translated to another measure.}

Figure \ref{fig0} provides an illustrative example of our  method, by
superimposing the errors from a QBX calculation of the double layer
potential in an interior and exterior  spherical region  on top of an
error plot from a computation of the potential using a standard Gauss-Legendre quadrature.   The error in ${\cal D} \sigma$ computed by the QBX expansion is much smaller than the error of the standard computation throughout both  spherical regions. These two regions are where the  expansions about the centers $\bc_\pm$ converge; it is here that the QBX method corrects the inaccuracies of the standard quadrature.

Our next example  (Table \ref{table1}) relates the error estimates in Section \ref{sect:qbxerrs} to numerical results. 
We compute the double layer
potential  on the surface of the unit sphere using the eigenfunction
density  $\sigma = Y_2^2$, and compare with the analytical result (\ref{eigenfunction}).  %We vary the number of target panels and
%keep everything else fixed. 
\tp{ In the top two sets of entries, for $p=3$ and $7$,  the distance  of the expansion center  from the interface is scaled as  $r_c=O(h^{1/2}$). This is the scaling discussed in Example 1 of Section \ref{sect:qbxerrs}, for which the sum of the coefficient and truncation errors  is $O(h^{(p+1)/2})$.  Thus, the analysis predicts a second order method for $p=3$ and a fourth order method for $p=7$, and this is approximately observed in practice. 
%The computed order of convergence is roughly consistent with this estimate.  
The third set of entries is a high accuracy computation  with $p=20$ and fixed $r_c$.
%The table shows rapid convergence of the
%relative error, determined by comparison with the analytical result
%(\ref{eigenfunction}).  
In this example,  the
truncation and coefficient errors are small enough  that the the total
error is dominated by the  interpolation of the density, which is $O(h^{q_{int}})$ with $q_{int}=7$. 
The numerical results are roughly consistent with this expected order of accuracy.  Here and below, $h$ refers to the length  of one side of a Gauss-Legendre panel in $(\theta, \phi)$; more precisely, the panel dimensions are $h \times 2h$ since we use the same number of panels in $\theta$ and $\phi$. In all subsequent examples, parameters are chosen so that the $O(h^{q_{int}})$ interpolation error is the dominant source of error. } 

%\begin{table}[htbp]
%\begin{center}
%\begin{tabular}{| l | l | l | l |}
%\hline
% $N_\theta(=N_\phi)$ & $l^2$ error & $l^\infty$ error & EOC \\
%\hline
% $4$ & $1.8  \times 10^{-5}$  & $2.8  \times 10^{-5}$ & $-$ \\
 %$8$ & $4.3 \times 10^{-7}$  &  $5.7 \times 10^{-7}$  & $5.39$\\
 % $16$ & $5.8 \times 10^{-9}$ & $8.0 \times 10^{-9}$  & $6.21$ \\ 
 %$32$ & $1.5 \times 10^{-11}$ & $2.5 \times 10^{-11}$ & $8.59$ \\
 % \hline  
%\end{tabular}
%\end{center}
%\caption{Relative $l^2$ and $l^\infty$ error in the computation of the
% on-surface double layer potential ${\cal D} Y_2^2$.  EOC is the
%  empirical order of convergence calculated using the $l^2$ error at
%  the previous level of grid refinement.  Parameter values are
%  $\kappa=8, p=20, d_{QBX}=0.7,$ and  $r_c=r=0.2$.   
% }
%\label{table1}
%\end{table} 

\begin{table}[htbp]
\begin{center}
\begin{tabular}{| l | l | l | l | l | l |}
\hline
$p$ &  $N_\theta(=N_\phi)$ & $r_c$ & $l^2$ error & $l^\infty$ error & EOC \\
\hline
$3$ &  $4$ & $0.2$ & $3.1  \times 10^{-2}$  & $4.8  \times 10^{-2}$ & $-$ \\
     &  $8$ & $0.15$ & $1.1 \times 10^{-2}$  &  $1.2 \times 10^{-2}$  & $2.0$\\
     &  $16$ & $0.1$ & $1.9 \times 10^{-3}$ & $1.5 \times 10^{-3}$  & $3.0$ \\ 
    &  $32$ & $0.075$ & $2.1 \times 10^{-4}$ & $3.0 \times 10^{-4}$ & $2.2$ \\
  \hline
$7$ &  $4$ & $0.2$ & $1.1  \times 10^{-2}$  & $1.2  \times 10^{-2}$ & $-$ \\
     &  $8$ & $0.15$ & $3.2 \times 10^{-4}$  &  $3.9 \times 10^{-4}$  & $4.9$\\
     &  $16$ & $0.1$ & $2.3 \times 10^{-5}$ & $1.6 \times 10^{-5}$  & $4.6$ \\
     &  $32$ & $0.075$ & $3.8 \times 10^{-7}$ & $5.5 \times 10^{-7}$ & $4.8$ \\
  \hline
$20$ &  $4$ & $0.2$ & $1.8  \times 10^{-5}$  & $2.8  \times 10^{-5}$ & $-$ \\
     &  $8$ & $0.2$ & $4.3 \times 10^{-7}$  &  $5.7 \times 10^{-7}$  & $5.4$\\
     &  $16$ & $0.2$ & $5.8 \times 10^{-9}$ & $8.0 \times 10^{-9}$  & $6.2$ \\ 
     &  $32$ & $0.2$ & $1.5 \times 10^{-11}$ & $2.5 \times 10^{-11}$ & $8.6$ \\
  \hline  
\end{tabular}
\end{center}
\caption{Relative $l^2$ and $l^\infty$ error in the computation of the
  on-surface double layer potential ${\cal D} Y_2^2$.  EOC is the
  empirical order of convergence calculated using the $l^2$ error at
  the previous level of grid refinement.  Other parameter values are
  $\kappa=2, d_{QBX}=0.35$ for $p=3,~7$ and   $\kappa=8, d_{QBX}=0.7$ for $p=20$.   
 }
\label{table1}
\end{table}

Figure \ref{fig1} shows how the relative  $l^\infty$ error varies with $p$ for the
 computation of the on-surface double layer potential ${\cal D} Y_2^2$  in   Table \ref{table1}.  At small values of $p$ the curves for different resolution $N_\theta$  overlap, indicating that the error is dominated by the truncation error.  For larger $p$ the interpolation error dominates, and thus the higher resolution computations are more accurate.  
%For larger $p$ these curves diverge due to the different interpolation and quadrature error at different resolutions. 
The error estimates discussed in Section \ref{sect:qbxerrs} suggest that 
the coefficient error grows with $p$ and can become dominant at sufficiently large $p$ when other parameters are held fixed.
This is observed in the lowest resolution ($N_\theta=4$) curve, but is not seen in this $p$ range for the higher resolution computations.

\begin{figure}
\vspace{-.5in}
  \center{\includegraphics[width=.75\textwidth]{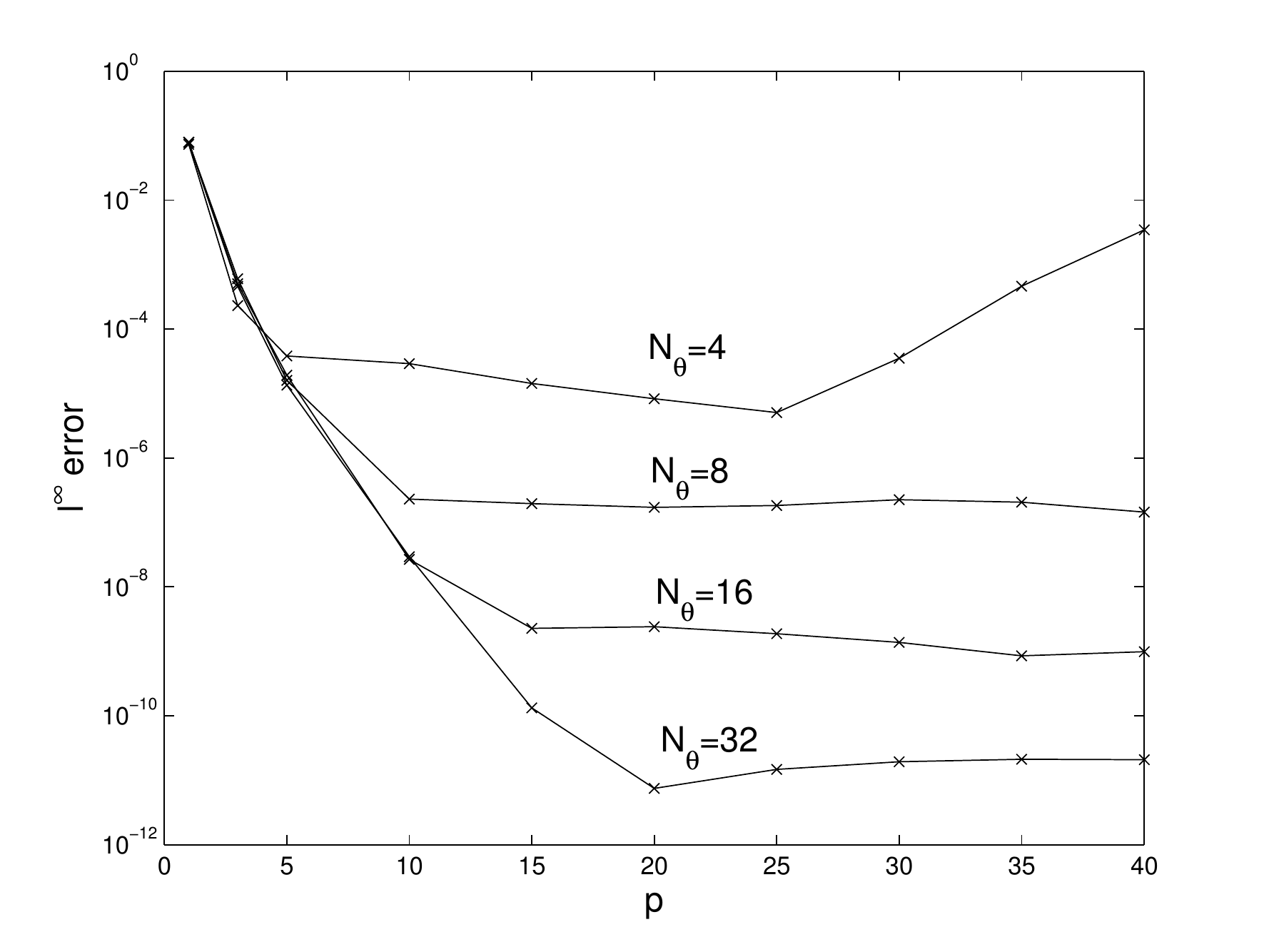}}   
  \caption{Relative $l^\infty$ error versus truncation $p$ in the computation of the on-surface double layer potential ${\cal D} Y_2^2$. Parameter values are $\kappa=8, ~r_c=0.2,~d_{QBX}=0.7$. } 
  \label{fig1} 
\end{figure}

Table \ref{table2} shows the same calculation of the on-surface double
layer potential  ${\cal D}Y_2^2$ as in Table \ref{table1} when $p=20$, but instead
of fixing  the cut-off parameter $d_{QBX}$ 
%and the distance $r_c$  of the centers from the boundary 
(as is done there) we now vary it with
the scaling $d_{QBX} \sim r_c \sim h$ (cf. \S \ref{sect:qbxerrs}).  In this and all subsequent
computations using this scaling, we employ an upsampling layer with
$\dup=2 \ d_{QBX}$ and a factor $\kappa=2$. With these parameter scalings, the truncation and coefficient errors are fixed  as the number of panels is increased.  \tp{ The  parameter values are chosen so that these errors are negligible and the dominant source of error is   the $O(h^7)$ interpolation error. This is  consistent with the  data in Table  \ref{table2}.    }
%Furthermore, this scaling of $d_{QBX}$ makes
%the local QBX correction have a constant or $O(1)$ operation count per target.   
%We find the decrease in total error as the number of panels is raised is similar to that observed in Table \ref{table1}.  

\begin{table}[htbp]
\begin{center}
\begin{tabular}{| l  | l | l | l | l | l | l  | l |}
\hline
  $N_\theta(=N_\phi)$  & $d_{QBX}$ & $r_c$ & $l^2$ error & $l^\infty$
                                                           error &
                                                                   no. ops. &
                                                                              ops. ratio & \tp{EOC}\\
\hline
  $2$ & $1.4$ & $0.4$ &  $5.6  \times 10^{-4}$  & $6.0 \times 10^{-4}$ & $ 3.5 \times 10^4$ &    $~~~~-$  &  $~~-$  \\
 \hline 
  $4$  & $0.7$ & $0.2$ &  $1.3  \times 10^{-5}$  & $1.3  \times 10^{-5}$ & $ 2.0 \times 10^5$ & 5.7 &\tp{ 5.5} \\
 \hline 
  $8$ & $0.35$ & $0.1$ &  $2.7 \times 10^{-7}$  & $1.8  \times 10^{-7}$ & $9.7 \times 10^5$ & 4.9 & \tp{5.6} \\
 \hline 
  $16$ & $0.17$ & $0.05$ &  $2.5 \times 10^{-9}$  & $1.53 \times 10^{-9}$ & $ 4.5 \times 10^6$ & 4.6 & \tp{6.8}  \\
 \hline 
\end{tabular}
\end{center}
\caption{Relative $l^2$ and $l^\infty$
 error  in the computation of the
  on-surface double layer potential ${\cal D} Y_2^2$. The cut-off
  parameter $\dqbx$ and distance of the centers from the boundary $r_c$
  are scaled as $\dqbx \sim r_c \sim h$.   
Parameter values are $\kappa=16$, $p=30$ for the QBX region, and  $\dup=2 \cdot \dqbx$, $\kappa=2$ for the local upsampling region. 
The number of operations in the dot product of the local QBX correction
(\ref{eq:target_spec_quad}) is shown, as is the ratio of two of these values for $N_\theta$ and $2 \cdot N_\theta$  panels.}
\label{table2}
\end{table}

The parameter scaling in Table \ref{table2} has the additional benefit of conferring an $O(N)$ complexity on
the local QBX correction over all targets.   
To verify  this, we show
in the table
the number of operations in the local QBX computation (\ref{eq:target_spec_quad}) over all targets, as well as the ratio of operations for $N_\theta$ and $2 \cdot N_\theta$  panels.   In principle this ratio should be  $4$, since  the number of targets increases by a factor of $4$ when $N_\theta=N_\phi$ is doubled.  The observed ratio is not exactly $4$, due to the  discrete (and nonuniform) panel sizes in physical space, but approaches $4$ asymptotically in $N_\theta$. 

We turn now to computation of the double layer potential at targets
points that are off, but close to, the surface.  Figure \ref{fig2}
show an example calculation of the relative error in the off-surface
computation of the double layer potential ${\cal D} \sigma$ using the
same source density $\sigma=Y_2^2$.  In the figure we evaluate the
potential at a set of targets along the $x$-axis, i.e., at
$\theta= \pi/2$, $\phi=0$ just outside the sphere.  The figure
compares computations using our QBX method with those using a local
upsampling method in which the local correction (i.e., the correction
to the double layer integral for source points close to the target) is
computed directly, without any local expansion, using oversampled
15-point Gauss-Legendre quadrature.  The two methods agree far enough
away from the sphere, but the upsampling method starts to lose
accuracy when the target point is about $4$ to $5$ grid points away
from the boundary (based on the upsampled grid spacing). Sufficiently
close to the boundary the error using upsampling alone is $O(1)$,
regardless of the fineness of the grid.  In contrast, the QBX method
maintains accuracy all the way up to the boundary.

\begin{figure}
\vspace{-.75in}
  \center{\includegraphics[width=.75\textwidth]{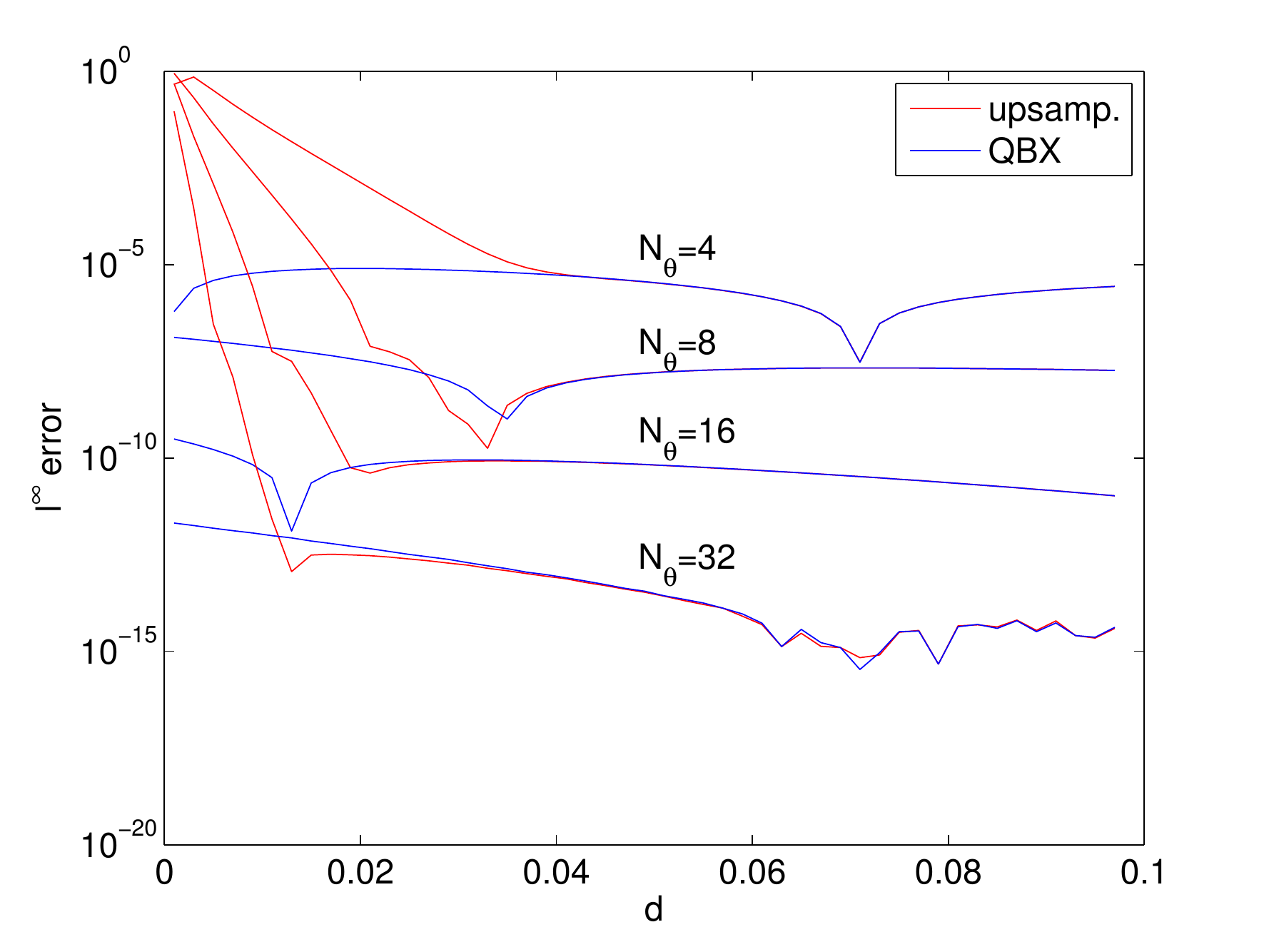}}   
  \caption{Comparison of relative $l^\infty$ error using the QBX method and  a local upsampling method to compute ${\cal D} Y_2^2$ at a target point a distance $d$ from the surface of the unit sphere. The different curves correspond to different numbers of panels $N_\theta(=N_\phi)$.  Parameter values are as in Table \ref{table1} ($p=20$). The oversampling factor for the local upsampling method is $\kappa=8$.  } 
  \label{fig2} 
\end{figure}

\subsection{Solution of the integral equation for different domains}

%Sources put inside body to  get exact solution. Use exact solution to
%define Dirichlet BCs. Solve integral equation, post-process to obtain
%solution in some sub-region (plane? zoomed in part?). 
%\begin{itemize}
%\item[a.]	One sphere. 
%\item[b.]	One body with $r=r(\theta,\varphi)$. 
%\end{itemize}

\subsubsection{Single sphere}

A more involved test case is to solve the second kind integral
equation (\ref{modified_second_kind}) or (\ref{eq:interior_second_kind}) that arises in the Green's function formulation of the Dirichlet problem. 
We use a Nystr{\"o}m method where the unknowns are the point values of the
density $\sigma$ at the target nodes, and employ GMRES to solve
the discrete equations
(\ref{eq:interior_discrete}) or (\ref{eq:exterior_discrete})
iteratively. At each iteration, the double layer potential is computed
via the target-specific QBX method with precomputation. Except where
noted, we set the GMRES tolerance to $10^{-10}$.  After solving for
$\sigma$, we use the  double layer representation
to evaluate the potential $u$ at a collection of target points.  
%For the interior problem we need to compute 
%$u(\bx)=\dlayer \sigma(\bx)$ and for the exterior problem we have
%$u= (\dlayer + BA') \sigma$, where $BA'\sigma$ is defined in
%(\ref{eq:BAprime}).
%To evaluate $\dlayer \sigma(\bx)$ whenever $\bx$ is close to $\partial %D$
%uses the algorithm described in Section \ref{sect:offsurface}. 
%%\ref{sect:full_alg}

%We  compare with exact solutions to compute the error.

As a first example, we specify boundary data on the unit sphere to be
that generated by an exact solution from
(\ref{exact_solution_int_ext}), and then compute the solution to the
interior/exterior Dirichlet problem.  The error in the computed
solution can  be assessed by comparison with the analytical
solution.  The results of some example computations are shown in Table
\ref{table3}.  The magnitude of the errors are similar to those in the
calculation of layer potentials alone. Note the number of iterations
in the GMRES calculation typically decreases as the number of panels
is increased, which is also observed in our other numerical tests.

\begin{table}[htbp]
\vspace{-.2in}
\begin{center}
\begin{tabular}{ |l | l | l | l | l | l | }
\hline
 side & $N_\theta(=N_\phi)$ & target $r_*$ &  $l^2$ error & $l^\infty$ error & its. \\
\hline
 int  & $4$ & $\sigma$  & $3.5  \times 10^{-6}$ & $5.0  \times 10^{-6}$ & 7 \\
       &         &  0.99& $6.6  \times 10^{-6}$  & $5.4  \times 10^{-6}$ & \\
       &         &  0.5& $7.1 \times 10^{-7}$  & $2.9  \times 10^{-7}$ & \\
\cline{2-6}
   & $8$ & $\sigma$  & $8.5  \times 10^{-8}$  & $1.0  \times 10^{-7}$ &  11\\ 
       &         &  0.99& $4.5  \times 10^{-8}$  & $4.9  \times 10^{-8}$ & \\
       &         &  0.5& $2.6  \times 10^{-8}$  & $1.3 \times 10^{-8}$ & \\
\cline{2-6}
    & $16$ & $\sigma$  & $2.7  \times 10^{-9}$  & $3.5  \times 10^{-9}$ & 4 \\
       &         &  0.99& $4.2 \times 10^{-10}$  & $2.7 \times 10^{-10}$ & \\
       &         &  0.5& $5.1 \times 10^{-11}$  & $2.4  \times 10^{-11}$ & \\ 
  \hline 
\hline
 ext  & $4$ & $\sigma$  & $4.1  \times 10^{-6}$ & $6.2 \times 10^{-6}$ & 6 \\
       &         &  1.01& $7.9  \times 10^{-6}$  & $7.1  \times 10^{-6}$ & \\
       &         &  1.5& $2.7  \times 10^{-6}$  & $9.9  \times 10^{-7}$ & \\
\cline{2-6}
   & $8$ & $\sigma$  & $3.2 \times 10^{-8}$  & $7.3  \times 10^{-8}$ &  5\\ 
       &         &  1.01& $1.7  \times 10^{-7}$  & $2.2  \times 10^{-7}$ & \\
       &         &  1.5& $4.4  \times 10^{-8}$  & $2.7  \times 10^{-8}$ & \\
\cline{2-6}
    & $16$ & $\sigma$  & $8.8  \times 10^{-10}$  & $1.1 \times 10^{-9}$ & 3 \\
       &         &  1.01& $1.6 \times 10^{-9}$  & $8.9 \times 10^{-10}$ & \\
       &         &  1.5& $1.1  \times 10^{-10}$  & $5.5  \times 10^{-11}$ & \\
\hline
\end{tabular}
\end{center}
\caption{ Errors in evaluating the solution to an interior/exterior Dirichlet problem at a set of targets on a sphere of radius $r_*$, via solving a second kind integral equation for the density $\sigma$.  Boundary data is given by (\ref{exact_solution_int_ext}) with $l=m=2$ and $\rho=1$. A $\sigma$ in the $r_*$ column denotes that the errors are listed for $\sigma$.  GMRES iteration counts are shown in the last column.  Fixed parameter values are $\kappa=8, p=20, r_c=0.2, \dqbx=0.7$.     }
\label{table3}
\end{table}

%\noindent
%{\bf Additions/changes to this table: }

%\begin{enumerate}
%\item
%Make the number of global panels commensurate with local panels, so can give data for fixed $\kappa_G$, and remove $N_{G\theta}$.
%\item
%Change to relative error.
%\end{enumerate}

Another test of our method is shown in Table \ref{table4}. For an exterior boundary value problem, we now generate an exact solution to Laplace's equation by introducing a collection of point charges in the interior region.   
%For an interior domain, we similarly construct an exact solution using point charges in the exterior.  
This exact solution is then used to provide Dirichlet data for our boundary value problem, which is solved using an integral equation. We test the accuracy of our solution at a collection of target points on a sphere of radius $r_*$  surrounding the domain $D$.   
 %(Should we put a figure that illustrates this for one and two spheres, with off-surface targets also shown, similar to Figure 7 in Klockner et al?).  

The results of this exterior calculation are  shown in Table \ref{table4}.  For both of the examples in Tables \ref{table3} and \ref{table4}, the error in solving the integral equation and evaluating the solution near the surface $\partial D$  is roughly  the same order of magnitude as the error in evaluating the double layer potential alone (cf. Tables \ref{table1} and \ref{table2}).   We note that accuracy is maintained at a similar level when the target $r_*$ is even closer to $1$.
This shows that the method can compute within a specified accuracy for target points that are arbitrarily close to the domain boundary.

\begin{table}[htbp]
\begin{center}
\begin{tabular}{ |l | l | l | l | l | l | }
\hline
 side & $N_\theta(=N_\phi)$ & target $r_*$ &  $l^2$ error & $l^\infty$ error & its. \\
\hline
 ext  & $4$ & $1.005$  & $1.4  \times 10^{-5}$ & $1.9  \times 10^{-5}$ & 10 \\
       &         &  1.5& $2.8  \times 10^{-6}$  & $1.9  \times 10^{-6}$ & \\
\cline{2-6}
   & $8$ & $1.005$  & $6.2  \times 10^{-7}$  & $8.0  \times 10^{-7}$ & 9\\
 
       &         &1.5& $8.9  \times 10^{-8}$  & $5.4  \times 10^{-8}$ & \\
\cline{2-6}
    & $16$ & $1.005$  & $5.8  \times 10^{-9}$  & $7.2  \times 10^{-9}$ & 8 \\
       &         &  1.5 & $1.5 \times 10^{-10}$  & $8.8 \times 10^{-11}$ & \\
  \hline 
\end{tabular}
\end{center}
\caption{ Accuracy of exterior Dirichlet computation for a domain with
  a single sphere, when  boundary data is generated by $49$ point
  charges spread around a spherical interior surface  with radius
  $\rho=0.5$. %The point charges are  rotated so that they are
  %unaligned with the surface discretization points. \textcolor{red}{Correct?}  
Errors are evaluated at a set of targets on a sphere of radius $r_*$.   Other parameter values are as in Table \ref{table3}.
 }
\label{table4}
\end{table}

%\noindent
%{\bf Should we put interior Dirichlet problem also?}

%\vspace{.25in}
%\noindent
%{\bf Additions/changes to this table: }

%\begin{enumerate}
%\item
%Make the number of global panels commensurate with local panels, so can give data for fixed $\kappa_G$, and remove $N_{G\theta}$.
%\item
%Change to relative error.
%\item
%Put interior Dirichlet problem too?

%\end{enumerate}

\subsubsection{Nonspherical shape}

Table \ref{table5} illustrates the accuracy of our method for  some  example nonspherical geometries.  In (a) we consider a triaxial ellipsoid with semi-major axes   lengths  $1/2, 1$, and  $2$.  In (b) we choose a shape with four-fold symmetry in the $\phi$-direction,  
whose boundary is described by 
$\bx (\theta, \phi ) = \rho_s (\theta, \phi) {\pmb \omega}(\theta, \phi)$ where
$\rho_s(\theta, \phi) = 1 + \varepsilon \sin^2 \theta \cos 4 \phi$ and ${\pmb \omega} (\theta, \phi)$ is the standard spherical-coordinate parameterization of a unit sphere,  and we take $\varepsilon=0.3$. This form satisfies the requirement that $\rho_s$ is independent of $\phi$ at $\theta=0, \pi$ and that $\partial_\theta \rho_s (\theta=0, \phi) = \partial_\theta \rho_s (\theta=\pi, \phi) =0$, which is necessary for smoothness of the surface shape at the poles. We introduce a collection of unit point charges in the interior of each shape, spread over a spherical surface of radius $\rho=0.2$, and use the point charges to generate boundary data for the exterior problem. The boundary shapes and interior point charges are shown in Figure \ref{fig4a}.

\begin{figure}
\vspace{-.5in}
%\center{(a) Triaxial ellipsoid}
  \center{\includegraphics[width=.49\textwidth]{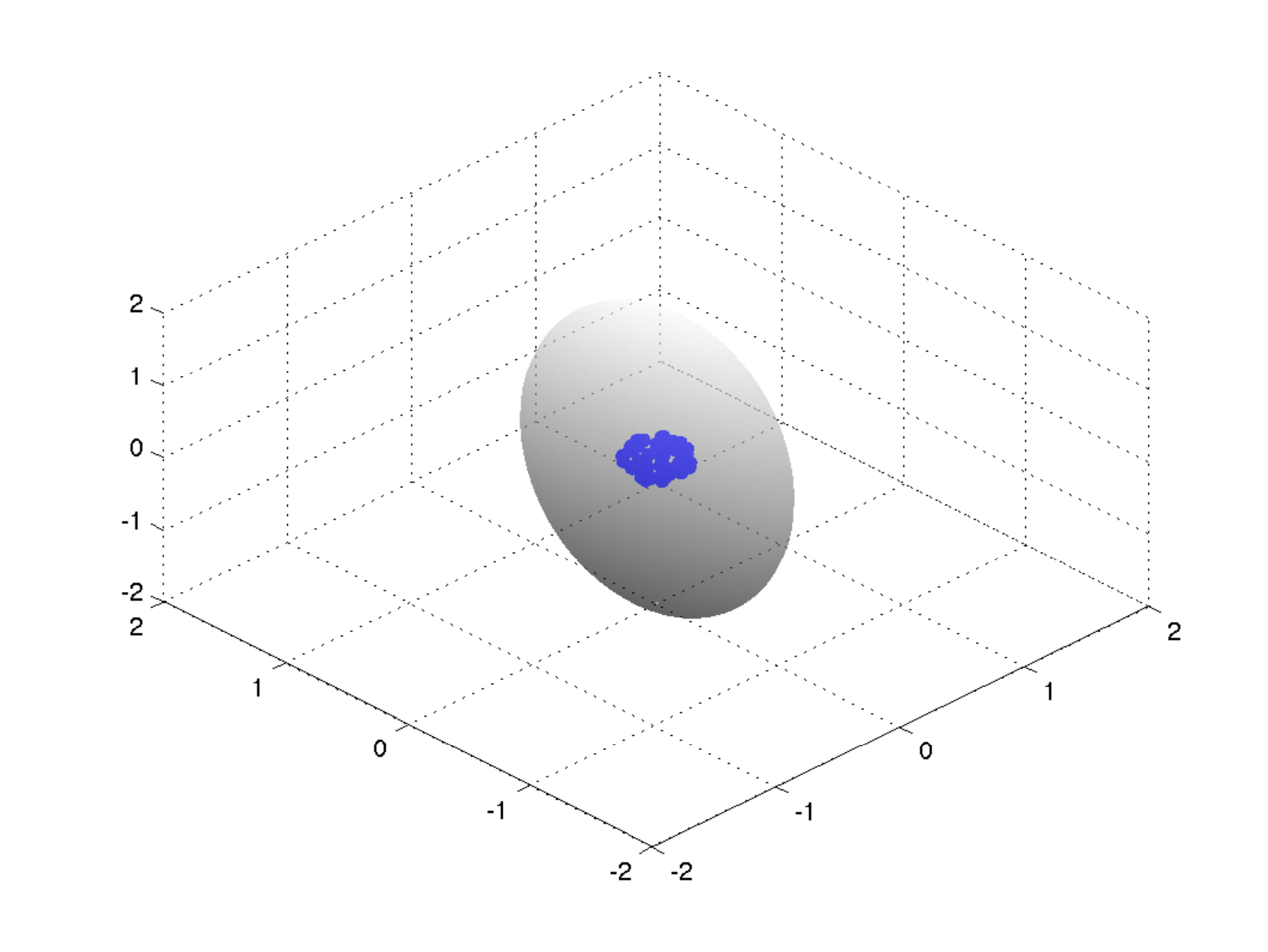}
%} 
%\center{(b) Four-fold symmetric shape}
% \center{
\includegraphics[width=.49\textwidth]{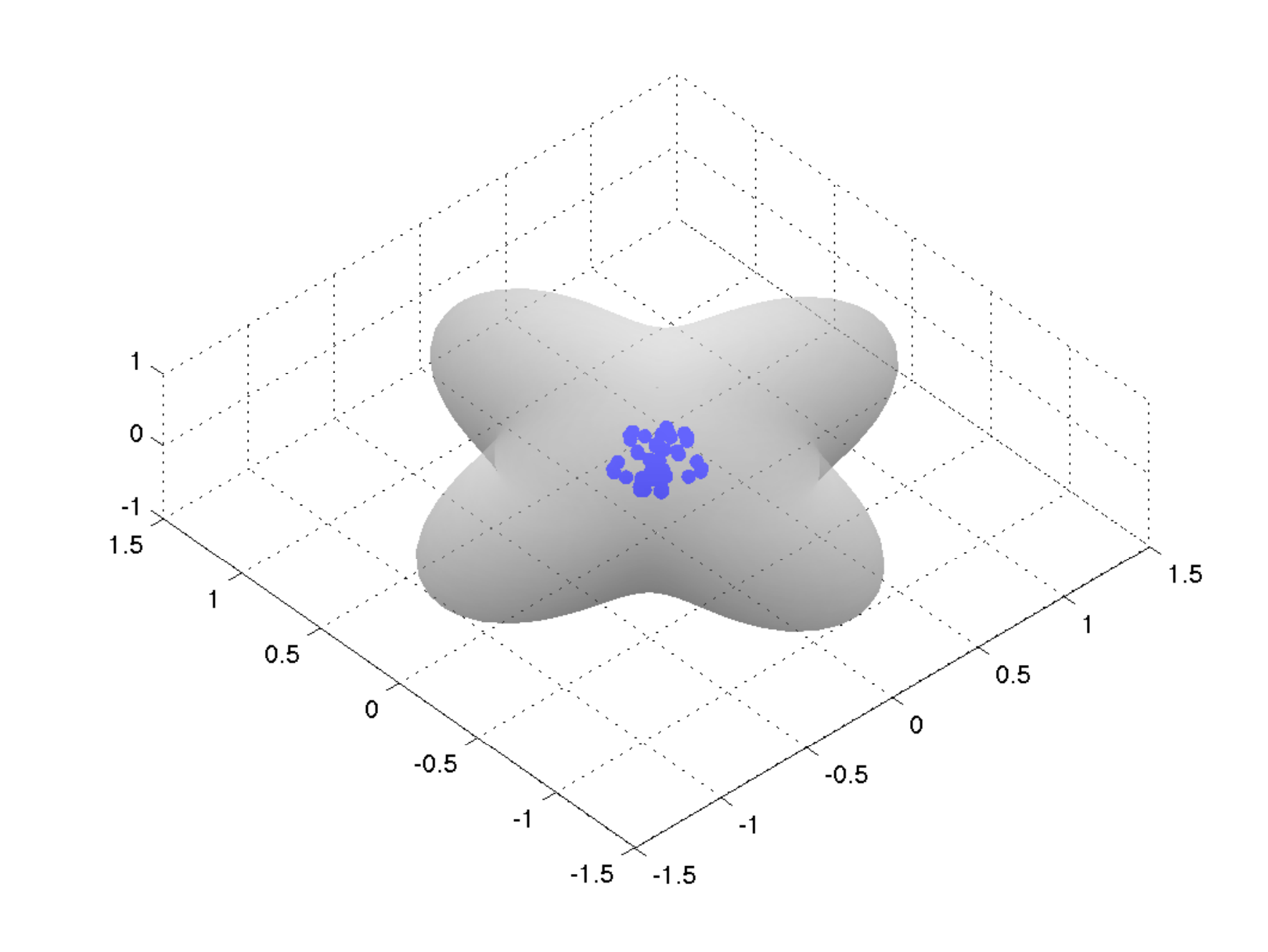}}   
  \caption{Boundary shapes and point charges for the tests of the solution
to the exterior Dirichlet problem in Table \ref{table5}. Left:
Triaxial ellipsoid, right: four-fold symmetry.} 
  \label{fig4a} 
\end{figure}

In these examples, we  employ  the same scaling of parameters as in Table \ref{table2};  namely  $\dqbx \sim r_c \sim h$.  The support of the region where the local QBX correction is made  then shrinks for each increase in $N_\theta$, which gives a  storage and computation cost of $O(1)$ per target. 

The results are shown in Table \ref{table5}.
The decrease in error with $N_\theta$ for the triaxial ellipsoid  is roughly similar to that which is observed for a  spherical shape. For the four-fold symmetric shape, the error does not decay as rapidly with $N_\theta$. This is possibly due to the presence of both positive and negative signed curvature  along the surface.

\begin{table}[htbp]
\vspace{-0.in}
\begin{center}
(a) Triaxial ellipsoid

\begin{tabular}{| l | l | l | l | l | l | l | }
\hline
  $N_\theta(=N_\phi)$ & $\mbox{target}~ r_*$ & $\dqbx$ & $r_c$ & $l^2$ error & $l^\infty$ error & its. \\
\hline
  $2$ & $1.005$ & $0.7$ & $0.2$ &  $2.9 \times 10^{-4}$  & $1.4 \times 10^{-3}$  & 16 \\
$$ & $2.5$ & &  &  $9.7 \times 10^{-5}$  & $1.0 \times 10^{-4}$  &   \\
\hline 
  $4$ & $1.005$ & $0.7$ & $0.2$ &  $3.3 \times 10^{-6}$  & $4.6 \times 10^{-6}$ & 18  \\
$$ & $2.5$ & $$ & $$ &  $1.9 \times 10^{-6}$  & $2.9 \times 10^{-6}$ &  \\
 \hline 
  $8$ & $1.005$ & $0.35$ & $0.1$ &  $1.7 \times 10^{-7}$  & $1.9 \times 10^{-7}$ &  17\\
$$ & $2.5$ & $$ & $$ &  $3.8 \times 10^{-9}$  & $1.2 \times 10^{-8}$ &  \\
 \hline 
  $16$ & $1.005$ & $0.17$ & $0.05$ &  $4.0 \times 10^{-9}$  & $5.2 \times 10^{-9}$  & 16  \\
$$ & $2.5$ & $$ & $$ &  $4.0 \times 10^{-11}$  & $5.9 \times 10^{-11}$  &   \\
 \hline 
\end{tabular}

(b) Four-fold symmetric shape
\vspace{.3in}
\begin{tabular}{| l | l | l | l | l | l | l | }
\hline
  $N_\theta(=N_\phi)$ & $\mbox{target}~ r_*$ & $\dqbx$ & $r_c$ & $l^2$ error & $l^\infty$ error & its. \\
\hline
  $2$ & $1.005$ & $1.4$ & $0.4$ &  $9.4 \times 10^{-2}$  & $6.8 \times 10^{-2}$  & 17  \\
$$ & $1.8$ & &  &  $7.5 \times 10^{-2}$  & $5.1 \times 10^{-2}$  &   \\
\hline 
  $4$ & $1.005$ & $0.7$ & $0.2$ &  $3.3 \times 10^{-5}$  & $3.2 \times 10^{-5}$ & 17  \\
$$ & $1.8$ & $$ & $$ &  $8.1 \times 10^{-5}$  & $5.6 \times 10^{-5}$ &  \\
 \hline 
  $8$ & $1.005$ & $0.35$ & $0.1$ &  $2.0 \times 10^{-6}$  & $3.1 \times 10^{-6}$ &  17\\
$$ & $1.8$ & $$ & $$ &  $4.3 \times 10^{-7}$  & $3.2 \times 10^{-7}$ &  \\
 \hline 
  $16$ & $1.005$ & $0.17$ & $0.05$ &  $9.4 \times 10^{-8}$  & $7.9 \times 10^{-8}$  &  16 \\
$$ & $1.8$ & $$ & $$ &  $3.5 \times 10^{-10}$  & $2.3 \times 10^{-10}$  &   \\
 \hline 
\end{tabular}
\end{center}
\caption{Accuracy of  exterior Dirichlet computation for nonspherical domains. Boundary data is generated by $49$ point charges spread on an interior spherical surface of radius $\rho=0.2$.  Errors are evaluated at a set of targets on a larger surrounding surface of the same shape, but scaled by 1.005, and on a spherical surface of radius $r_*=2.5$ in (a) and $r_*=1.8$ in (b).  Other parameters values are as in Table \ref{table2}.  }
\label{table5}
\end{table}

%\vspace{.25in}
%\noindent
%{\bf Additions/changes to this table: }

%\begin{enumerate}
%\item
%Make the number of global panels commensurate with local panels, so can give data for fixed $\kappa_G$, and remove $N_{G\theta}$.
%\item
%Change to relative error.
%\item
%Pick a different shape instead of this (e.g. triaxial ellipsoid?).

%\end{enumerate}

\subsection{Two spheres close to touching } 

\begin{figure}[htbp]
\vspace{-.0in}
%\resizebox{5cm}{!}{ 
% \center{
%\includegraphics[width=.75\textwidth]{Figures/two_sphere.eps}}
%\put(-217,190){{\large $D_1$}}
%\put(-115,190){{\large $D_2$}}
%} 
\center{
\resizebox{7cm}{!}{ 
 \includegraphics[width=.75\textwidth]{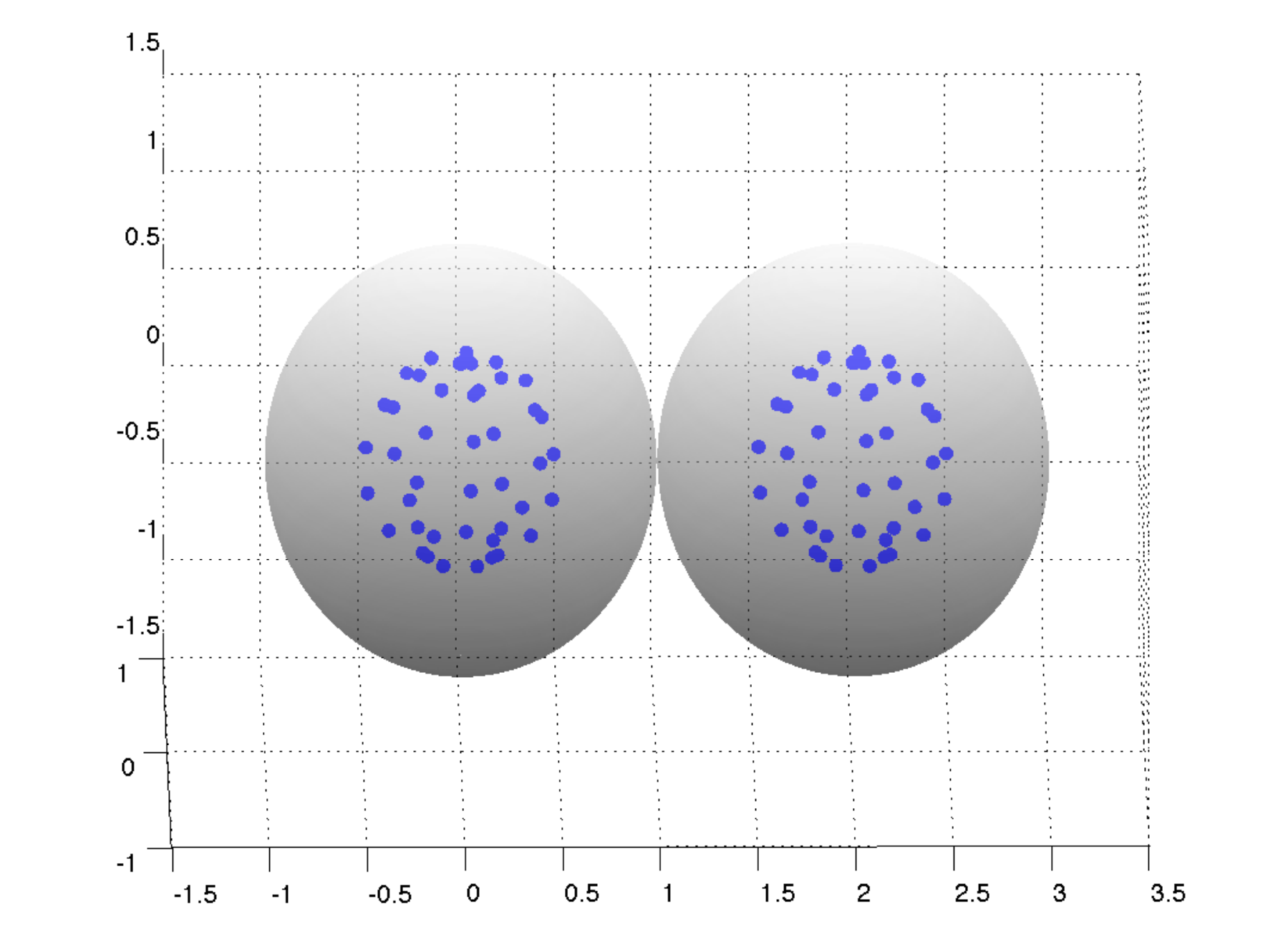}}
} 
\put(-170,100){{ $D_1$}}
\put(-45,100){{ $D_2$}}
 \caption{Boundary locations and point charges for the test of the solution
 in Table \ref{table6}. The solution is evaluated at a set of targets on a spherical surface that surrounds $D_1$ and bisects the gap between the two spheres.} 
  \label{fig4} 
\end{figure}

An example using the full target-specific QBX algorithm to solve an exterior Dirichlet problem for two nearly touching spheres is given in Table \ref{table6}.  Let $D_1$ and $D_2$ denote the two spherical regions, each  with unit radius. We introduce point sources inside $D_1$ and $D_2$ at a radius $\rho=0.5$ with respect to each region's center to generate boundary data.
After solving the integral equation for the density  $\sigma$, we evaluate the solution $u(\bx)$ at a collection of targets located on a spherical surface that surrounds $D_1$. The  radius $r_*$ of the target surface is chosen to  bisect the gap between $D_1$ and $D_2$. This is a challenging test case  for our method. 

\begin{table}[htbp]
\begin{center}
\begin{tabular}{ |l | l | l | l | l | l | }
\hline
 side & $N_\theta (=N_\phi)$  & target $r_*$ &  $l^2$ error & $l^\infty$ error & its. \\
\hline
 ext  & $2$ & $1.005$  & $7.1  \times 10^{-4}$ & $5.5  \times 10^{-4}$ & 31 \\
\cline{2-6}
   & $4$ & $1.005$  & $6.6  \times 10^{-6}$  & $1.1 \times 10^{-5}$ & 34\\
\cline{2-6}
    & $8$ & $1.005$  & $2.1  \times 10^{-7}$  & $3.1  \times 10^{-7}$ & 27\\
\cline{2-6}
    & $16$ & $1.005$  & $1.7  \times 10^{-9}$  & $1.3  \times 10^{-9}$ & 25\\
  \hline 
\end{tabular}
\end{center}
\caption{ Accuracy of exterior Dirichlet computation for a domain with
  two unit spheres, with centers a distance $d=2.01$ apart. Boundary
  data is generated by $49$ point charges located on a spherical
  surface inside each sphere  with radius $r_{p}=0.5$. 
%The point charges are rotated so that they do not align with the boundary discretization points. 
Errors are evaluated at a set of targets on a sphere of radius $r_*$ surrounding $D_1$.  The number of panels $N_\theta$ is per sphere. Other parameter values are as in Table \ref{table2}. }
\label{table6}
\end{table}

The error shows good convergence as the number of panels is increased. This demonstrates the accuracy of our method in computing solutions to boundary value problems for closely spaced surfaces.

\begin{figure}
\centering
\begin{subfigure}[t]{.5\textwidth}
  \centering
  \includegraphics[width=\textwidth]{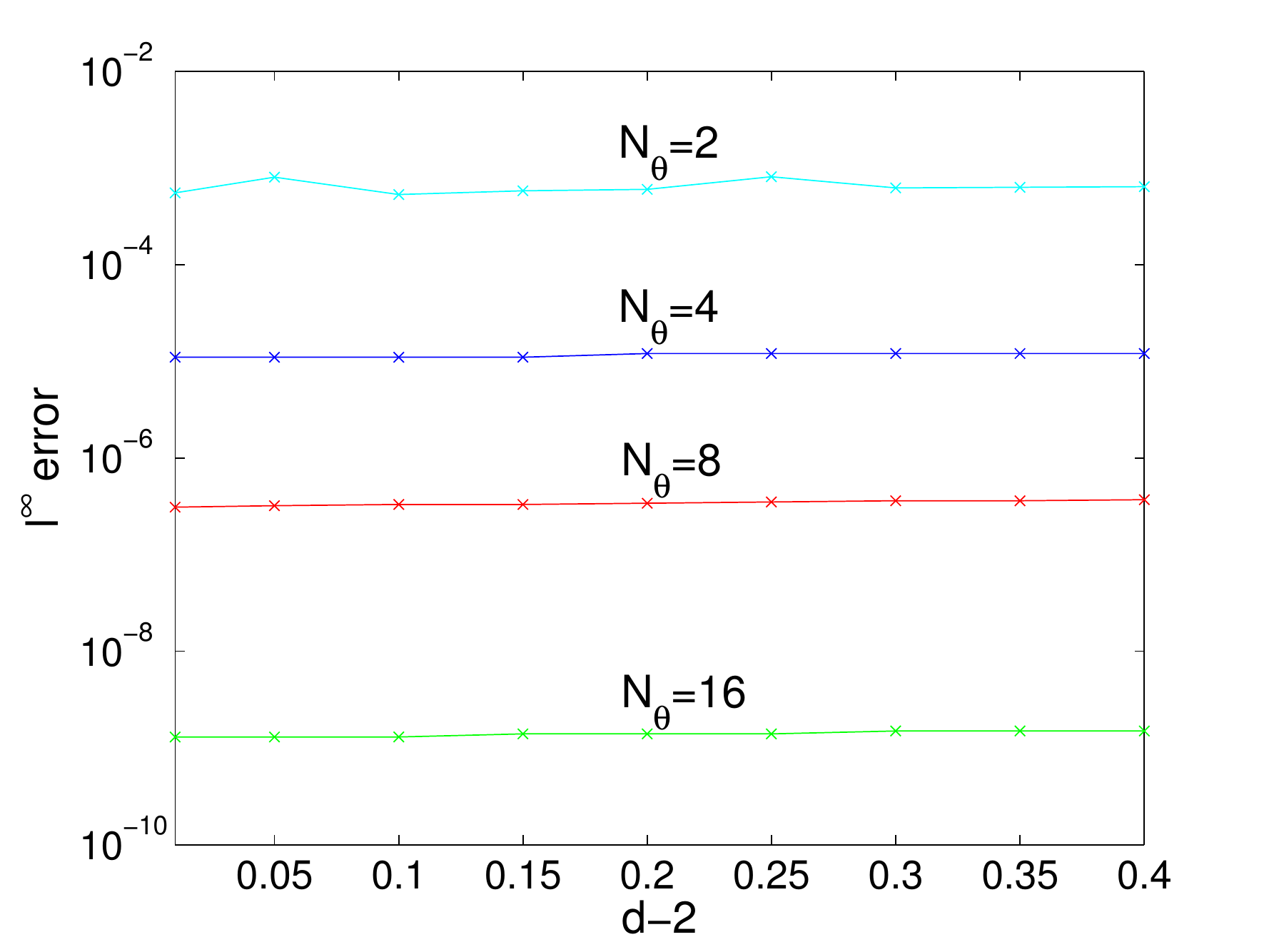}
  \caption{$l^\infty$ error for the  exterior Dirichlet computation with two unit spheres versus surface separation distance $d-2$. Boundary data, parameter values, and computation of error are as in Table \ref{table6}. }
  \label{fig5}
\end{subfigure}%  
~~~~~
\begin{subfigure}[t]{.5\textwidth}
  \centering
  \includegraphics[width=\textwidth]{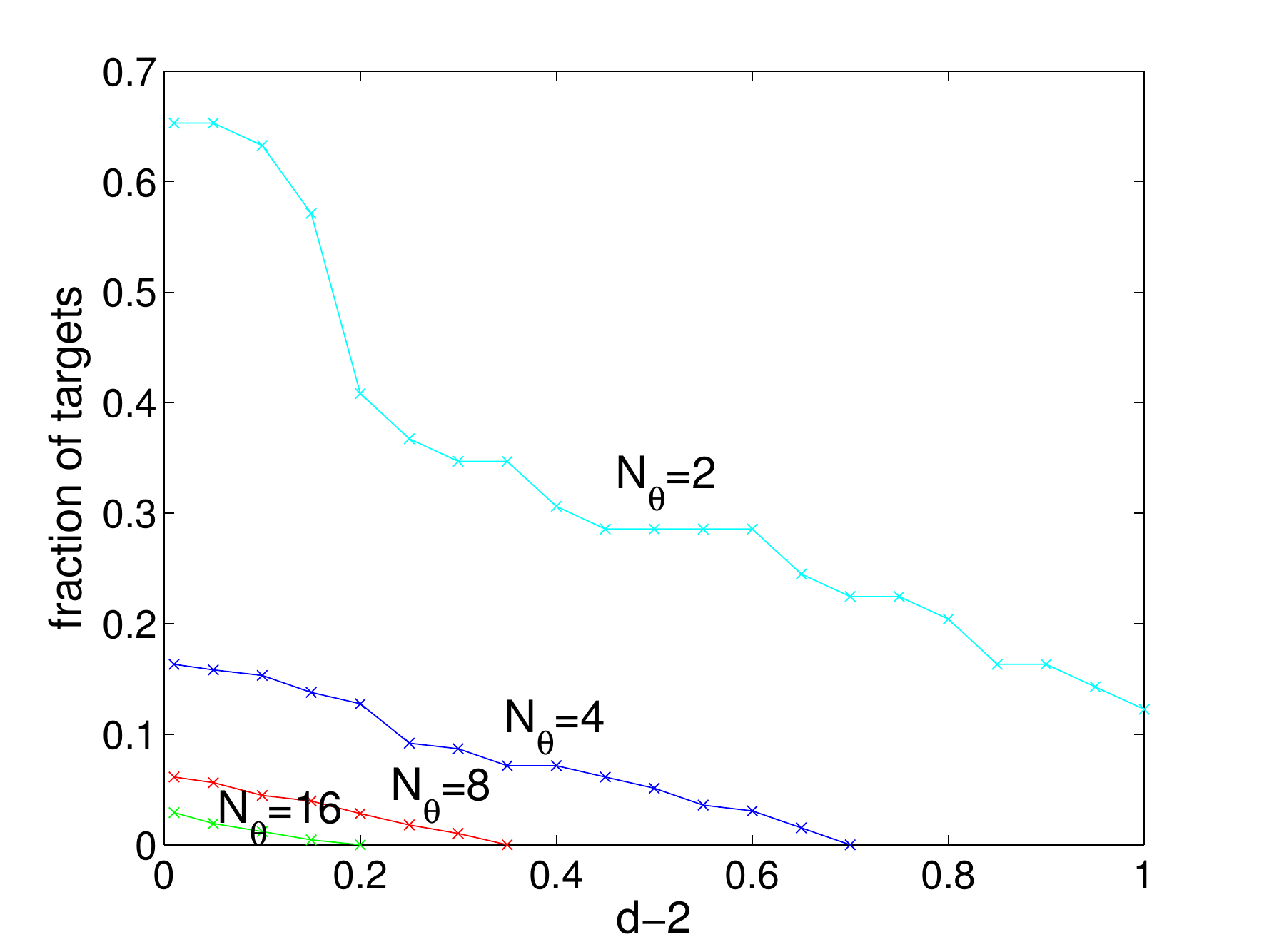}
  \caption{Fraction of targets for which the on-the-fly local QBX
    correction is made, versus center separation distance $d$, for the
    two-sphere example in Table \ref{table6}. Parameter values are the
    same as in Figure \ref{table6}.}
  \label{figure6}
\end{subfigure}
\caption{Effect of separation distance for two sphere problem}
\label{fig5and6}
\end{figure}

%\begin{figure}[htbp]
%%\vspace{-0.3in}
% \center{\includegraphics[width=.5\textwidth]{Figures/two_sphere_error_vs_d.eps}} 
%  \caption{$l^\infty$ error for the  exterior Dirichlet computation with two unit spheres versus surface separation distance $d-2$. Boundary data, parameter values, and error calculation are as in Table \ref{table6}.  } 
%  \label{fig5} 
%\end{figure}

Figure \ref{fig5} shows 
the $l^\infty$ error for the two-sphere calculation in Figure \ref{fig4} as the  surface separation distance $d$ is varied. The error is found to be    uniform in the separation distance.  

We comment on the efficiency of our method for nearly touching particles.
When performing the on-surface computation for a single particle, or
for multiple particles which are well-separated, there is no need to
make the on-the-fly local QBX correction for the interaction of target
and source points that lie on different particle surfaces. In this
case, all of the local QBX corrections are for the interaction of
target/source points that lie on the same surface. These can be
precomputed, and the efficiency of the full method is about the same
as  for computing the global integral using the original grid. 

In the computation for two nearly touching particles, however, the on-the-fly local QBX correction is made for target points on each particle that are near the other particle's surface. 
%Although this correction can take a significant fraction of the total CPU time, we find the fraction diminishes as the resolution is increased.  
Figure \ref{figure6} shows the fraction of target points for which the
on-the-fly local correction is made in the two-sphere example in Table
\ref{table6}, plotted versus separation distance $d$.  At higher
resolutions, the fraction of targets that require the local QBX
correction decreases.  This is because of the decrease in $d_{QBX}$
with increasing resolution.  Thus, 
%for spherical surfaces, 
the
fraction of CPU time spent making the on-the-fly local correction
becomes negligible as resolution is increased. The fraction of
corrected targets also decreases with increasing separation $d$, as
expected.

%\begin{figure}[htbp]
%%\vspace{-1.in}
%  \center{\includegraphics[width=.5\textwidth]{Figures/targets_close.eps}}   
%  \caption{Fraction of targets for which the on-the-fly local QBX
%    correction is made, versus center separation distance $d$, for the
%    two-sphere example in Table \ref{table6}. Parameter values are the
%    same as in Figure \ref{table6}. 
%\textcolor{red}{If larger text font in Figure 6 and 7, can be shrunk
%  and put side by side}
%} 
%  \label{figure6} 
%\end{figure}

%\vspace{.25in}
%\noindent
%{\bf Additions/changes to this table: }

%\begin{enumerate}
%\item
%Make the number of global panels commensurate with local panels, so can give data for fixed $\kappa_G$.
%\item
%Change to relative error.
%\item
%Do cases with different sphere spacing?

%\end{enumerate}

\subsection{A larger problem}

\begin{figure}
\vspace{-0.6in}
  \center{\includegraphics[width=.85\textwidth]{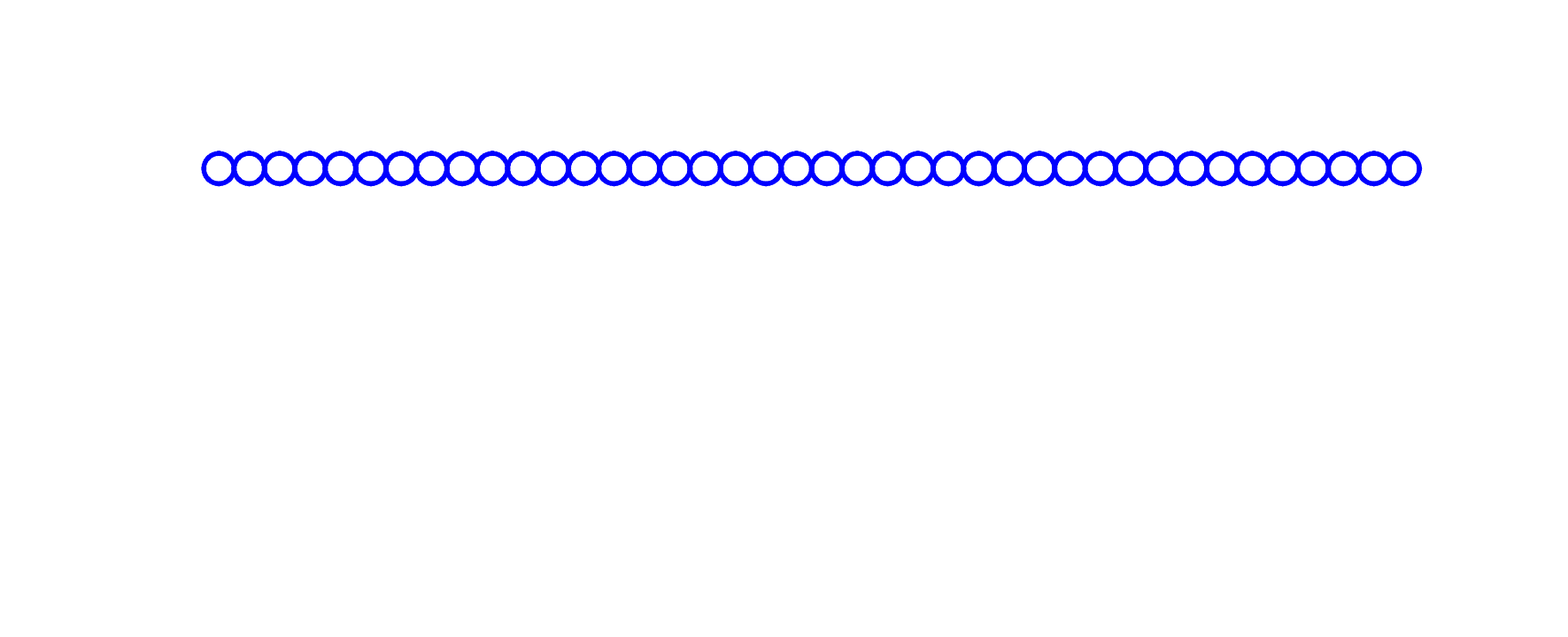}}   
  \vspace{-1.in}\caption{Configuration of 40 spheres of unit radius for the exterior Dirichlet computation in Table \ref{table7}. The distance between spheres is $0.01$. } 
  \label{figure8} 
\end{figure}

\begin{table}[htbp]
\begin{center}
\begin{tabular}{| l | l | l | l |}
\hline
  No. ~spheres & $l^2$ error & $l^\infty$ error & Its. \\
\hline
%  $2$ & $1.8  \times 10^{-5}$  & $6.6  \times 10^{-6}$ & $34$ \\
%  $3$ & $4.8 \times 10^{-5}$  &  $2.0 \times 10^{-5}$ & $37$ \\
  $4$ & $2.0 \times 10^{-5}$ & $8.3 \times 10^{-6}$ & $45$  \\
  $10$ & $6.8 \times 10^{-5}$ & $2.2 \times 10^{-5}$ & $66$ \\
  $20$ & $3.6 \times 10^{-6}$ & $5.6 \times 10^{-6}$ & $96$ \\
  $30$ & $3.7 \times 10^{-6}$ & $5.7 \times 10^{-6}$ & $154$ \\
  $40$ & $2.5 \times 10^{-5}$ & $8.3 \times 10^{-6}$ & $267^*$ \\
  \hline  
\end{tabular}
\end{center}
\caption{ Accuracy of exterior Dirichlet computation for multiple unit spheres.  The number of panels is fixed at $N_\theta(=N_\phi)=4$ per sphere,  with other parameters as in  Table \ref{table2}.  Generation of boundary data and location of target surface are as
in Table \ref{table6}. The asterisk denotes that the GMRES tolerance was raised from $10^{-10}$ to $10^{-8}$ for the computation with $40$ spheres.  
 }
\label{table7}
\end{table}

Our last example considers the exterior Dirichlet problem for a domain with a larger number of unit spheres. We now fix the number  of panels per sphere at $N_\theta (= N_\phi)=4$, but vary the number of particles in the domain.  For simplicity, the particles are arranged in a simple linear configuration (see figure \ref{figure6}), with the distance between spheres fixed at $0.01$. The boundary Dirichlet data is generated in the same way as for the two-sphere example, i.e., by introducing point sources inside each sphere.  

Here we use a treecode, the implementation of which closely follows  \cite{Krasny},  to compute the global double layer integral. The treecode algorithm divides interface points into a nested set of clusters and approximates velocity interactions between a point and a distant cluster using Taylor's expansion. Interaction of nearby points is computed using direct summation.   The treecode algorithm introduces two additional parameters, the order of the Taylor's series expansion $p_T$, and a separation parameter $\epsilon_T$. For the particle-cluster interaction, the Taylor's approximation is chosen over the direct sum when $\epsilon_T>  R_T/D_T$, where $R_T$ is the radius of the cluster and $D_T$ is the distance from the target point to the cluster center. A decrease in $\epsilon_T$ improves the accuracy of the treecode, but slows the computations. Our computations use $\epsilon_T=0.2$ and $p_T=5$. For a small number of spheres, these values were found to give results that are indistinguishable from direct summation, for the selected number of panels.

The results of our computations are shown in Table \ref{table7}. The
observed accuracy, which is limited by the discretization error, is
found to be roughly independent of the number of spheres, but the
number of GMRES iterations grows due to ill-conditioning.
In \cite{Klinteberg2016},
a similar behavior is observed as the number of spheroidal particles
grows, and the problem is alleviated by introducing a block diagonal
preconditioner.

%In \cite{Greenbaum1993}, similar behavior is observed in the 2D
%problem and an efficient preconditioner is developed which prevents
%the growth in the number of iterations.  The design of such a
%preconditioner for the 3D problem will be the focus of future work.

\subsection{Operation count for TSQBX versus QBX}
%%\ref{eq:qbx_loc_exp_h} replace target_spec_ref by qbx_loc_exp_h
\tg{
Truncating the expansions of the double layer potential at an
expansion order $n=p$, 
the target specific QBX (TSQBX) expansion (\ref{eq:qbx_loc_exp_h}) 
has $p+1$ terms, where as the corresponding QBX expansion (cf. (\ref{eq:qbx_sum}))  has $(p+1)^2$ terms. 
Hence, there will be more coefficients to compute and more terms to
evaluate and sum in the general QBX expansion as compared to the target specific
expansion. On the other hand, the coefficients in the general QBX expansion do
not depend on the target point. Hence, several target points can
use the same expansion for evaluation, if they fall within the radius
of convergence of the expansion. 
} %%end green color. 

Let us choose a set $\tau=\{ \bx_i: i=1, \ldots T\}$ of $T$ targets
inside a sphere of radius $r$ centered at a center $\bc$, and ask for the value
of $T$ at which the QBX computation of the double layer potential on
$\tau$ is similar in cost to the the TSQBX computation.  

\tg{   %green color   
The cost of evaluating the QBX expansion at $T$ targets equals the
cost of computing the $(p+1)^2$ coefficients $z_{nm}^h$ plus the cost
of evaluating the truncated sum of the form (\ref{eq:qbx_sum}) at the targets:
\be \label{QBX_complexity}
QBX ~ complexity  = c_{QBX} N_Q (p+1)^2+ c_T  T (p+1)^2.
\ee
Here $c_{QBX} N_Q$ is the cost of evaluating  one of the coefficients
$z_{nm}^h$ using the  localized discrete version of   (\ref{eq:qbx_znm}), 
with $N_Q$ the total number of (upsampled) quadrature  points, and
%% in $\setofpanels$, 
$c_T$ is the cost of evaluating a single term in the sum
(\ref{eq:qbx_sum}), given the coefficients $z^h_{nm}$.  By the same
reasoning, the cost of evaluating the TSQBX expansion
(\ref{eq:qbx_loc_exp_h}) at $T$ targets is 
\be 
\label{TSQBX_complexity} 
TSQBX ~ complexity = c_{TSQ} N_Q T (p+1) + c_S T (p+1).
\ee Both terms are now proportional to $p+1$ instead of $(p+1)^2$, but on
the other hand, the first term is  proportional to $T$ as
coefficients are recomputed for each target. 
%Here $c_{TSQ} N_Q$ is the cost of evaluating one of the coefficients
%$z_n^h(\bx)$, and $c_S$ is the cost of evaluating a single term in the
%sum (\ref{TSQBX_sum}), given the coefficients $z_n^h(\bx)$.
Typically
$c_{TSQ} N_Q>>c_S$ so we can neglect the second term in (\ref{TSQBX_complexity}).
} %%end green color. 

\tg{   %green color   
The number of targets $T$ for which the two methods have comparable
complexity is then found by equating (\ref{QBX_complexity}) and
(\ref{TSQBX_complexity}).  The result is 
\be \label{number_targets}
T=\frac{c_{QBX}N_Q(p+1)}{c_{TSQ}N_Q-c_T(p+1)}.  
\ee 
A rough count of the
number of operations to compute $z^h_{nm}$ and $z_n^h$ suggests
that $c_{QBX} \approx 2 c_{TSQ}$, for the same set of panels and
quadrature points, due to the slightly more complicated integrand in
(\ref{eq:qbx_znm}) compared to (\ref{eq:qbx_znD}).  We also expect
that we can neglect the second term in the denominator of
(\ref{number_targets}) compared to the first term. Taking into account
these simplifications, we have 
\be \label{target_crossover} 
T \sim 2(p+1).  
\ee
} %%end green color. 

\tg{   %green color  
When using QBX for on-surface evaluations, one center is often used
per target point. Even if a few centers use the same expansion, 
it is clearly beneficial to use the target specific expansion. 
This is also true e.g.\ when the double layer potential is computed on
nearly touching spheres as in our examples. On the other hand, if the
solution is to be computed in a dense set of points close to a
surface, such that the number of points per center/expansion grows
past $2p$, it will be beneficial to use the original QBX
expansion. 
} %%end green color. 

\section{Conclusions}

We have developed a local target specific QBX method to evaluate singular and nearly singular layer potentials in 3D, and have applied it to boundary value problems for Laplace's equation in multiply-connected domains. 
Here, the approach to QBX is different from \cite{Rachh2016}  in that our QBX algorithm is not designed to be integrated into the FMM or other hierarchical fast algorithm.  Our method takes a local approach, and the QBX correction is only applied over those surface  panels that are close to the evaluation point, for which a standard quadrature has large error. For other panels the integral is well resolved using standard quadrature, and the QBX correction is not necessary.  
\tb{We consider only domains with smooth boundaries, but the method can potentially be adapted to nonsmooth domains (e.g., with corners) by applying special quadratures \cite{Helsing2016}. }

Our local QBX method is designed to  have $O(N)$  complexity for $N$ surface discretization points. This is achieved  by (i) combining with a fast hierarchical method such as an FMM or treecode to compute the contribution to the layer potentials from source panels that are outside of the local correction patch, and (ii) scaling numerical parameters so that the local QBX correction at a given evaluation point can be computed with $O(1)$ complexity.   We  emphasize that the fast hierarchical algorithm in step (i)  is  decoupled from the QBX expansions, which simplifies the implementation of our method.  A detailed error analysis developed here and in \cite{Klinteberg2016b}  aids in the selection of parameters for step (ii).
%Expressions for the truncation and quadrature errors developed here and in \cite{Klinteberg2016b} aid in the selection of parameters for (ii).

The  QBX expansion coefficients are computed by oversampled Gauss-Legendre quadrature, which can be expensive, but our method is accelerated by several key choices. For one, we  make use  of a target specific expansion in Legendre polynomials, rather than the usual spherical harmonic or Taylor's series expansions.  The target specific expansion requires only $p$ terms to achieve the same accuracy as $O(p^2)$ spherical harmonic expansion terms or $O(p^3)$ terms of a Taylor's expansion. Secondly, we precompute the contributions to the QBX expansion coefficients from panels that lie on the same surface as the evaluation point, which gives a significant speed-up. 
Contributions to the expansion coefficients from panels that lie on a different surface than the evaluation point are computed 'on-the-fly'. We do this with  acceptable efficiency  by employing the target specific expansion, designing an adaptive oversampling scheme, and making a judicious choice of numerical parameters.  

Finally, although we consider the specific application to Laplace's equation, the method developed here can be extended to other applications involving a boundary integral formulation. This includes problems in Stokes flow, potential flow, electromagnetics, and elasticity theory.  We intend to apply our method to some of these applications, including those involving time-evolving interfaces, in future work.

%UNUSED REF \cite{Lindbo2012}

\subsection*{Acknowledgments}
This work has been supported by the Knut and Alice Wallenberg Foundation under grant no. KAW2014.0338 and is gratefully acknowledged. 
The authors also gratefully  acknowledge support by the G{\"o}ran Gustafsson Foundation for Research in the Natural Sciences (A.K.T.), and the National Science Foundation grant DMS-1412789 (M.S.). 

%\begin{appendices}

\section{Appendix A: Equivalence of Cartesian Taylor expansions and spherical
  harmonics expansions}
\label{app:equivexp}

In this appendix, we want to explicitly show the equivalence of the
Cartesian Taylor expansion (\ref{Taylors_approx}) and the spherical harmonics expansion
(\ref{eqn:sph_expansion}) when they are truncated appropriately. 

Consider the recursion relation (\ref{eqn:rec_rel}) for the $b_{\bk}$ coefficients.
Introduce
\[
B_n(\bc,\bx,\by)=\sum_{\normk=n} b_{\bk}(\bc,\by)
(\bx-\bc)^\mathbf{k}, 
\]
such that
\begin{equation}
\frac{1}{|\bx-\by|}=\sum_{n=0}^{\infty}  B_n(\bc,\bx,\by)
\label{eqn:Bn_expansion}
\end{equation}
Now, multiply the recursion relation (\ref{eqn:rec_rel}) by
$(\bx-\bc)^{\bk}$ and sum over indices $\bk$ with $\normk=n$. 
This yields
\begin{align}
& nR^2\sum_{\normk=n} b_{\bk}(\bx-\bc)^{\bk}
-(2 n-1) \sum_{\normk=n} \sum_{i=1}^3
(y_i-(x_c)_i)b_{\bk-\ei}(\bx-\bc)^{\bk} \\
& \hspace{6.0cm} +(n -1) \sum_{\normk=n} \sum_{i=1}^3 b_{\bk-2\ei} (\bx-\bc)^{\bk} =0, 
\end{align}
which can be written as (omitting the argument of $B$), 
\[
nR^2 B_n -(2n-1)B_{n-1} \sum_{i=1}^3
((y_i-(x_c)_i)(\bx-\bc)^{\ei}+(n-1) B_{n-2} \sum_{i=1}^3 (\bx-\bc)^{2\ei}=0,
\]
and so
\[
nR^2 B_n -(2n-1)B_{n-1} (\by-\bc) \cdot (\bx-\bc)+(n-1) B_{n-2}
|\bx-\bc |^2=0.
\]
Introduce $r=| \bx -\bc|$, $\alpha=(\by-\bc)\cdot (\bx-\bc)$. 
With 
\begin{equation}
B_n(\bx,\bc,\by)=\frac{1}{R} \left(\frac{r}{R}\right)^n
P_n \left( \frac{\alpha}{rR} \right), \quad P_0=1, 
\label{eqn:defBn_w_Pn}
\end{equation}
the recursion for $P_n$ becomes, 
\[
nP_n(z)-(2n-1)zP_{n-1}(z)+(n-1)P_{n-2}(z)=0, \quad n=1,2,\ldots
\]
with $P_0=1$ and $P_{-1}=0$. 
This is the recursion for the Legendre polynomials, and hence $P_n(z)$ is the
Legendre polynomial of degree $n$.

The spherical harmonics expansion about a center $\bc$ is given in 
(\ref{eqn:sph_expansion}). This expansion is 
obtained from the expansion in Legendre polynomials (\ref{eqn:Pn_expand})
using the Legendre polynomial addition theorem. 
In the expansion with Legendre polynomials,  $\theta$ is the angle
between $\bx-\bc$ and $\by-\bc$.
In (\ref{eqn:defBn_w_Pn}), we have 
\[
\frac{\alpha}{rR}=\frac{(\by-\bc) }{|\by-\bc|} \cdot \frac{(\bx-\bc)}{| \bx -\bc|  }
\]
and hence $\alpha/(rR)=\cos(\theta)$ with $\theta$ as above, and we
have seen how to show the equivalence of the expansions
(\ref{Taylors_approx}) and (\ref{eqn:Pn_expand}).

From the above analysis we can conclude that the error incurred by
truncating the Taylor expansion (\ref{Taylors_approx}) after including
all spherical shells such that $\|\bk\| \le p$ is the same as the
error obtained when truncating the spherical harmonics expansion
(\ref{eqn:sph_expansion}) at $n=p$, i.e.\ once all spherical harmonics
up to degree $p$ have been included.

\section*{Appendix B: Truncation error estimates} \label{sec:AppendixB}

We derive an expression for the truncation error in the case of the the
single layer potential, with the result for the double layer potential in 
Theorem \ref{DL_error_nonplanar}
following similarly. Consider first the simplified case in which the
local correction to the single layer potential, which we denote by
$\slayer_L \sigma (\bx)$, involves integration over a planar
disk-shaped surface $\setofpanels$ of radius $\bar{R}<<1$, i.e.,
\be \label{single_layer_local} \slayer_L \sigma (\bx) = \int_\setofpanels
\sigma (\by) G (\bx,\by) ~dS_\by.  \ee Let the origin $O$ of a Cartesian
coordinate system $(x_1,x_2,x_3)$ be located at the center of
$\setofpanels$, with the $x_3$ axis normal to the disk (see Figure \ref{fig_trunc}).  We will compute
$\slayer_L \sigma (\bx)$ by expanding the Green's function
$G(\bx,\by)=1/(4 \pi |\bx-\by|)$ in a Taylor's series with respect to
$\bx$ about the point $\bc=(0,0,c)$, which lies either above or below
the center of the disk. We further assume that 
\be 
\bar{R}^2<<|c|<<\bar{R}<<1.  
\ee 
Since $\setofpanels$ is planar we set
$y_3=0$ and in an abuse of notation denote $\by=(y_1,y_2)$ and
$r_y = (y_1^2 + y_2^2)^{1/2}$.  For convenience set $x_1=x_2=0$, and
consider the target point to lie on the $x_3-$axis inside the radius
of convergence for the Taylor's series.

We  define $t=(c-x_3)/(c^2+r_y^2)^{1/2}$
and $z=c/(c^2+r_y^2)^{1/2}$ and write the Green's function as
\[
G(x_3,r_y)= \frac{1}{ 4  \pi}  \frac{1}{(c^2+r_y^2)^{1/2} (1-2zt+t^2)^{1/2}},
\]
where we suppress the explicit  dependence of $t$ and $z$ on $x_3, r_y$. The Taylor's series expansion   of the Green's function is provided by the generating function \cite{ArfkenWeber}
\be \label{generating_function}
\frac{1}{(1-2zt+t^2)^{1/2}}=  \sum_{n=0}^\infty P_n(z) t^n,
\ee
so that the error in truncating the Taylor's series after $p$ terms is
\be \label{error_truncation}
E_T=\frac{1}{ 4 \pi} \left| \int_\setofpanels \frac{\sigma(\by)}{(c^2+r_y^2)^{1/2}}  \sum_{n=p+1}^\infty P_n \left( \frac{c}{(c^2+r_y^2)^{1/2}} \right)  \left(  \frac{c-x_3}{(c^2+r_y^2)^{1/2}} \right)^n  ~dS_\by \right|.
\ee

We now assume the density $\sigma(\by)$ is a smooth function of $\by$, and  using standard multi-index notation,  expand it in a Taylor's series 
about $\by={\bf 0}$ for $|\by|<\bar{R}$:
\[
\sigma(\by)=\sum_{\| \bk \| \geq 0 } \frac{\by^\bk}{\bk !} \partial_\by^\bk \sigma(\by={\bf 0}),
\]
where $\bk=(k_1,k_2)$ is an integer multi-index with all $k_i \geq 0$, and $\|\bk \|=k_1+k_2$.
Substituting this into (\ref{error_truncation}) and writing the integral in polar coordinates yields
\begin{multline} \label{Error_p}
E_T= \frac{1}{4 \pi} \left| \sum_{\| \bk \| \geq 0} \frac{ \partial_\by^\bk \sigma(0)}{ \bk!}  F(\bk) \right. \\
 \left. \times \int_0^{\bar{R}} \frac{r_y^{\| \bk \| +1}}{(c^2+r_y^2)^{1/2}}   \sum_{n=p+1}^\infty P_n \left( \frac{c}{(c^2+r_y^2)^{1/2} }\right) \left( \frac{c-x_3}{(c^2+r_y^2)^{1/2} }\right)^n dr_y \right|,
\end{multline}
where 
\[
F(\bk) =  \int_0^{2 \pi} \cos^{k_1} \theta \sin^{k_2} \theta ~d \theta
\]
is  the angle integral.  This integral  is zero unless $k_1$ and $k_2$ are both even, so we subsequently take them to be even.  We now substitute the 
explicit representation of the Legendre function
\[
P_n(z)=\frac{1}{2^n} \sum_{j=0}^{\floor{\frac{n}{2}}} (-1)^j \binom{n}{j} \binom{2n-2j}{n} z^{n-2j},
\]
where $\floor{\cdot}$ is the floor function, into (\ref{Error_p}).  After some rearrangement, we can write the error (\ref{Error_p}) as
\begin{multline} \label{Error_p1}
E_T= \left|
 \sum_{\| \bk\| \geq 0}  a_\bk\sum_{n=p+1}^\infty \frac{(c-x_3)^n}{2^n}\sum_{j=0}^{\floor{\frac{n}{2}}} (-1)^j \binom{n}{j} \binom{2n-2j}{n} c^{n-2j}  \right. \\ \left.  \times \int_0^{\bar{R}} r_y^{\| \bk \| +1}
\left( \frac{1}{(c^2+r_y^2)^{1/2}} \right)^{2n-2j+1} \ dr_y  \right|,
\end{multline}
where $a_\bk= \frac{1}{4 \pi} \frac{ \partial_\by^\bk \sigma({\bf 0})}{ \bk!}  F(\bk)$. The final step of the truncation error analysis for a planar surface is to compute the $r_y$ integral  in the above equation.

To compute this integral, we carry out  $s=\frac{\|\bk\|}{2}+1$ integrations-by-parts, which yields
\begin{multline} \label{ry_int_by_parts}
\int_0^{\bar{R}} r_y^{\| \bk \| +1}
\left( \frac{1}{(c^2+r_y^2)^{1/2}} \right)^{2n-2j+1} \ dr_y  =\\
 -\sum_{m=1}^s  \frac{\Pi_{i=1}^{m-1}\left\{ \|\bk\|-(2i-2) \right\}}{\Pi_{i=1}^m \left\{ 2n-2j-(2i-1) \right\}} \frac{{\bar{R}}^{\|\bk\|-(2m-2)}}
{(c^2+{\bar{R}}^2)^{[2n-2j-(2m-1)]/2}}\\
+ \frac{\Pi_{i=1}^{s-1}\left\{ \|\bk\|-(2i-2) \right\}}{\Pi_{i=1}^s \left\{ 2n-2j-(2i-1) \right\}} c^{-2n+2j+2s-1},
\end{multline}
where we use the notation $\prod_{i=1}^0 f(i) =1$.
An explanation of  the terms in this equation is as follows. The first $s-1$ integrations
each give a boundary contribution at $r_y=\bar{R}$, which gives the first $s-1$ terms in the sum. There is zero boundary  contribution at $r_y=0$, due to the power of $r_y$ in the integrand. The final integration gives both a boundary contribution at $r_y=\bar{R}$, which is the $m=s$ term in the sum, and a boundary contribution at $r_y=0$, which is the final term.

Together, equations (\ref{Error_p1}) and (\ref{ry_int_by_parts}) provide an exact representation of the truncation error for a  planar surface.
However, substituting  (\ref{ry_int_by_parts}) into (\ref{Error_p1}) and combining  like terms, we see that the $nth$ term in the sum has a factor of $c^{-n+2s-1}$, which is large for small $s$ and $|c|<<1$,  coming from the expression in the third line of  (\ref{ry_int_by_parts}).
The crux of the analysis is to overcome this large factor. The subsequent estimates rely on the following two lemmas on binomial coefficients, which are proven in Appendix C. 

\begin{lemma}  \label{lemma_bin1}
For any integers $1 \leq m \leq \floor{\frac{n}{2}}$ and $2 \leq n<\infty$, the binomial coefficients satisfy the identity,
\be
\sum_{j=0}^{\floor{\frac{n}{2}}} \left\{ \frac{(-1)^j \binom{n}{j} \binom{2n-2j}{n}}{\prod_{i=1}^{m} [2n-2j-(2i-1)]} \right\}=0
\ee
\end{lemma}

\begin{lemma}  \label{lemma_bin2}
Let  $d_k^{(n)}=(-1)^k \binom{n}{k} \binom{2n-2k}{n}/2^n$ be the coefficient of the monomial $z^{n-2k}$  in  the Legendre polynomial $P_n(z)$. Then $d_k^{(n)}$ satisfies the bound
$|d_k^{(n)}|  \leq (1+\sqrt{2})^n$.

\end{lemma}

Continuing with the calculation, when $1 \leq s \leq {\floor{\frac{n}{2}}}$, then the contribution to the error (\ref{Error_p1}) from the large factor (third line) in (\ref{ry_int_by_parts}) sums to zero by Lemma \ref{lemma_bin1}. Since the remaining terms in (\ref{Error_p1}), (\ref{ry_int_by_parts})  are rather complicated, we simplify the result by presenting the leading order contribution to the error in the small parameters $c$, $\bar{R}$, and $c/\sqrt{c^2+\bar{R}^2}$.
The leading order contribution to the error is given by the $\|\bk\|=0$, $n=p+1$, and $j= {\floor{\frac{n}{2}}}$  term in the sum and is 
$O  \left( \frac{\left[(1+\sqrt{2})|c-x_3| \right]^{p+1}}{ \left( \sqrt{c^2+{\bar{R}}^2} \right)^p}  \right)$ for $p$ odd (when $p$ is even there is an additional factor of $c/\sqrt{c^2+{\bar{R}}^2}$ in the leading order error). In making this estimate, we have used Lemma (\ref{lemma_bin2}). When ${\floor{\frac{n}{2}}} < s$ we can no longer use Lemma (\ref{lemma_bin1}), but in this case the contribution to the sum (\ref{Error_p1})  from the third line of (\ref{ry_int_by_parts}) is at most $O  \left( \left[(1+\sqrt{2})|c-x_3| \right]^{p+1}   \right)$, which is smaller than the leading order contribution coming from the other terms. These remarks are summarized in the following:

\begin{lemma} \label{lemma_planar}
Let $\setofpanels$ be a planar disk of radius $0<\bar{R}<<1$, and let the origin of a Cartesian coordinate system $(x_1,x_2,x_3)$  be at the center of the disk, with the $x_3$ axis normal to the disk.  Let  $E_T$ (see  (\ref{Error_p1}))  be the truncation error  of the local single layer potential $\slayer_L \sigma (\bx)$ 
evaluated  by Taylor's expansion of order $p$ about the point $\mathbf{c}=(0,0,c)$, where ${\bar{R}}^2<<|c|<<{\bar{R}}$.  Assume the target point $(0,0,x_3)$  lies inside the radius of convergence of the Taylor's series, i.e., $r \leq |c|$ where $r=|x_3-c|$.  Then $E_T$ satisfies the bound
\begin{multline} \label{E_p_lemma}
E_T \leq   C  \ \alpha_p \left| \sigma({\bf 0}) \right| \frac{\left[(1+\sqrt{2}) \ r \right]^{p+1}}{ \left( \sqrt{c^2+{\bar{R}}^2} \right)^p}   \left( 1 + O \left( \frac{c^2}{c^2+{\bar{R}}^2} \right) \right)\\
+ O \left( \alpha_p \  {\bar{R}}^2 \frac{ r^{p+1}}{ \left( \sqrt{c^2+{\bar{R}}^2} \right) ^p} \sum_{\|\bk\| = 2} \left| \partial_\by^{\bk} \sigma({\bf 0}) \right|  \right) 
\end{multline}
where  $\alpha_p=1$ for $p$ odd and $c/\sqrt{c^2+{\bar{R}}^2}$ for $p$ even, and $C$ is a constant. 

\end{lemma}

\noindent
{\bf Proof:}   The leading order term follows from the comments preceding the lemma (the prefactor $\sigma(0)$ comes from $a_0$ in (\ref{Error_p1})).  
%The $O(c^2/(c^2+{\bar{R}}^2)$ correction and the second line of (\ref{E_p_lemma})  are the
The next order corrections, correspond, respectively,  to the $\|\bk\|=0$, $n=p+1$, $j=\floor{n/2}-1$ term in  (\ref{Error_p1}), (\ref{ry_int_by_parts})  and the  $\|\bk\|=2$, $n=p+1$, $j=\floor{n/2}$ term there. \\

We now generalize to the case in which the surface patch $\setofpanels$ is  nonplanar. We assume that $\setofpanels$  is smooth, and that $\bar{\bx}$ is a point on $\setofpanels$ such that if  $B_c({\bc})$  is a ball of radius $c$ about the expansion center $\bc$ then $\overline{B}_c(\bc) \cap \setofpanels = \{ \bar{\bx} \}$.  
%The boundary of $\Gamma$ is denoted by $C_\Gamma$ and is assumed to be a circle of radius ${\bar{R}}$. 
The surface patch $\setofpanels$ is assumed to be such that its projection $R_\Gamma$  onto the tangent plane at $\bar{\bx}$ is a disk of radius $\bar{R}$.
We place the  origin  $O$ of a Cartesian coordinate system at $\bar{\bx}$,  and assume the $x_3$ axis is directed along the line $\bar{\bx}-\bc$, i.e., normal to $\setofpanels$ at $\bar{\bx}$.  The situation is illustrated in Figure \ref{fig_trunc}. 
%We further assume that $C_\Gamma$ is parallel to 
%the tangent plane at $\bx$.  
%The   surface $\Gamma$ is parameterized by $Y(y_1,y_2)$, where $\by=(y_1,y_2)$ varies over the planar region $R_\Gamma$ bounded by $C_\Gamma$ (this supposes that the surface is a graph in these coordinates). The situation is illustrated in Figure \ref{fig_nonplanar}.

% In writing the surface integral in (\ref{single_layer_local}),  it is conveneint to incorporate the surface element into a modified density function $\Sigma(\by) = \sigma(\by) 
%(1+Y^2_{y_1}(\by) + Y^2_{y_2}(\by))^{1/2}$.   

The local  single layer potential is written as
\[
\slayer_L \tilde{\sigma} (\bx)= \int_{R_\Gamma} \tilde{\sigma}(y_1,y_2)  \ G(\bx, y_1, y_2, Y(y_1,y_2) ) \ dy_1  \ dy_2.
\]
where we have parameterized the surface $\setofpanels$ by $Y(y_1,y_2)$ in which  $(y_1,y_2)$ varies over the planar region $R_\Gamma$ 
%bounded by $C_\Gamma$ 
(this supposes that the surface is a graph in these coordinates). Additionally, we have introduced a modified 
density function $\tilde{\sigma}(y_1,y_2) = \sigma(y_1,y_2) 
(1+Y^2_{y_1}(y_1,y_2) + Y^2_{y_2}(y_1,y_2))^{1/2}$ which incorporates the surface element.
Following the analysis for the planar case,  we compute $\slayer_L \tilde{\sigma} (\bx)$   by expanding the Green's function 
 in a Taylor's series with respect to $\bx$  about the point $\bc=(0,0,c)$. We  assume that the parameter scaling
(\ref{scaling_b_R}) holds,
and    as before,  set $x_1=x_2=0$ and consider the target point to lie on the $x_3-$axis inside the radius of convergence for the Taylor's series. If $G^{(p)}(x_3,\by)$ denotes the $p$-term Taylor's expansion of $G(x_3, \by)$ in powers of $c-x_3$ (where now $\by=(y_1,y_2,Y(y_1,y_2)$), then the truncation error for a nonplanar surface  is 
\begin{eqnarray} 
E_T  &=&    \left|  \int_{R_\Gamma} \tilde{\sigma}(y_1,y_2) \left(  G(x_3,\by)-G^{(p)}(x_3,\by) \right) \ dy_1 \ dy_2 \right|  \label{nonplanar_truncation}\\
&=& \left| \sum_{n=p+1}^\infty  c_n (c-x_3)^n. \label{c_n} \right|
\end{eqnarray}
where the $c_n$ are real.
In the following analysis, we compute the leading order part of  $c_n$ in the small parameters $c, ~{\bar{R}}$, and $c/\sqrt{c^2+{\bar{R}}^2}$.

We rotate the $y_1$ and $y_2$ axes so that they are aligned with the directions of principal curvature of $\setofpanels$ at $\bar{\bx}$. Under our assumptions, the surface $Y(y_1,y_2)$ has a Taylor's expansion about $(0,0)$ for $r_y<{\bar{R}}$ of the form
\be \label{Y_Taylors}
Y(y_1,y_2) = b_1 y_1^2 + b_2 y_2^2 + \sum _{\|\bk\| \geq 3} b_\bk y_1^{k_1} y_2^{k_2}.
\ee
where $b_1,b_2,b_\bk$ are the Taylors coefficients and we recall $r_y=(y_1^2+y_2^2)^{1/2}$. Note that $b_1$ and $b_2$ are also the principal curvatures of the surface $\setofpanels$.  The modified density function is also expanded about $(0,0)$,  and subsequently we only consider the leading order term $\tilde{\sigma}(0,0) \equiv \tilde{\sigma}_0$, since it can be shown, as for the planar case,  that higher order terms in the density give a  higher order contribution to  the truncation error.

 Next,  substitute $c+(x_3-c) - Y$ for $x_3-Y$ in the Green's function,  factor out $D^{1/2}(r_y,x_3)$ where $D(x_3, r_y)=c^2+r_y^2-2c(c-x_3)+(c-x_3)^2$ and apply the binomial expansion to obtain
\be \label{G_expansion_nonplanar}
G(x_3,\by)= \frac{1}{4 \pi}   \frac{1}{D^{1/2}} \sum_{l=0}^{\infty} \binom{-1/2}{l} \left[ \frac{2(c-x_3)Y-2cY+Y^2}{D} \right]^l
\ee
where we define $\binom{-1/2}{0} = 1$. This expansion is justified by the assumption (\ref{scaling_b_R}).  Substitute the expansion (\ref{Y_Taylors}) for $Y$ into (\ref{G_expansion_nonplanar}) and represent the surface integral in polar coordinates as 
\be \label{single_layer_polar}
\slayer_L \tilde{\sigma}(x_3) \approx \tilde{\sigma}_0 \int_0^{2 \pi} \int_0^{\bar{R}}  G(x_3, r_y, \theta)  \ r_y  \ dr_y  \  d \theta.
\ee
where $G(x_3,r_y,\theta)$ is given by (\ref{G_expansion_nonplanar}) and $Y(r_y,\theta)$  by (\ref{Y_Taylors})  with $y_1=r_y \cos \theta$ and $y_2=r_y \sin \theta$, and we have approximated $\sigma \approx \sigma_0$.

The truncation error (\ref{nonplanar_truncation}) is calculated by expanding (\ref{G_expansion_nonplanar})  in powers 
of $c-x_3$ using (\ref{generating_function}). 
It can be shown   that the leading order contribution to the Taylor's coefficient $c_n$ in (\ref{c_n})  comes from the $l=0$ term in the  sum  (\ref{G_expansion_nonplanar}), and  the next order correction from the $l=1$ term, while the  contributions from $l=2, 3, \ldots$ are successively higher order.  
% and evaluating the $\theta$ integral,  we obtain a series containing $r_y$ integrals of the form
%\[
%\int_0^R \frac{r_y^{2k}}{D^{-1/2-2j}}
%\]
%for positive integers $k$ and $j$ with $k \geq j$.  
We will  calculate   the leading order $l=0$ and $1$ contributions to the truncation error.  In doing so, we can neglect the  $Y^2$  term in (\ref{G_expansion_nonplanar}) and approximate  $Y \approx b_1 y_1^2 + b_2 y_2^2$, since the neglected terms give higher order contributions to $c_n$.  
After we substitute the $l=0$ and $1$ terms from  (\ref{G_expansion_nonplanar}) into (\ref{single_layer_polar}), make the above approximations, and compute the angle integral, we obtain
for the  truncation error the expression
\be \label{l=0_and_1}
%\slayer_L \tilde{\sigma} (x_3) \approx \frac{\tilde{\sigma}_0}{2}  \int_0^{\bar{R}} \left[ \frac{1}{D^{1/2}(r_y,x_3)} + \frac{\left( c+(x_3-c) \right) \kappa r_y^2}{D^{3/2}(r_y,x_3)}  \right] r_y
%\ d r_y,
E_T \approx \frac{\sigma_0}{2}  \left| \int_0^{\bar{R} } \left[ \tilde{G} (x_3, r_y) - \tilde{G}^{(p)}(x_3, r_y) \right] r_y \ dr_y \right|,
\ee
where
\be \label{Gtil_def}
\tilde{G}(x_3, r_y)=  \frac{1}{D^{1/2}(x_3, r_y)} + \frac{\left( c+(x_3-c) \right) {\cal H} r_y^2}{D^{3/2}(x_3, r_y)},
\ee
and  ${\cal H}=(b_1+b_2)/2$ is the mean curvature of $\setofpanels$ at $\bar{\bx}$. 
Here``$\approx$'' means leading order in the sense of the largest contribution to the to the Taylor's coefficients $c_n$  in our small parameters.

We still need to expand (\ref{l=0_and_1}), (\ref{Gtil_def}) in powers of $c-x_3$.
The  first  term in (\ref{Gtil_def})  is identical to one that arises in the single layer potential  for a planar surface with constant density, and is expanded as above. 
%The analysis presented there gives the truncation error corresponding to this term. 
The second  term in (\ref{Gtil_def}) is  integrated by parts. This gives a boundary contribution at $r_y={\bar{R}}$ and an integral term:
\[
\tilde{\sigma}_0 {\cal H} \left( c+ (x_3-c) \right)  \cdot \left( -\frac{{\bar{R}}^2}{2 D^{1/2}(x_3, {\bar{R}})} + \int_0^{\bar{R}} \frac{r_y}{D^{1/2}(x_3, r_y)} \ d r_y \right).
\]  The first term above is Taylor expanded in   $c-x_3$ using (\ref{generating_function}), after factoring out $(c^2+{\bar{R}}^2)^{1/2}$.  The second (integral) term is again identical to one which arises in the planar analysis, and is treated the same as there.  We leave the details to the reader, and summarize the result:
\begin{theorem} \label{Ep_nonplanar}
Let  $\setofpanels$ be a smooth surface and $\bar{\bx}$ a point on the surface such that the projection $R_\Gamma$ of $\setofpanels$ onto the tangent plane at $\bar{\bx}$ is a disk of radius ${\bar{R}}$.  Let $\bc$ be an expansion center with $\overline{B}_c(\bc) \cap \setofpanels = \left\{ \bar{\bx} \right\}$. Let the origin of a Cartesian coordinate system $(x_1,x_2,x_3)$ be at $\bar{\bx}$, and let the $x_3$ axis be directed along the line $\bar{\bx}-\bc$.   Let  $E_T$ be the truncation error  defined in Lemma \ref{lemma_planar} for center point $\bc=(0,0,c)$, and assume ${\bar{R}}^2<<|c|<<{\bar{R}}<<1$. Then for any target point $(0,0,x_3)$ inside the radius of convergence of the Taylor's series, $E_T$ satisfies the bound
\begin{multline}
E_T \leq   C \ \alpha_{p} \left| \sigma({\bf 0}) \right| \frac{\left[(1+\sqrt{2}) \ r \right]^{p+1}}{ \left( \sqrt{c^2+{\bar{R}}^2} \right)^p}   \left( 1 + O \left( \frac{c^2}{c^2+{\bar{R}}^2} \right) \right)\\
+ O \left( \alpha_{p+1}  \ |\sigma({\bf 0})| \  {\cal H} {\bar{R}}^2 \frac{ r^{p+1}}{ (\sqrt{c^2+{\bar{R}}^2})^{p+1}} \right)  \label{SL_trunc_nonplanar}
\end{multline}
where $\alpha_p$ and $r$  are as in Lemma \ref{lemma_planar},  ${\cal H}$ is the mean curvature of $\setofpanels$  at $\bar{\bx}$, and $C$ is a constant.

\end{theorem}

\noindent

Note that the leading order truncation error for the single layer potential is a factor of ${\bar{R}}$ smaller than the truncation error of the double layer potential in Theorem \ref{DL_error_nonplanar}.  
The truncation error estimates have been derived assuming the scaling (\ref{scaling_b_R}), but  they are expected to hold for $|c|/{\bar{R}}$ sufficiently small that    the expansions above are  valid,
%It is not necessary for  
% $|c|/{\bar{R}}$  to tend to zero as ${\bar{R}} \rightarrow 0$, and 
%In our numerical tests, the parameter scaling is such that 
e.g. if $|c|/\bar{R}$ tends to a small constant as the panel size $h \rightarrow 0$.

\section*{Appendix C: Proofs of Lemmas \ref{lemma_bin1} and \ref{lemma_bin2}}

The proof of Lemma \ref{lemma_bin1} makes use of the following identity 
from Corollary 2 in  \cite{ruiz1996}:
\be \label{bin_identity}
\sum_{j=0}^n (-1)^j \binom{n}{j} Q(j) =0,
\ee
where $Q(z)$ is any polynomial of degree less than $n$, and $n>0$.
We first prove the lemma  for  the special case $m=\floor{n/2}$. Write out the binomial coefficients, cancel  common factors,  and factor out a power of $2$ to obtain
\begin{eqnarray}
\sum_{j=0}^{\floor{\frac{n}{2}}} \left\{ \frac{(-1)^j \binom{n}{j} \binom{2n-2j}{n}}{\prod_{i=1}^{\floor{n/2}} [2n-2j-(2i-1)]} \right\} \nonumber
&=&
2^{\ceil{n/2}} \sum_{j=0}^{\floor{n/2}} \frac{(-1)^j}{j!  (\floor{n/2}-j)!}, \nonumber
\\
&=& \frac{2^{\ceil{n/2}}}{\floor{n/2}!} \sum_{j=0}^{\floor{n/2}} \binom{\floor{n/2}}{j}(-1)^j,  \nonumber\\
&=& 0, \nonumber
\end{eqnarray}
where $\ceil{\cdot}$ is the ceiling function, and in the last equality we have used the identity (\ref{bin_identity}) with $Q(z)=1$.
To prove the lemma for $1 \leq m < \floor{n/2}$, we write
\begin{eqnarray}
\sum_{j=0}^{\floor{\frac{n}{2}}} \left\{ \frac{(-1)^j \binom{n}{j} \binom{2n-2j}{n}}{\prod_{i=1}^{m} [2n-2j-(2i-1)]} \right\}
&=&
\sum_{j=0}^{\floor{\frac{n}{2}}} \left\{ \frac{(-1)^j \binom{n}{j} \binom{2n-2j}{n}}{\prod_{i=1}^{\floor{n/2}} [2n-2j-(2i-1)]} \right\} \\
&~& ~~~~~\times
\prod_{i=m+1}^{\floor{n/2}} (2n-2j-(2i-1)), \nonumber
 \\
&=&
\frac{2^{\ceil{n/2}}}{\floor{n/2}!} \sum_{j=0}^{\floor{n/2}} \binom{\floor{n/2}}{j}(-1)^j Q_{\floor{n/2}-m}(j), \nonumber \\
&=&
0 \nonumber
\end{eqnarray}
where $Q_{\floor{n/2}-m}(j)$ is a polynomial of degree $\floor{n/2}-m$ in $j$.   The last identity follows from (\ref{bin_identity}).

We next consider the proof of Lemma \ref{lemma_bin2}.  The $n$th degree Legendre polynomial is 
\[
P_n(z)=\frac{1}{2^n} \sum_{k=0}^{\floor{n/2}} (-1)^k \binom{n}{k}
\binom{2n-2k}{n} z^{n-2k}.
\]
Let $a_j^{(n)}$ be the coefficient of the monomial $z^j$ in $P_n (z).$ 
%for $0 \leq j \leq n$.  
We seek a uniform in $j$ bound on the magnitude of the coefficients $a_j^{(n)}$, that is,  we find a constant $b \geq 1$ such that $| a_j^{(n)} | \leq b^n$ for any $n$ and $0 \leq j  \leq n$.  To do this, we make use of the recursion formula
\[
(n+1) P_{n+1}(z) = (2n+1) z P_n(z) - n P_{n-1}(z).
\]
for $n \geq 1$ with $P_0(z)=1$, from which it is easy to see that 
\[
\left| a_{j+1}^{(n+1)} \right| \leq 2 \left| a_j^{(n)} \right| + \left| a_{j+1}^{(n-1)} \right|, ~~\mbox{for} ~0 \leq j \leq n,
\]
and $| a_{0}^{(n+1)} | \leq | a_{0}^{(n-1)}  |.$
To find a suitable (smallest) $b$, set $b^{n+1}=2 b^n + b^{n-1}$, which gives
$b=1 + \sqrt{2}$.

%\end{appendices}

\bibliographystyle{plain} 
\bibliography{qbx_refFeb26}

\end{document}